\documentclass[12pt]{amsart}
\usepackage{float}

\usepackage{amsmath, amsfonts, amsthm, mathrsfs, amssymb}
\usepackage{bm}        
\usepackage{mathtools}

\usepackage{graphicx}  
\usepackage{tikz}
\usetikzlibrary{cd, trees}
\usepackage{tikz-qtree}
\usepackage{forest}

\usepackage{booktabs}   
\usepackage{multirow}   
\usepackage{algorithm}
\usepackage{algorithmic}

\usepackage{xcolor}

\usepackage{comments}

\usepackage{fullpage}   
\usepackage{balance}    
\usepackage{abstract}   
\usepackage{color}     
\usepackage{url}        
\usepackage{csquotes}

\usepackage[hidelinks]{hyperref} 
\usepackage[backend=biber]{biblatex}

\usepackage{hyperref} 
\hypersetup{
    colorlinks=true,
    linkcolor=blue,
    urlcolor=red,
    pdftitle={},
    }

\newtheorem{definition}{Definition}[section]
\newtheorem{theorem}{Theorem}[section]
\newtheorem{remark}{Remark}[section]
\newtheorem{example}{Example}[section]

\newtheorem{lemma}{Lemma}[section]
\newtheorem{proposition}{Proposition}[section]
\newtheorem{corollary}{Corollary}[section]
\title{On the geometry of punctual Hilbert schemes on singular curves and their motivic zeta functions }

\author{Mounir Hajli}
\address{School of Science, Westlake University\\
Hangzhou, China 310000}
\email{hajli@westlake.edu.cn}

\author{Hussein Mourtada}
\address{Universit\'e Paris Cit\'e, Sorbonne Universit\'e \\CNRS, IMJ-PRG, F-75013 Paris}
\email{hussein.mourtada@imj-prg.fr}

\author{Wenhao Zhu}
\address{Universit\'e Paris Cit\'e,  IMJ-PRG, F-75013 Paris}
\email{wenhao.zhu.alg@gmail.com}


\begin{document}
\maketitle

\begin{abstract}
Inspired by the work of Soma and Watari, we introduce a tree structure on the subsemimodules of the semigroup $\Gamma$ associated with an irreducible plane curve singularity $(C,O)$. Building on results of Oblomkov, Rasmussen, and Shende, we show that for certain classes of singularities this tree captures essential geometric information about the punctual Hilbert schemes of $(C,O)$. As an application, we derive an explicit formula for the motivic Hilbert zeta function of simple curve singularities.
\vskip 0.1cm
A point of the Hilbert scheme corresponds to an ideal of the local ring $\mathcal{O}_{C,O}$ of the singularity. We investigate the stratification of these Hilbert schemes arising from conditions on the minimal number of generators of the defining ideals, and we describe geometric features of the resulting strata, including their dimension and closure relations. We then study a bivariate motivic zeta function that naturally extends a zeta function with integer coefficients and that, in the case of a plane curve singularity, has recently been related to the topology of the link of the singularity.
\end{abstract}

\tableofcontents

\section{Introduction}
The main objects of study in this article are the punctual Hilbert schemes of irreducible curve singularities defined over an algebraically closed field $k$ of characteristic~$0$. 
Given a curve singularity $(C,O)$ and an integer $\ell \in \mathbb{N}$, the $\ell$-th punctual Hilbert scheme of $(C,O)$, denoted $C^{[\ell]}$, is the moduli space parametrizing $0$-dimensional subschemes of $(C,O)$ of length $\ell$. This is a special case of Grothendieck's Hilbert scheme, which more generally parametrizes subschemes of projective space with fixed Hilbert polynomial.\vskip 0.1cm
From an algebraic viewpoint, let $A = \mathcal{O}_{C,O}$ denote the local ring of $C$ at $O$. Then the $k$-points of $C^{[\ell]}$ are in bijection with ideals of colength $\ell$ in $A$:
\[
C^{[\ell]} := \left\{ I \subset A \mid I \text{ is an ideal and } \dim_k \frac{A}{I} = \ell \right\}.
\]\vskip 0.1cm
Let $\overline{A}$ denote the integral closure of $A$ in its total ring of fractions. Since $C$ is an irreducible curve singularity, $\overline{A}$ is a discrete valuation ring. The associated valuation is denoted by $v \colon \overline{A} \setminus \{0\} \to \mathbb{N}$, and we define the delta invariant of the singularity by
\[
\delta := \dim_k \frac{\overline{A}}{A}.
\]

Pfister and Steenbrink \cite{pfister1992reduced} proved that  
  the punctual Hilbert schemes $C^{[l]}$ can be embedded in a closed linear subvariety of the Grassmannian $Gr(\delta,\overline{A}/I(2\delta)),$ where $I(2\delta)$ is the ideal consisting of elements $h \in \overline{A}$ (or equivalently of $A$) whose valuation $v(h)$ is larger than or equal to $2\delta.$ Moreover, the authors "provided" the defining equations of this embedding. In general, even when the defining equations of punctual Hilbert schemes are known, it remains difficult to understand their geometry or to compute their invariants.\\

Progress in understanding the geometry of punctual Hilbert schemes has been relatively slow, due in part to the complexity of their structure even in simple singular cases. Soma and Watari \cite{soma2015punctual, soma2014punctual} studied the geometry of punctual Hilbert schemes for plane curve singularities of types $A_{2d}$, $E_6$, and $E_8$. In the case of unibranched curve singularities, the punctual Hilbert scheme $C^{[\ell]}$ coincides with the compactified Jacobian $\overline{J C}$ for sufficiently large $\ell$ \cite{piontkowski2007topology}.

For curves defined by the equation $y^m - x^n = 0$ with $\gcd(m,n) = 1$, Lusztig and Smelt \cite{lusztig1991fixed} computed the homology of the associated Hilbert schemes. Piontkowski \cite{piontkowski2007topology} later extended these results to curves with a single Puiseux pair. Subsequently, Gorsky, Mazin, and Oblomkov \cite{gorsky2022generic} derived explicit formulas for the Poincaré polynomials of Hilbert schemes associated with singularities having Puiseux exponents $(nd, md, md+1)$.

A key insight in this line of research is that the intersection of the compactified Jacobian $\overline{J C}$ with certain Schubert cells in the affine Grassmannian has the structure of an affine space. This geometric fact admits rich combinatorial interpretations, particularly in relation to $(q,t)$-Catalan theory \cite{gorsky2016affine, gorsky2013compactified, gorsky2014compactified, hilburn2023bfn}.\\

For curve singularities with a single Puiseux pair, Oblomkov, Rasmussen, and Shende \cite{oblomkov2018hilbert} proved that the punctual Hilbert scheme $C^{[\ell]}$ admits a stratification by affine spaces, a result that plays a foundational role in the motivic study of these moduli spaces.\\

Recently, punctual Hilbert schemes have attracted considerable attention from two distinct but interconnected perspectives, both of which provide key motivations for this work.

On one hand, these schemes are the fundamental building blocks for the \emph{motivic Hilbert zeta function}, introduced by Bejleri, Ranganathan and Vakil ~\cite{bejleri2020motivic}. This zeta function is defined as the generating series of the classes of punctual Hilbert schemes in the Grothendieck ring of varieties:
\[
Z^{\mathrm{Hilb}}_{X}(q) = \sum_{\ell \geq 0} \big[\mathrm{Hilb}^{[\ell]}(X)\big]q^\ell \in 1+q K_0(\mathrm{Var}_k)[[q]].
\]

Compared to the classical motivic zeta function of Kapranov~\cite{kapranov2000elliptic}, the motivic Hilbert zeta function provides a finer invariant, capturing detailed information about the singularities of a variety $X$ defined over an algebraically closed field $k$.

On the other hand, a series of conjectures by Oblomkov, Rasmussen, and Shende~\cite{oblomkov2012hilbert, oblomkov2018hilbert} posits a deep connection between the geometry of punctual Hilbert schemes of plane curve singularities and the topology of associated algebraic links, when working over $\mathbb{C}$. More precisely, they conjecture that the (reduced) Khovanov--Rozansky homology of the link of the singularity can be recovered from the homology of the Hilbert schemes. Maulik~\cite{maulik2016stable} proved the first of these conjectures using an inductive argument based on blow-ups, analyzing the variation of both Hilbert schemes and the HOMFLY polynomial under such transformations.

Very recently, Rossinelli~\cite{rossinelli2024motivic} discovered a new connection between \emph{curvilinear} punctual Hilbert schemes and the Igusa motivic zeta function, further highlighting the role of Hilbert schemes as a unifying object in singularity theory and motivic integration.  Diaconescu, Porta, Sala and Vosoughinia \cite{diaconescu2023flops} made a progress in the case of space curve singularities. Fasola, Graffeo, Lewanski and Ricolfi \cite{fasola2025invariants} computed the motivic classes of punctual nested Hilbert schemes on a smooth pointed surface.  Graffeo,  Monavari, Moschetti, and  Ricolfi \cite{graffeo2024motive} proved a closed formula for the generating function of the motives $[\mathrm{Hilb}^d(\mathbb{A}^n)_0] \in K_0(\mathrm{Var}_{\mathbb{C}})$ of punctual Hilbert schemes, summing over $n$, for fixed $d > 0$.

\vskip 0.5cm

The main goal of this paper is to describe certain geometric aspects of the punctual Hilbert schemes associated with curve singularities. We focus on two classes: first, plane curve singularities defined by equations of the form $y^m - x^n = 0$ with $\gcd(m,n) = 1$; and second, (not necessarily plane) curve singularities whose semigroup $\Gamma := v(A\setminus  \{0\})$ is monomial.

The latter class consists of singularities for which, within their equisingularity class—i.e., the set of curve singularities sharing the same semigroup—there exists a unique representative up to analytic isomorphism (see Section \ref{Appendix} for futher details). This class was originally introduced by Pfister and Steenbrink~\cite{pfister1992reduced}, and plays a distinguished role in the classification of unibranch curve singularities due to its rigidity under deformation.\\

To achieve this, we make extensive use of $\Gamma$-subsemimodules, where $\Gamma$ is the semigroup of an irreducible curve singularity. A \emph{$\Gamma$-subsemimodule} is a subset $\Delta \subset \Gamma$ such that $\Gamma + \Delta \subseteq \Delta$. It follows directly from the properties of ideals and the valuation $v$ that, for any ideal $I \subset A$, the value set 
\[
v(I\setminus \{0\}) = \{ v(f) \mid f \in I \setminus \{0\}\}
\]
is a $\Gamma$-subsemimodule of $\Gamma$. This fundamental observation links the algebraic structure of ideals in $\mathcal{O}_{C,O}$ to the combinatorics of the semigroup $\Gamma$.

Inspired by recent work of Soma and Watari  \cite{soma2014punctual} on the geometry of Hilbert schemes, as well as that of Oblomkov, Rasmussen, and Shende \cite{oblomkov2018hilbert} on knot invariants and curve singularities, we associate to a curve singularity with semigroup $\Gamma$ a leveled graph $G_\Gamma$ defined as follows.

The vertices at level $\ell$ are the elements of the set
\[
\mathscr{D}_\ell := \left\{ \Delta \subset \Gamma \mid \Delta \text{ is a } \Gamma\text{-subsemimodule with } \#(\Gamma \setminus \Delta) = \ell \right\}.
\]
We draw a directed edge from $\Delta \in \mathscr{D}_\ell$ to $\Delta' \in \mathscr{D}_{\ell-1}$ if 
\[
\Delta' = m(\Delta) := \Delta \cup \{\gamma_\Delta\}, \quad \text{where} \quad \gamma_\Delta := \max(\Gamma \setminus \Delta).
\]

We study  the graph $G_\Gamma$ and prove that it has the structure of a tree (Theorem \ref{Tree}): the important thing about this theorem is that it explains how to obtain all $\Gamma-$submodules. Now, each subsemimodule $\Delta \in \mathscr{D}_\ell$ defines a constructible subset $C^{[\Delta]} \subseteq C^{[\ell]}$, consisting of all ideals $I \subset A$ in $C^{[\ell]}$ such that $v(I\setminus \{0\}) = \Delta$. One of the main result of the paper, is that, for certain curve singularities  an edge of the graph induces a peicewise trivial fibration (the proof uses the work in \cite{oblomkov2018hilbert} regarding the defining equations of $C^{[\Delta]}$); More precisely, we have the following theorem:

\begin{theorem}[Main Theorem]\label{thm:fibration}(see also Theorem \ref{piecewise fibration}) 
Let $C$ be either a plane curve singularity defined by $x^p - y^q = 0$ with $\gcd(p,q) = 1$, 
 or a curve singularity with a monomial semigroup $\Gamma$. Let $\Delta$ be a $\Gamma$-subsemimodule of $\Gamma$, and let $\gamma_\Delta = \max(\Gamma \setminus \Delta)$. Then there exists a canonical morphism
\[
\pi_{\Delta} \colon C^{[\Delta]} \longrightarrow C^{[m(\Delta)]},
\]
which is isomorphic to a trivial fibration over its image. The fiber of $\pi$ is isomorphic to an affine space $\mathbb{A}^{B(\Delta)}$, where
\[
B(\Delta) = \# \left\{ \gamma_i \mid \gamma_i < \gamma_\Delta \right\}
\]
i.e. is number of minimal generators below $\gamma_\Delta$,
and the $\gamma_i$ range over the minimal generators of $\Delta$ as a $\Gamma$-subsemimodule.
\end{theorem}

We explicitly determine the image of this morphism. In the case of plane curve singularities defined by $x^p - y^q = 0$ with $\gcd(p,q)=1$, or more generally for singularities with monomial semigroups, this allows us to recover a key result from \cite{oblomkov2018hilbert}: $C^{[\Delta]}$ is an affine space, and its dimension is completely determined by the combinatorial invariants of the $\Gamma$-subsemimodule $\Delta$.

We also analyze an example of a plane curve singularity with two Puiseux pairs (see Example \ref{fibration for 6 9 19}). In this case, the morphism \[\pi_{\Delta}: C^{[\Delta]} \longrightarrow C^{[m(\Delta)]}\] still induces a piecewise trivial fibration, but the dimension of the fiber is not solely governed by the number of minimal generators below $\gamma_\Delta$.  New phenomenas arise and must be further studied to fully understand the geometry of the strata and the structure of this morphism; we think that it is a fibration. \\

One important application of Theorem~\ref{thm:fibration} is the development of an algorithm to compute the motivic Hilbert zeta function for irreducible curve singularities, as introduced in \cite{bejleri2020motivic}:
\[
Z_{(C, O)}^{\mathrm{Hilb}}(q) := 1 + \sum_{\ell=1}^{\infty} \big[C^{[\ell]}\big] \, q^{\ell} \in K_0({\mathrm{Var}}_{{\mathbb{C}}})[[q]],
\]
where $K_0({\mathrm{Var}}_{{\mathbb{C}}})$ denotes the Grothendieck ring of varieties over ${\mathbb{C}}$. 
The recursive structure of the tree $G_\Gamma$, combined with the piecewise trivial fibration property established in Theorem~\ref{thm:fibration}, allows us to express each class $\big[C^{[\ell]}\big]$ in terms of classes of strata $C^{[\Delta]}$ and their fibrations, leading to an effective recursive computation.

\begin{theorem}\label{thm:zeta-formulas}(see Theorem \ref{thm:E6E8})
For the simple singularities of types $A_{2d}$, $E_6$, $E_8$, $W_8$, and $Z_{10}$, the motivic Hilbert zeta function is given by the following \textbf{rational} expressions in $K_0({\mathrm{Var}}_{{\mathbb{C}}})[[q]]$:

\[
Z_{(C_{A_{2d}}, O)}^{{\mathrm{Hilb}}}(q) = \frac{1 - (\mathbb{L} q^2)^{d+1}}{(1 - q)(1 - \mathbb{L} q^2)}
\]

\[
Z_{(C_{E_6}, O)}^{{\mathrm{Hilb}}}(q) = \frac{1 + \mathbb{L} q^2 + \mathbb{L}^2 q^3 + \mathbb{L}^2 q^4 + \mathbb{L}^3 q^6}{1 - q}
\]

\[
Z_{(C_{E_8}, O)}^{{\mathrm{Hilb}}}(q) = \frac{1 + \mathbb{L} q^2 + \mathbb{L}^2 q^3 + \mathbb{L}^2 q^4 + \mathbb{L}^3 q^5 + \mathbb{L}^3 q^6 + \mathbb{L}^4 q^8}{1 - q}
\]

\[
Z_{(C_{W_8}, O)}^{{\mathrm{Hilb}}}(q) = \frac{1 + \mathbb{L} q^2 + 2\mathbb{L}^2 q^3 + \mathbb{L}^3 q^4 + \mathbb{L}^3 q^5 + (\mathbb{L}^3 + \mathbb{L}^4) q^6 + \mathbb{L}^4 q^8}{1 - q}
\]

\[
Z_{(C_{Z_{10}}, O)}^{{\mathrm{Hilb}}}(q) = \frac{1 + (\mathbb{L} + \mathbb{L}^2) q^2 + \mathbb{L}^2 q^3 + 2\mathbb{L}^3 q^4 + \mathbb{L}^3 q^5 + 2\mathbb{L}^4 q^6 + (\mathbb{L}^4 + \mathbb{L}^5) q^8 + \mathbb{L}^5 q^{10}}{1 - q}
\]
where $\mathbb{L} = [\mathbb{A}^1]$ 
 denotes the Lefschetz motive in $K_0({\mathrm{Var}}_{{\mathbb{C}}})$. 
\end{theorem}

For the singularities of types $A_{2d}$, $E_6$, and $E_8$, these formulas were independently obtained by Watari~\cite{watari2024motivic}, using a different approach based on explicit geometric analysis and results from~\cite{soma2014punctual, soma2015punctual}. Our method, based on the tree structure $G_\Gamma$ and piecewise trivial fibrations, provides a unified and algorithmic derivation that extends naturally to more general singularities, including $W_8$ and $Z_{10}$.\\

In Section \ref{Fixed-number of generators}, in relation to the conjectures of Oblomkov, Rasmussen, Shende, and Maulik~\cite{oblomkov2012hilbert, oblomkov2018hilbert, maulik2016stable}, we study the defining equations of certain stratifications of the punctual Hilbert schemes. For a given $\Gamma$-subsemimodule $\Delta$ and an integer $m \in \mathbb{N}$, we consider the  \[
C^{[\Delta], \leq m} \subseteq C^{[\Delta]},
\]
whose closed points correspond to ideals $I \subset A = \mathcal{O}_{C,O}$ such that $v(I\setminus \{0\}) = \Delta$ and $I$ can be minimally generated by at most $m$ elements. Note that $C^{[\Delta], \leq m} $ is a constructible set, see Proposition \ref{decomposition of C Delta m}.

To describe this subset geometrically, we introduce a stratification defined in terms of \emph{syzygies over $\Gamma$} among the minimal generators of $\Delta$. Specifically, for each subset $\underline{i} = \{i_1, \dots, i_m\} \subseteq \{1, \dots, n\}$ of indices corresponding to a choice of $m$ minimal generators of $\Delta$, and for each combinatorial type of syzygy $\underline{j}$ (encoding dependencies among the generators), we define a constructible subset 
\[
Y_{\underline{i}_{\underline{j}}} \subseteq C^{[\Delta], \leq m}.
\]
These subsets are defined by explicit algebraic equations and inequalities derived from the semigroup structure and the valuations of generators and their relations.

We prove the following structure theorem:
\begin{theorem}\label{thm:stratification}
The set $C^{[\Delta], \leq m}$ admits a finite decomposition
\[
C^{[\Delta], \leq m} = \bigcup_{\substack{\underline{i} \subseteq \{1,\dots,n\} \\ |\underline{i}| = m}} \bigcup_{\underline{j}} Y_{\underline{i}_{\underline{j}}},
\]
where each $Y_{\underline{i}_{\underline{j}}}$ is a locally closed constructible subset defined by syzygy conditions on the generators indexed by $\underline{i}$.

Moreover, for an irreducible curve singularity with semigroup $\Gamma = \langle p, q \rangle$ with $\gcd(p,q)=1$, each stratum $Y_{\underline{i}_{\underline{j}}}$ is isomorphic to a fibred product
\[\label{eq:Yij-geometry}
Y_{\underline{i}_{\underline{j}}} \cong ({\mathbb{C}}^*)^{n - m} \times_{\operatorname{Spec}{\mathbb{C}}} \mathbb{A}^{N(\Delta) - n + m},
\]
where $n$ is the number of minimal generators of $\Delta$ as a $\Gamma$-subsemimodule, and $N(\Delta)$ is the dimension of  $C^{[\Delta]}$ which is an affine space  proven in  \parencite[Theorem 13]{oblomkov2018hilbert}.  
\end{theorem}

Moreover, in the case of an irreducible curve singularity with semigroup $\Gamma = \langle p, q \rangle$ (where $\gcd(p,q) = 1$), we can explicitly determine the intersections of the strata $Y_{\underline{i}_{\underline{j}}}$. These intersections are governed by refinements of syzygy conditions and encode the combinatorial dependencies among different minimal generating sets. This geometric understanding provides a foundation for computing a refined motivic invariant that tracks not only the colength of ideals but also their minimal number of generators.

To this end, for $m, \ell \in \mathbb{N}$, let $C^{[\ell],m} \subseteq C^{[\ell]}$ denote the constructible subset consisting of ideals $I \subset A = \mathcal{O}_{C,O}$ of colength $\ell$ that require exactly $m$ minimal generators. We define the \emph{generalized motivic Hilbert zeta function}:

\begin{align}
Zm^{\mathrm{Hilb}}_{(C,O)}(a^2, q^2) 
&= \sum_{\ell \geq 0} \sum_{m \geq 1} q^{2\ell} (1 - a^2)^{m-1} \big[C^{[\ell],m}\big]\nonumber  \\
&= \sum_{\ell \geq 0} \sum_{\Delta \in \mathscr{D}_\ell} \sum_{m \geq 1} q^{2\ell} (1 - a^2)^{m-1} \big[C^{[\Delta],m}\big], \nonumber
\end{align}
where the second sum runs over all $\Gamma$-subsemimodules $\Delta$ with $\#(\Gamma \setminus \Delta) = \ell$, and $C^{[\Delta],m} = C^{[\Delta]} \cap C^{[\ell],m}$.

This series $Z^{\mathrm{Hilb}}_{(C,O)}(a^2, q^2)$, introduced in this paper, naturally generalizes the zeta function studied in \cite{oblomkov2012hilbert}. There, the Euler characteristic specialization (i.e., the point-counting or topological Euler characteristic) of this series was conjecturally related to the HOMFLY polynomial of the associated algebraic knot. Our construction lifts this invariant to the Grothendieck ring of varieties $K_0({\mathrm{Var}}_{{\mathbb{C}}})$, thereby encoding richer geometric information about the stratification of the Hilbert schemes by the number of generators.

\section{Punctual Hilbert schemes}\label{Punctual Hilbert schemes}
In this section, we recall the definition of the punctual Hilbert schemes of a curve singularity and review some of its fundamental properties. Let $(C,O)$ be the germ of a unibranch curve singularity defined over an algebraically closed field $k$ of characteristic zero, and let $A := \mathcal{O}_{C,O}$ denote its local ring at $O$. Then $A$ is a one-dimensional local domain, and its integral closure $\overline{A}$ in its field of fractions is a discrete valuation ring, isomorphic to $k[[t]]$.

Let $v \colon \overline{A} \setminus \{0\} \to \mathbb{N}$ be the associated discrete valuation, extended by $v(0) = \infty$, which maps a nonzero power series in $k[[t]]$ to its order in $t$. The \emph{value semigroup} of $A$ (or of the singularity $(C,O)$) is defined as
\[
\Gamma := v(A \setminus \{0\}) \subseteq \mathbb{N}.
\]
This is a numerical semigroup: a submonoid of $\mathbb{N}$ with finite complement.

For $n \in \mathbb{N}$, define the ideal
\[
\overline{I}(n) := \{ f \in \overline{A} \mid v(f) \geq n \}
\]
in $\overline{A}$, and set $I(n) := \overline{I}(n) \cap A$, which is an ideal in $A$. The \emph{conductor} of $A$ in $\overline{A}$ is the smallest integer $c$ such that $\overline{I}(c) \subset A$, i.e.,
\[
c := \min\left\{ n \in \mathbb{N} \mid \overline{I}(n) \subset A \right\}.
\]
Equivalently, $c$ is the smallest element such that $t^c k[[t]] \subseteq A$ under a fixed isomorphism $\overline{A} \cong k[[t]]$.

The \emph{$\delta$-invariant} of the singularity is defined as
\[
\delta := \dim_k(\overline{A}/A) = \#(\mathbb{N} \setminus \Gamma).
\]
It measures the arithmetic genus defect of the singularity. We have the inequalities
\[
\delta + 1 \leq c \leq 2\delta,
\]
and $c = 2\delta$ if and only if $A$ is Gorenstein~\cite{serre1959groupes}.

For $\ell \in \mathbb{Z}_{>0}$, the \emph{$\ell$-th punctual Hilbert scheme} of $(C,O)$ is the moduli space
\[
C^{[\ell]} := \left\{ I \subset A \mid I \text{ is an ideal with } \dim_k(A/I) = \ell \right\}.
\]
It parametrizes zero-dimensional subschemes of length $\ell$ supported at $O$.\\

Pfister and Steenbrink~\cite{pfister1992reduced} showed that the punctual Hilbert schemes of a unibranch curve singularity can be embedded as closed subvarieties of a Grassmannian. More precisely, let $\mathscr{M}$ denote the reduced subscheme of the Grassmannian $\operatorname{Gr}(\delta, \overline{A}/I(2\delta))$ defined by
\[
\mathscr{M} := \left\{ W \in \operatorname{Gr}(\delta, \overline{A}/I(2\delta)) \mid W \text{ is an } A\text{-submodule of } \overline{A}/I(2\delta) \right\}.
\]
One can verify that $\mathscr{M}$ is a linear subvariety of the Grassmannian—i.e., it is defined by a system of linear equations in the Plücker coordinates.

\begin{proposition}[\cite{pfister1992reduced}, Theorem~3]\label{prop:pfister-steenbrink}
For every $\ell > 0$, there exists a closed embedding
\[
\phi_\ell \colon C^{[\ell]} \longrightarrow \mathscr{M}.
\]
Moreover, $\phi_\ell$ is bijective (as a map of sets) when $\ell \geq c$, the conductor of the singularity.
\end{proposition}

Since $c \leq 2\delta$, this implies that for $\ell \geq 2\delta$, the structure of $C^{[\ell]}$ stabilizes in the sense that its points are completely determined by the fixed ambient variety $\mathscr{M}$. Consequently, to understand the classes $[C^{[\ell]}]$ in the Grothendieck ring of varieties $K_0(\mathrm{Var}_k)$, it suffices to study the punctual Hilbert schemes for $\ell \leq 2\delta$.

A powerful method to compute these classes is to stratify each $C^{[\ell]}$ into constructible subsets using the valuation $v$ and the value semigroup $\Gamma = v(A \setminus \{0\})$. To this end, recall that a subset $\Delta \subset \Gamma $ is called a \emph{$\Gamma$-subsemimodule} if
\[
\Delta + \Gamma \subseteq \Delta.
\]
Note that such a $\Delta$ need not be contained in $\Gamma$ a priori, but if $\Delta = v(I \setminus \{0\})$ for some nonzero ideal $I \subset A$, then $\Delta \subset \Gamma$. Indeed, for any $f \in I$, $g \in A$, we have $gf \in I$, and $v(gf) = v(g) + v(f)$, so $v(I \setminus \{0\})$ is closed under addition by elements of $\Gamma$, hence is a $\Gamma$-subsemimodule.

For any $\Gamma$-subsemimodule $\Delta$, we define the stratum
\[
C^{[\Delta]} := \left\{ I \in C^{[\ell]} \mid \ell = \#(\Gamma \setminus \Delta),\ v(I \setminus \{0\}) = \Delta \right\}.
\]

Consider the set
\[
\mathscr{D}_\ell := \left\{ \Delta \subset \Gamma \mid \Delta \text{ is a } \Gamma\text{-subsemimodule with } \#(\Gamma \setminus \Delta) = \ell \right\}.
\]

\begin{lemma}\parencite[Lemme 5.1.24]{pol2016singularites} For positive integer $\ell$, 
an ideal $I$ of $A$ belongs to  $C^{[\ell]}$ if and only if $ v(I \setminus \{0\}) \in \mathscr{D}_\ell $. 
\label{ideal codimension}
\end{lemma}

It follows from the  Lemma \ref{ideal codimension} that we have a finite stratification
\[\label{stratification}
C^{[\ell]} = \bigsqcup_{\Delta \in \mathscr{D}_\ell} C^{[\Delta]}.
\]
Each stratum $C^{[\Delta]}$ consists of ideals $I \subset A$ with $v(I \setminus \{0\}) = \Delta$, and the decomposition reflects the combinatorial classification of ideals via their value semimodules.

It is shown in \parencite[Lemma~5]{pfister1992reduced} that each stratum $C^{[\Delta]}$ arises as the intersection of $C^{[\ell]}$ with a Schubert cell in the Grassmannian $\operatorname{Gr}(\delta, \overline{A}/I(2\delta))$. In particular, this implies that $C^{[\Delta]}$ is a open subset of $C^{[\ell]}$, and its geometry is compatible with the cellular decomposition of the ambient Grassmannian.

\section{The tree structure of subsemimodules}\label{tree}

Let $\Gamma$ be the value semigroup associated with the germ of an irreducible curve singularity $(C,O)$. The goal of this section is to endow the set of $\Gamma$-subsemimodules with a tree structure. We will show in the next section that this combinatorial object plays a crucial role in understanding the geometry of the punctual Hilbert schemes of $C$. This construction is inspired by the work of Soma and Watari~\cite{soma2014punctual, soma2015punctual}.\\

For $\Delta \in \mathscr{D}_\ell$, let $\gamma_1, \dots, \gamma_n$ be a minimal system of generators of $\Delta$ as a $\Gamma$-subsemimodule. Up to reordering, we may assume $\gamma_1 < \cdots < \gamma_n$. We use the notation
\[
\Delta = (\gamma_1, \dots, \gamma_n)_{\Gamma} := \bigcup_{i=1}^n (\gamma_i + \Gamma),
\]
where the sum denotes the smallest $\Gamma$-subsemimodule containing all $\gamma_i + \Gamma$. By minimality, we have
\[
\Delta \supsetneq \bigcup_{\substack{j=1 \\ j \neq i}}^n (\gamma_j + \Gamma) \quad \text{for all } i \in \{1, \dots, n\}.
\]

The following result appears in \cite{soma2014punctual}:

\begin{lemma}\label{lem:del}
Let $\Delta = (\gamma_1, \dots, \gamma_n)_{\Gamma}  \in \mathscr{D}_\ell$, where $\{\gamma_1, \dots, \gamma_n\}$ is a minimal system of generators. Then for every $i \in \{1, \dots, n\}$, the set
\[
\Delta \setminus \{\gamma_i\}
\]
belongs to $\mathscr{D}_{\ell+1}$.
\end{lemma}

\begin{proof}
We must show that $\Delta \setminus \{\gamma_i\}$ is a $\Gamma$-subsemimodule and that $\#(\Gamma \setminus (\Delta \setminus \{\gamma_i\})) = \ell + 1$.

First, we verify that $\Delta \setminus \{\gamma_i\}$ is closed under addition by $\Gamma$. Let $x \in \Gamma$ and $y \in \Delta \setminus \{\gamma_i\}$. Since $\Delta$ is a $\Gamma$-subsemimodule, $x + y \in \Delta$. Suppose, for contradiction, that $x + y = \gamma_i$. Then $x \ne 0$ (since $y \ne \gamma_i$), and thus $\gamma_i = x + y \in \gamma_j + \Gamma$ for some $j \ne i$ (as $y \in (\gamma_j + \Gamma)$ for some $j \ne i$). This implies that $\gamma_i$ is in the $\Gamma$-subsemimodule generated by $\{\gamma_j \mid j \ne i\}$, contradicting the minimality of the generating set. Therefore, $x + y \ne \gamma_i$, so $x + y \in \Delta \setminus \{\gamma_i\}$.

Hence, $\Delta \setminus \{\gamma_i\}$ is a $\Gamma$-subsemimodule. Since $\Gamma \setminus (\Delta \setminus \{\gamma_i\}) = (\Gamma \setminus \Delta) \cup \{\gamma_i\}$ and $\gamma_i \notin \Gamma \setminus \Delta$, we have
\[
\#(\Gamma \setminus (\Delta \setminus \{\gamma_i\})) = \#(\Gamma \setminus \Delta) + 1 = \ell + 1.
\]
Thus, $\Delta \setminus \{\gamma_i\} \in \mathscr{D}_{\ell+1}$.
\end{proof}

For $n \in \mathbb{N}$, define
\[
\mathscr{D}_{\ell,n} := \left\{ \Delta \in \mathscr{D}_\ell \mid \Delta \text{ has exactly } n \text{ minimal generators as a } \Gamma\text{-subsemimodule} \right\}.
\]
Note that we have the disjoint decomposition
\[
\mathscr{D}_\ell = \bigsqcup_{n \geq 1} \mathscr{D}_{\ell,n}.
\]

It follows from Lemma~\ref{lem:del} that for each $i \in \{1, \dots, n\}$, there is a canonical \emph{deletion map}
\[
d_{\ell,i} \colon \mathscr{D}_{\ell,n} \to \mathscr{D}_{\ell+1}, \quad \Delta \mapsto \Delta \setminus \{\gamma_i\},
\]
where $\gamma_i$ is one of the minimal generators of $\Delta$.

A natural and important question is whether every element of $\mathscr{D}_{\ell+1}$ arises in this way—i.e., can all $\Gamma$-subsemimodules of codimension $\ell+1$ be obtained by removing a minimal generator from some $\Gamma$-subsemimodule of codimension $\ell$ ? In other words:
\[
\mathscr{D}_{\ell+1} \overset{?}{=} \bigcup_{n \geq 1} \bigcup_{i=1}^n d_{\ell,i}(\mathscr{D}_{\ell,n}).
\]
The answer is affirmative, as will follow from Proposition~\ref{prop:surjective}, which shows that the deletion maps are jointly surjective. To prepare for this result, we first establish two lemmas.

\begin{definition}
For a  $\Gamma-$subsemimodule $\Delta$. We define the \emph{Frobenius element} of $\Delta$ by
\[
\gamma_{\Delta} := \max(\Gamma \setminus \Delta).
\]

\end{definition}

\begin{lemma}\label{lem:delta_union}
Let $\ell \in \mathbb{N}$ and $\Delta \in \mathscr{D}_{\ell+1}$. Then $\Delta \cup \{\gamma_{\Delta}\} \in \mathscr{D}_\ell$.
\end{lemma}

\begin{proof}
We must show that $\Delta \cup \{\gamma_{\Delta}\}$ is a $\Gamma$-subsemimodule and that its complement in $\Gamma$ has cardinality $\ell$.

First, we verify that $\Delta \cup \{\gamma_{\Delta}\}$ is closed under addition by $\Gamma$. Let $x \in \Gamma$ and $y \in \Delta \cup \{\gamma_{\Delta}\}$. We consider two cases:

\begin{enumerate}
    \item \textbf{Case $y \in \Delta$:}  
    Since $\Delta$ is a $\Gamma$-subsemimodule, $x + y \in \Delta \subset \Delta \cup \{\gamma_{\Delta}\}$.

    \item \textbf{Case $y = \gamma_{\Delta}$:}  
    If $x = 0$, then $x + y = \gamma_{\Delta} \in \Delta \cup \{\gamma_{\Delta}\}$.  
    If $x > 0$, then $x + \gamma_{\Delta} > \gamma_{\Delta} = \max(\Gamma \setminus \Delta)$. Since $x + \gamma_{\Delta} \in \Gamma$, it cannot lie in $\Gamma \setminus \Delta$, so $x + \gamma_{\Delta} \in \Delta \subset \Delta \cup \{\gamma_{\Delta}\}$.
\end{enumerate}

Thus, $\Delta \cup \{\gamma_{\Delta}\}$ is a $\Gamma$-subsemimodule. Moreover, since $\gamma_{\Delta} \notin \Delta$, we have
\[
\#(\Gamma \setminus (\Delta \cup \{\gamma_{\Delta}\})) = \#(\Gamma \setminus \Delta) - 1 = (\ell + 1) - 1 = \ell.
\]
Therefore, $\Delta \cup \{\gamma_{\Delta}\} \in \mathscr{D}_\ell$.
\end{proof}

\begin{lemma}\label{lem:gamma_generator}
Let $\Delta \in \mathscr{D}_{\ell+1}$. Then the Frobenius element $\gamma_\Delta = \max(\Gamma \setminus \Delta)$ is part of every minimal generating set of the $\Gamma$-subsemimodule $\Delta \cup \{\gamma_\Delta\}$.
\end{lemma}

\begin{proof}
Suppose, for contradiction, that $\gamma_\Delta$ is not a minimal generator of $\Delta \cup \{\gamma_\Delta\}$. Then it can be written as
\[
\gamma_\Delta = x + y
\]
for some $x \in \Delta$ and $y \in \Gamma \setminus \{0\}$. Since $y > 0$, we have $x < \gamma_\Delta$. But $x \in \Delta$, so $x \in \Gamma$, and thus $\gamma_\Delta = x + y \in \Gamma$. Moreover, since $\Delta$ is a $\Gamma$-subsemimodule and $x \in \Delta$, $y \in \Gamma$, we have $\gamma_\Delta = x + y \in \Delta$. This contradicts the fact that $\gamma_\Delta \notin \Delta$.

Therefore, $\gamma_\Delta$ cannot be expressed as a sum of an element in $\Delta$ and a nonzero element of $\Gamma$, so it must belong to every minimal generating set of $\Delta \cup \{\gamma_\Delta\}$.
\end{proof}

\begin{proposition}\label{prop:surjective}
The deletion maps cover all of $\mathscr{D}_{\ell+1}$, i.e.,
\[
\mathscr{D}_{\ell+1} = \bigcup_{n \geq 1} \bigcup_{i=1}^n d_{\ell,i}(\mathscr{D}_{\ell,n}).
\]
In other words, every element of $\mathscr{D}_{\ell+1}$ arises as $\Delta' \setminus \{\gamma_i\}$ for some $\Delta' \in \mathscr{D}_\ell$ and some minimal generator $\gamma_i$ of $\Delta'$.
\end{proposition}

\begin{proof}
Let $\Delta \in \mathscr{D}_{\ell+1}$. By Lemma~\ref{lem:delta_union}, the set $\Delta \cup \{\gamma_\Delta\}$ belongs to $\mathscr{D}_\ell$. By Lemma~\ref{lem:gamma_generator}, $\gamma_\Delta$ is a minimal generator of $\Delta \cup \{\gamma_\Delta\}$.

Let $\{\gamma_1, \dots, \gamma_n\}$ be a minimal generating set of $\Delta \cup \{\gamma_\Delta\}$, with $\gamma_i = \gamma_\Delta$ for some $i$. Then, by definition of the deletion map $d_{\ell,i}$, we have
\[
d_{\ell,i}(\Delta \cup \{\gamma_\Delta\}) = (\Delta \cup \{\gamma_\Delta\}) \setminus \{\gamma_i\} = (\Delta \cup \{\gamma_\Delta\}) \setminus \{\gamma_\Delta\} = \Delta.
\]
Hence, $\Delta$ lies in the image of $d_{\ell,i}$, and the union of the deletion maps is surjective.
\end{proof}

It follows from Lemma~\ref{lem:delta_union} that we can define a key map for the sequel: for $\ell \in \mathbb{N}$ with $1 \leq \ell < c$, set
\[
m_{\ell+1} \colon \mathscr{D}_{\ell+1} \to \mathscr{D}_{\ell}, \quad \Delta \mapsto \Delta \cup \{\gamma_{\Delta}\},
\]
where $\gamma_{\Delta} = \max(\Gamma \setminus \Delta)$ is the Frobenius element of $\Delta$. This map, which "adds back" the largest gap in $\Gamma \setminus \Delta$, will play a central role in the recursive study of strata in the punctual Hilbert schemes.

The following definition introduces one of the main objects of this article. Recall that $c$ denotes the conductor of the curve singularity $(C,O)$.
\begin{definition}\label{def:G_Gamma}
The \emph{$\Gamma$-subsemimodule graph}, denoted $G_\Gamma = (V, E)$, is the directed, levelled graph defined as follows:
\begin{itemize}
    \item \textbf{Vertices:} For each $\ell$ with $1 \leq \ell \leq c$, we denote the set of the vertices at level  $V_\ell:=\mathscr{D}_\ell$, so that
    \[
    V = \bigsqcup_{\ell=1}^{c} V_\ell.
    \]

    \item \textbf{Edges:} For $2 \leq \ell \leq c$, there is a directed edge from a vertex corresponding to $\Delta \in \mathscr{D}_\ell$ to a vertex corresponding to $\Delta' \in \mathscr{D}_{\ell-1}$ if and only if
    \[
    m_\ell(\Delta) = \Delta'.
    \]
    We denote by $E_\ell$ the set of such edges from level $\ell$ to level $\ell-1$, so that
    \[
    E = \bigsqcup_{\ell=2}^{c} E_\ell.
    \]
\end{itemize}
We refer to $G_\Gamma$ as the \emph{value tree} or \emph{semimodule tree} associated with the singularity $(C,O)$.
\end{definition}

\begin{theorem}\label{Tree}
The $\Gamma$-subsemimodule graph $G_{\Gamma}$ admits a canonical tree structure. The set $\mathscr{D}_1$ consists of a single element, $\Gamma \setminus \{0\}$, which we designate as the root of the graph $G_{\Gamma}$.
\end{theorem}

\begin{proof}
We establish that $G_{\Gamma}$ is a tree by verifying the two defining properties of a tree in graph theory:

\begin{enumerate}
    \item \textbf{Connectedness:} Let $\Delta_1 \in \mathscr{D}_{\ell_1}$ and $\Delta_2 \in \mathscr{D}_{\ell_2}$ be arbitrary vertices in $G_{\Gamma}$. Then there exists a finite sequence of maps
    \[
    \Delta_1 \to m_{\ell_1}(\Delta_1) \to \cdots \to \Gamma \setminus \{0\} \leftarrow \cdots \leftarrow m_{\ell_2}(\Delta_2) \leftarrow \Delta_2,
    \]
    where each arrow corresponds to an edge in $G_{\Gamma}$, and $\Gamma \setminus \{0\} \in \mathscr{D}_1$ is the root. This shows that any two vertices are connected by a path, so $G_{\Gamma}$ is connected.

    \item \textbf{Acyclicity:} it follows from the uniqueness of Frobenius element of a $\Gamma$-subsemimodule.

\end{enumerate}

Thus, $G_{\Gamma}$ is connected and acyclic, and hence has the structure of a tree, with root $\Gamma \setminus \{0\}$.
\end{proof}

We can identify the edges $E_\ell = \{m_{\ell|_{\Delta}} : \Delta \to m_\ell(\Delta) \mid \Delta \in \mathscr{D}_\ell\}_{2\leq \ell \leq c}$ to $\{d_{\ell-1,i \mid m_\ell(\Delta)}:    m_\ell(\Delta) \rightarrow \Delta \mid \Delta \in \mathscr{D}_\ell\}_{2\leq \ell \leq c}$ or to $\{d_{\ell,i\mid \Delta}:\Delta\rightarrow d_{\ell,i}(\Delta)\mid m_{\ell+1}d_{\ell,i}(\Delta)=\Delta\}_{1\leq \ell< c}$. \\

In the sequel, for a fixed parameter $\ell$, we adopt the following simplified notations when there is no ambiguity: set $d_i := d_{\ell,i}$ and $m := m_\ell$.

For a $\Gamma$-subsemimodule $\Delta$, we define its \emph{conductor} as
\[
c(\Delta) := \min\left\{ i \in \Delta \mid i + \mathbb{N} \subseteq  \Delta \right\}.
\]

The tree $G_\Gamma$ has the following properties:

\begin{proposition}\label{property_of_tree}
Let $\Gamma$ be a numerical semigroup. We consider the following construction:

\begin{enumerate}
    \item[(i)] Define a sequence of $\Gamma$-subsemimodules $\{\Delta^{(\ell)}\}_{1 \leq \ell \leq c}$ recursively by:
    \begin{itemize}
        \item $\Delta^{(1)} \coloneqq \Gamma \setminus \{0\}$,
        \item $\Delta^{(\ell)} \coloneqq d_1(\Delta^{(\ell-1)}) \in \mathscr{D}_\ell$ for $1 < \ell \leq c$.
    \end{itemize}
    For each $1 < \ell \leq c$, we have $m \circ d_1(\Delta^{(\ell-1)}) = \Delta^{(\ell-1)}$, so the map 
     $d_1$  induces an edge 
    \[
    \Delta^{(\ell-1)} \to \Delta^{(\ell)}
    \]
    in the tree $G_\Gamma$. We refer to $\Delta^{(\ell)}$ as the \emph{level-$\ell$ root vertex} of $G_\Gamma$.

    \item[(ii)] For any $\Gamma$-subsemimodule $\Delta \in \mathscr{D}_\ell$, the following dichotomy holds:
    \begin{itemize}
        \item If $\gamma_\Delta < \min(\Delta)$, then $\Delta = \Delta^{(\ell)}$ for some $\ell \geq 1$. In this case, $\Delta$ lies on the root path of the tree.
        \item If $\gamma_\Delta > \min(\Delta)$, then $m \circ d_1(\Delta) \neq \Delta$, and hence the deletion $d_1(\Delta)$ does not correspond to an edge pointing back to $\Delta$ in $G_\Gamma$. Such $\Delta$ corresponds to a non-root vertex off the main branch.
    \end{itemize}
\end{enumerate}
\end{proposition}

\begin{proof}
We prove each part in turn.

\medskip

{(i)} We proceed by induction to show that $ m_\ell(\Delta^{(\ell)}) = \Delta^{(\ell-1)} $ for all $ 2 \leq \ell \leq c $.

For the base case $ \ell = 2 $, recall that $ \Delta^{(1)} = \Gamma \setminus \{0\} $, so $ \min(\Delta^{(1)}) $ is the smallest positive element of $ \Gamma $. By construction, $ \Delta^{(2)} = d_1(\Delta^{(1)}) = \Delta^{(1)} \setminus \{\min(\Delta^{(1)})\} $. Thus,
\[
\Gamma \setminus \Delta^{(2)} = \{0\} \cup (\Gamma \setminus \Delta^{(1)}) = \{0, \min(\Delta^{(1)})\},
\]
and since $ \min(\Delta^{(1)}) > 0 $, we have
\[
\gamma_{\Delta^{(2)}} = \max(\Gamma \setminus \Delta^{(2)}) = \min(\Delta^{(1)}).
\]
Therefore,
\[
m_2(\Delta^{(2)}) = \Delta^{(2)} \cup \{\gamma_{\Delta^{(2)}}\} = \Delta^{(2)} \cup \{\min(\Delta^{(1)})\} = \Delta^{(1)},
\]
which establishes the base case.

Now suppose the claim holds for some $ \ell - 1 $, i.e., $ m_{\ell-1}(\Delta^{(\ell-1)}) = \Delta^{(\ell-2)} $. By definition, $ \Delta^{(\ell)} = d_1(\Delta^{(\ell-1)}) = \Delta^{(\ell-1)} \setminus \{\min(\Delta^{(\ell-1)})\} $. Then
\[
\Gamma \setminus \Delta^{(\ell)} = (\Gamma \setminus \Delta^{(\ell-1)}) \cup \{\min(\Delta^{(\ell-1)})\},
\]
and since $ \min(\Delta^{(\ell-1)}) > \max(\Gamma \setminus \Delta^{(\ell-1)}) $ (by the inductive structure), we have
\[
\gamma_{\Delta^{(\ell)}} = \max(\Gamma \setminus \Delta^{(\ell)}) = \min(\Delta^{(\ell-1)}).
\]
Thus,
\[
m_\ell(\Delta^{(\ell)}) = \Delta^{(\ell)} \cup \{\gamma_{\Delta^{(\ell)}}\} = \Delta^{(\ell)} \cup \{\min(\Delta^{(\ell-1)})\} = \Delta^{(\ell-1)},
\]
completing the induction.

\medskip

{(ii)} Suppose $ \gamma_\Delta < \min(\Delta) $. Then $ \min(m(\Delta)) = \min(\Delta \cup \{\gamma_\Delta\}) = \gamma_\Delta $, since $ \gamma_\Delta \notin \Delta $ and $ \gamma_\Delta < \min(\Delta) $. In particular, $ \gamma_\Delta $ is strictly smaller than all elements of $ \Delta $, so it becomes the minimal element of $ m(\Delta) $. Therefore,
\[
d_1(m(\Delta)) = m(\Delta) \setminus \{\min(m(\Delta))\} = (\Delta \cup \{\gamma_\Delta\}) \setminus \{\gamma_\Delta\} = \Delta,
\]
which shows that $ \Delta $ lies on the image of the $ d_1 $-path starting from $ \Delta^{(1)} $.

Moreover, since $ c(m(\Delta)) \leq \gamma_\Delta < c(\Delta) $, the conductor strictly decreases along this process. Hence, repeated application of $ m $ eventually yields a subsemimodule $ m^k(\Delta) $ such that $ \gamma_{m^k(\Delta)} < \min(m^k(\Delta)) $ and $ c(m^k(\Delta)) $ is minimal. This process terminates at $ \Delta^{(1)} = \Gamma \setminus \{0\} $, so $ \Delta = \Delta^{(\ell)} $ for some $ \ell $.

Conversely, if $ \gamma_\Delta > \min(\Delta) $, then $ \min(m(\Delta)) = \min(\Delta) $, and $ \gamma_\Delta $ is not the smallest element of $ m(\Delta) $. Therefore, $ d_1(m(\Delta)) $ removes $ \min(\Delta) $, not $ \gamma_\Delta $, so
\[
d_1(m(\Delta)) \neq \Delta,
\]
and thus the edge $ \Delta \to m(\Delta) $ is not reversible via $ d_1 $. This means $ \Delta $ does not lie on the main $ d_1 $-branch of the tree.
\end{proof}

\begin{example}\label{ex:34}
Let $A = k[[t^3, t^4]]$, the local ring of the plane curve singularity defined by $y^3 - x^4 = 0$. The associated value semigroup is
\[
\Gamma = \langle 3,4 \rangle = \{0, 3, 4, 6, 7, 8, 9, \dots\},
\]
which contains all integers $n \geq 6$ except 5.

Consider the subsemimodule $\Delta = (3,4) = \Gamma \setminus \{0\} \in \mathscr{D}_1$, which is the root of the tree $G_\Gamma$. Applying the deletion maps $d_{1,i}$ corresponding to each minimal generator, we obtain:
\begin{itemize}
    \item $d_{1,1}((3,4)_{\Gamma} ) = (3,4)_{\Gamma}  \setminus \{3\} = \{4,6,7,8,9,\dots\} = (4,6)_{\Gamma}  \in \mathscr{D}_2$,
    \item $d_{1,2}((3,4)_{\Gamma} ) = (3,4)_{\Gamma}  \setminus \{4\} = \{3,6,7,8,9,\dots\} = (3,8)_{\Gamma}  \in \mathscr{D}_2$.
    \item $d_{2,1}((4,6)_{\Gamma} ) = (4,6)_{\Gamma}  \setminus \{4\} = \{6,7,8,9,\dots\} = (6,7,8)_{\Gamma}  \in \mathscr{D}_3. $
\end{itemize}
Note that both $(4,6)$ and $(3,8)$ are distinct elements in $\mathscr{D}_2$, corresponding to two different branches from the root.
\begin{figure}[h]
\centering
\begin{tikzpicture}[level distance=0.1cm,
   level 1/.style={sibling distance=5cm, level distance=1.5cm},
   level 2/.style={sibling distance=6cm, level distance=2cm},
    level 3/.style={sibling distance=3cm, level distance=2cm},
    level 4/.style={sibling distance=2.3cm, level distance=2cm},
    level 5/.style={sibling distance=2cm, level distance=2cm},
   grow'=up]
 
\node {$(3,4)_{\Gamma} $} 
   child {node {$(4,6)_{\Gamma} $} 
    child {node {$(6,7,8)_{\Gamma} $}  
        child {node {$(7,8,9)_{\Gamma} $}
            child {node {$(8,9,10)_{\Gamma} $}
                child{node {$(9,10,11)_{\Gamma} $}edge from parent node[left]{$d_{1}$}}
                    child {node {$(8,10)_{\Gamma} $}edge from parent node[right]{$d_{2}$}}
                        child {node {$(8,9)_{\Gamma} $}edge from parent node[right]{$d_{3}$}}
                        edge from parent node[left]{$d_{1}$}}
                            child {node {$(7,9)_{\Gamma} $}
                            child [missing]{}
                            child [missing]{}
                                child {node {$(7,12)_{\Gamma} $}edge from parent node[right]{$d_{2}$}}
                                edge from parent node[right]{$d_{2}$}}
                                    child {node {$(7,8)_{\Gamma} $}edge from parent node[right]{$d_{3}$}}
                                    edge from parent node[left]{$d_{1}$}}  
                                        child {node {$(6,8)_{\Gamma} $}
                                            child[missing]{}
                                                child[missing]{}
                                                    child {node {$(6,11)_{\Gamma} $}
                                                    child {node {$(6)_{\Gamma} $}edge from parent node[right]{$d_{2}$}}edge from parent node[right]{$d_{2}$}}
                                                    edge from parent node[right]{$d_{2}$}}   
                                                        child {node {$(6,7)_{\Gamma} $}edge from parent node[right]{$d_{3}$}}edge from parent node[left]{$d_{1}$}}
                                                            child {node {$(4,9)_{\Gamma} $ }
                                                                child {node {$(4)_{\Gamma} $}edge from parent node[right]{$d_{2}$}}edge from parent node[right]{$d_{2}$}}
                                                                    edge from parent node[left]{$d_{1}$}}
                                                                        child {node {$(3,8)_{\Gamma} $} 
                                                                            child {node {$(3)_{\Gamma} $}
                                                                            edge from parent node[right]{$d_{2}$}}edge from parent node[right]{$d_{2}$}
};

\end{tikzpicture}
    \caption{Tree for the case of  $E_{6}$ type singularity}
    \label{fig:enter-label}
\end{figure}\\
\end{example}

\section{Piecewise fibrations induced by the edges of the \texorpdfstring{$\Gamma$}{Gamma}-subsemimodules tree}\label{fibrations}

In this section, we focus on the geometry of punctual Hilbert schemes of irreducible plane curve singularities defined by equations of the form $y^p - x^q = 0$ with $\gcd(p, q) = 1$ (the so-called $(p,q)$ case), as well as singularities whose value semigroups are monomial, in a sense to be specified in  Section \ref{Appendix}.

Our main result Theorem~\ref{piecewise fibration} shows that an edge in the $\Gamma$-subsemimodule tree $G_\Gamma$, joining a semimodule $\Delta$ to $m(\Delta)$, induces a piecewise trivial fibration between the corresponding strata $C^{[\Delta]}$ and $C^{[m(\Delta)]}$ in the punctual Hilbert schemes. This fibration structure reveals a recursive geometric organization of the Hilbert scheme.

Corollary~\ref{dim of Hilbert scheme} provides a new perspective on the geometry of $C^{[\Delta]}$ in the case $C = \{y^p = x^q\}$, recovering a result from \parencite[Theorem~13]{oblomkov2018hilbert} via our combinatorial and valuative approach. Our method relies on the explicit defining equations of the strata $C^{[\Delta]}$, introduced in \parencite[Proposition~12]{oblomkov2018hilbert}.

We begin by recalling some key constructions from \cite{oblomkov2018hilbert}. Let $(C,O)$ be the germ of a unibranch curve singularity—either of type $(p,q)$ or with a monomial semigroup—and let $A := \mathcal{O}_{C,O} \subset {\mathbb{C}}[[t]]$ denote its complete local ring. We fix a basis of $A$ that is compatible with the monomial basis of the semigroup algebra ${\mathbb{C}}[\Gamma]$. Specifically, for each $i \in \Gamma$, we consider elements of the form
\[
\phi_i = t^i + \sum_{j > i} a_{i,j} t^j,
\]
where $a_{i,j} \in {\mathbb{C}}$. In the $(p,q)$ case and for monomial semigroups, we may choose $\phi_i = t^i$, i.e., the basis is purely monomial.

Let $J \subset A$ be an ideal, and let $\Delta = v(J \setminus \{0\})$ be its value semimodule. Fix a minimal system of generators $\gamma_1, \dots, \gamma_n$ of $\Delta$ as a $\Gamma$-subsemimodule. For each $\gamma_j$, we choose a generator $f_{\gamma_j} \in J$ with $v(f_{\gamma_j}) = \gamma_j$, of the form
\[
f_{\gamma_j} = \phi_{\gamma_j} + \sum_{k \in \Gamma_{>\gamma_j} \setminus \Delta} \lambda_j^{k - \gamma_j} \phi_k,
\]
where $\Gamma_{>\gamma_j} = \{k \in \Gamma \mid k > \gamma_j\}$, and the coefficients $\lambda_j^{k - \gamma_j} \in {\mathbb{C}}$ are uniquely determined modulo higher-order terms. These representatives are central to the parametrization of the stratum $C^{[\Delta]}$.

The condition $k \in \Gamma_{>\gamma_j} \setminus \Delta$ can be achieved through a normalization or elimination process. For example, consider the ideal
\[
\langle t^4, t^6 + t^7 \rangle \subset {\mathbb{C}}[[t^2, t^3]].
\]
This ideal can also be generated by $t^4$ and $t^6$, since
\[
t^6 = (t^6 + t^7) - t^3 \cdot t^4,
\]
and $t^3 \in {\mathbb{C}}[[t^2, t^3]]$. This kind of reduction allows us to eliminate terms whose valuations lie in $\Delta$, ensuring that each generator is expressed using only basis elements $\phi_k$ with $k \in \Gamma \setminus \Delta$ and $k > \gamma_j$.

This normalization guarantees that the ideal $J$ with $v(J\setminus \{0\}) = \Delta$ is uniquely determined by the coefficients $\lambda_j^{k - \gamma_j}$. In other words, the choice of generators in the form
\[
f_{\gamma_j} = \phi_{\gamma_j} + \sum_{k \in \Gamma_{>\gamma_j} \setminus \Delta} \lambda_j^{k - \gamma_j} \phi_k \in \mathcal{O}_{C}
\]
provides a global chart for the stratum of ideals with fixed value semimodule $\Delta$.

Thus, we can identify the space of such ideals with an affine space of parameters. Define
\[
\operatorname{Gen}_{\Delta} := \operatorname{Spec} {\mathbb{C}}\left[ \lambda_j^{k - \gamma_j} \mid 1 \leq j \leq n,\ k \in \Gamma_{>\gamma_j} \setminus \Delta \right].
\]
Then $\operatorname{Gen}_{\Delta}$ is an affine space of dimension
\[
N = \sum_{j=1}^{n} \#(\Gamma_{>\gamma_j} \setminus \Delta),
\]
and each point in $\operatorname{Gen}_{\Delta}$ corresponds to a tuple of generators $(f_{\gamma_1}, \dots, f_{\gamma_n})$.

We define the deformed generators as functions on $\operatorname{Gen}_{\Delta}$:
\[
\tau_{\gamma_j}(\lambda^{\bullet}_{\bullet}) = \phi_{\gamma_j} + \sum_{k \in \Gamma_{>\gamma_j} \setminus \Delta} \lambda_j^{k - \gamma_j} \phi_k, \quad \lambda_j^{k - \gamma_j} \in {\mathbb{C}}[[\lambda^{\bullet}_{\bullet}, t]].
\]

This data defines an \emph{exponential map}
\[
\operatorname{Exp}_{\Delta} \colon \operatorname{Gen}_{\Delta} \longrightarrow \bigsqcup_{\ell \geq 1} C^{[\ell]}, \quad (\lambda^{\bullet}_{\bullet}) \mapsto \left\langle \tau_{\gamma_1}(\lambda^{\bullet}_{\bullet}), \dots, \tau_{\gamma_n}(\lambda^{\bullet}_{\bullet}) \right\rangle ,
\]
which sends a tuple of coefficients to the ideal generated by the corresponding deformed generators.

By the uniqueness of the coefficients $\lambda_j^{k - \gamma_j}$ for a given ideal $J$ with $v(J\setminus \{0\}) = \Delta$, established via the elimination process, the exponential map restricts to a bijective morphism
\[
\operatorname{Exp}_{\Delta} \colon \operatorname{Exp}_{\Delta}^{-1}(C^{[\Delta]}) \longrightarrow C^{[\Delta]}.
\]
This bijectivity follows from \parencite[Theorem~27]{oblomkov2012hilbert}, which ensures that every ideal in $C^{[\Delta]}$ admits a unique normalized generating set of the prescribed form. Consequently, the stratum $C^{[\Delta]}$ embeds into the parameter space $\operatorname{Gen}_{\Delta}$, yielding a locally closed embedding
\[
C^{[\Delta]} \hookrightarrow \operatorname{Gen}_{\Delta} \cong \mathbb{A}^N,
\]
where $N = \sum_{j=1}^n \#(\Gamma_{>\gamma_j} \setminus \Delta)$.

\begin{remark}\label{cancelation}

It is important to note that the inclusion $C^{[\Delta]} \hookrightarrow \operatorname{Gen}_\Delta$ does \emph{not} imply that every point in $\operatorname{Gen}_\Delta$ corresponds to an ideal $J$ with $v(J\setminus \{0\}) = \Delta$. 
This phenomenon is illustrated by the following example, adapted from \cite{oblomkov2012hilbert}. Let $\mathcal{O}_C = {\mathbb{C}}[[t^3, t^7]]$, whose value semigroup is $\Gamma = \langle 3,7 \rangle$. Consider the $\Gamma$-subsemimodule $\Delta = (6, 10)_{\Gamma} $, minimally generated by $6$ and $10$. The element
\[
J = \langle t^6 + t^7, t^{10}\rangle  \in \operatorname{Gen}_\Delta
.\] 
However, $v(J\setminus \{0\}) \neq \Delta$. Indeed, consider the combination:
\[
t^7 \cdot (t^6 + t^7) - t^3 \cdot t^{10} = t^{13} + t^{14} - t^{13} = t^{14} \in J.
\]
Thus, $v(t^{14}) = 14 \in v(J\setminus \{0\})$, but $14 \notin (6, 10 )_{\Gamma} $, since $(6,10 )_{\Gamma}  = \{6, 9, 10,12, 15, 16,\dots\}$ and $14$ is not in this semimodule. Hence, $v(J\setminus \{0\}) \supsetneq \Delta$, and $J \notin C^{[\Delta]}$.
\end{remark}

To fully determine
 $C^{[\Delta]}$, one must control the syzygies among the generators of the ideal, such as the one appearing in the previous remark. These syzygies can introduce elements of valuation outside $\Delta$, and thus must be constrained to ensure that $v(J\setminus \{0\}) = \Delta$.

Let $\gamma_1, \dots, \gamma_n$ be a fixed minimal system of generators of the $\Gamma$-subsemimodule $\Delta$. This choice determines a surjective map of $\mathbb{C}[\Gamma]$-modules:
\[
G \colon {\mathbb{C}}[\Gamma]^{\oplus n} \to {\mathbb{C}}[\Delta] := {\mathbb{C}}[\Gamma] \cdot \{ t^j \mid j \in \Delta \},
\quad
f_i e_i \mapsto f_i t^{\gamma_i},
\]
where $\{e_1, \dots, e_n\}$ is the standard basis, and the target is the ${\mathbb{C}}[\Gamma]$-submodule of ${\mathbb{C}}[[t]]$ generated by $\{t^j \mid j \in \Delta\}$.

We extend this to a presentation:
\begin{equation}\label{eq:syzygy-presentation}
{\mathbb{C}}[\Gamma]^{\oplus m} \xrightarrow{S} {\mathbb{C}}[\Gamma]^{\oplus n} \xrightarrow{G} {\mathbb{C}}[\Delta] \to 0,
\end{equation}
where $S$ denotes the composition ${\mathbb{C}}[\Gamma]^{\oplus m} \to \ker G \hookrightarrow {\mathbb{C}}[\Gamma]^{\oplus n}$. According to \parencite[Lemma~4]{piontkowski2007topology}, the kernel $\ker G$ is generated by homogeneous elements of the form
\[
(0, \dots, t^{b_{\gamma_i}}, 0, \dots, 0, -t^{b_{\gamma_{i'}}}, 0, \dots, 0),
\]
where $b_{\gamma_i} + \gamma_i = b_{\gamma_{i'}} + \gamma_{i'} =: \sigma_i$. We call $\sigma_i$ the \emph{order} of the $i$-th syzygy.

Writing $S = (s_1, \dots, s_m)$, each entry of $S$ is of the form
\[
(s_i)_j = u_i^j t^{\sigma_i - \gamma_j}, \quad u_i^j \in {\mathbb{C}}.
\]

Now, a point $\lambda \in \operatorname{Gen}_\Delta$ determines a lift
\[
\mathcal{G}_\lambda \in \operatorname{Hom}_{\mathcal{O}_C}(\mathcal{O}_C^{\oplus n}, \mathcal{O}_C),
\quad
(\mathcal{G}_\lambda)_j := \tau_{\gamma_j}(\lambda^\bullet_\bullet),
\]
where $\tau_{\gamma_j}$ are the deformed generators defined earlier. This lift replaces the monomial map $G$ with an actual generating set of an ideal in $\mathcal{O}_C$.

For integers $a < b$, define the open interval
\[
(a, b) := \{ c \in \mathbb{Z} \mid a < c < b \}.
\]
For a fixed integer $e$ with $0 \leq e < c(\Delta)$, we define
\[
\Delta_{> e, < c(\Delta)} := (e, c(\Delta)) \cap \Delta.
\]
In particular, for $s \in \Delta_{> 0, < c(\Delta)}$, we fix a decomposition
\[
s = \gamma_{g(s)} + \rho(s), \quad \rho(s) \in \Gamma,
\]
which defines a map
\[
g \colon \Delta_{> 0, < c(\Delta)} \to \{1, \dots, n\},
\]
assigning to each $s$ the index of the generator $\gamma_{g(s)}$ from which it arises.

Let $\operatorname{Syz}_\Delta$ be the affine space with coordinates
\[
\nu_{is}^{s - \sigma_i}, \quad \text{for } i = 1, \dots, m, \quad \sigma_i < s < c(\Delta), \quad s \in \Delta.
\]
To a closed point $\nu \in \operatorname{Syz}_\Delta$, we associate an $n \times m$ matrix $\mathcal{S}_\nu$ whose $i$-th column is given by
\[
(\mathcal{S}_\nu)_i^j = u_i^j \phi_{\sigma_i - \gamma_j} + \sum_{\substack{s \in \Delta_{> \sigma_i, < c(\Delta)} \\ g(s) = j}} \nu_{is}^{s - \sigma_i} \phi_{s - \gamma_j}.
\]

\begin{proposition}[\cite{oblomkov2018hilbert}, Proposition~12]\label{prop:OS12}
The subvariety of $\operatorname{Gen}_\Delta \times \operatorname{Syz}_\Delta$ defined by the condition
\[
\mathcal{G}_\lambda \circ \mathcal{S}_\nu \equiv 0 \mod t^{c(\Delta)}
\]
maps bijectively onto  $C^{[\Delta]}$. 
\end{proposition}

Consequently, $C^{[\Delta]}$ can be realized as a closed subvariety of $\operatorname{Gen}_\Delta$, defined by the ideal $\mathcal{I} \subset {\mathbb{C}}[\lambda_{\bullet}, \nu^{\bullet}_{\bullet,\bullet}]$ generated by the entries of the composition $\mathcal{G}_\lambda \circ \mathcal{S}_\nu$, truncated at order $c(\Delta)$.

Explicitly, for each syzygy index $i$, the $i$-th component of the composition is given by:
\[
(\mathcal{G}_\lambda \circ \mathcal{S}_\nu)_i = \sum_{j=1}^n (\mathcal{G}_\lambda)_j \cdot (\mathcal{S}_\nu)_i^j.
\]

Expanding the product and collecting terms up to order $< c(\Delta)$, we obtain:
\[
\begin{aligned}
& \sum_j\left(u_i^j \phi_{\sigma_i-\gamma_j} \phi_{\gamma_j}+\sum_{s\in\Delta_{> \sigma_{i},< c(\Delta)}, g(s)=j}\nu_{i s}^{s-\sigma_i} \phi_{s-\gamma_j} \phi_{\gamma_j}+\sum_{k \in \Gamma_{>\gamma_j}\setminus \Delta} u_i^j \lambda_j^{k-\gamma_j} \phi_{\sigma_i-\gamma_j} \phi_k\right. \\
& \left.+\sum_{\substack{s\in \Delta_{> \sigma_{i},< c(\Delta)}\\
g(s)=j\\k \in \Gamma_{>\gamma_j}\setminus \Delta}}\nu_{i s}^{s-\sigma_i} \lambda_j^{k-\gamma_j} \phi_{s-\gamma_j} \phi_k\right) \\
&
\end{aligned}
\]

All terms are evaluated modulo $t^{c(\Delta)}$, and the vanishing of these expressions defines the ideal $\mathcal{I}$ cutting out the fiber product. The projection to $\lambda$-coordinates then yields the defining equations of $C^{[\Delta]}$ as a constructible (or closed) subset of $\operatorname{Gen}_\Delta$.

\noindent
 
We expand $(\mathcal{G}_{\lambda}\circ \mathcal{S}_{\nu})_{i}$ in the basis $\phi_{k}$ (as ${\mathbb{C}}$-basis) and denote by $Eq_{i}^{r}$ the coefficient of $\phi_{r+\sigma_{i}}$. Note that $Eq_{i}^{r}$ does not appear (or is trivial) if $r+\sigma_{i}\notin\Gamma$ or $r+\sigma_{i}\geq c(\Delta)$. The nontrivial equations are of the following form:
\[\label{eqform}
    Eq_{i}^{r}=L_{i}^{r}+\mbox{ terms in } \lambda^{<r},\nu^{<r},
\]   
where
\[L_{i}^{r}:=\delta_{\Delta\cap (\sigma_{i},c(\Delta))}(r+\sigma_{i})\nu^{r}_{i,r+\sigma_{i}}+\sum_{j=1}^{n}\delta_{\Gamma\setminus\Delta}(r+\gamma_{j})u_{i}^{j}\lambda_{j}^{r},\]
$\delta_{\Delta\cap (\sigma_{i},c(\Delta))}$ and $\delta_{\Gamma\setminus\Delta}$ being indicator functions. Note that the polynomials $L_{i}^{r}$ are linear.\\

\begin{remark}[\cite{oblomkov2018hilbert}]
Let $C$ be a plane curve singularity defined by $x^p = y^q$ with $\gcd(p,q) = 1$. Then, for all $r$ and $i$, the linear parts $L_i^r$ of the equations $Eq_i^r$ are linearly independent.
\end{remark}

We now introduce the definition of a discrete syzygy for a $\Gamma$-subsemimodule, which is associated with the presentation given in \eqref{eq:syzygy-presentation}.

For $\Delta$ a $\Gamma-$subsemimodule of $\Gamma,$ we denote by $T_\Delta=\{\gamma_{1}<\dots<\gamma_{n}\}$  its minimal system of generators.\\

For a $\Gamma$-subsemimodule $\Delta \subset \Gamma$, let 
\[
T_\Delta = \{\gamma_1 < \gamma_2 < \cdots < \gamma_n\}
\]
denote its unique minimal system of generators as a $\Gamma$-module.

\begin{definition}\label{syzygy and aug syzygy}
1. We define the Augmented Syzygy  of  $\Delta$  as 
\[
ASyz(\Delta) := \{( \gamma_{i_{1}},\gamma_{j_{1}},  \sigma_{1}), \dots, (\gamma_{i_{m}},\gamma_{j_{m}},  \sigma_{m}) \}  / \sim
\]
where $\sim$ is the equivalence relation given by $(\gamma_a, \gamma_b, \sigma) \sim (\gamma_b, \gamma_a, \sigma)$ and $(\gamma_{i_{1}},\gamma_{j_{1}},  \sigma_{1}), \dots, (\gamma_{i_{m}},\gamma_{j_{m}},  \sigma_{m})$ be the tuples corresponding to the syzygy $S=(s_{1},\dots,s_m)$  in the presentation \eqref{eq:syzygy-presentation}.
\end{definition}

\noindent
2. We define the \emph{syzygy set} of $\Delta$ as
\[
\operatorname{Syz}(\Delta) := \left\{ \sigma \in \Delta \mid 
\begin{array}{c}
\exists\, \gamma_{i_1} \neq \gamma_{i_2} \in T_\Delta, \text{ and } b_1, b_2 \in \Gamma, \\
\text{such that } \sigma = \gamma_{i_1} + b_1 = \gamma_{i_2} + b_2
\end{array}
\right\}.
\]
This set consists of all elements in $\Delta$ that can be expressed in at least two distinct ways as a sum of a generator $\gamma_i$ and an element of $\Gamma$.

It follows from the definition that $\operatorname{Syz}(\Delta)$ is itself a $\Gamma$-subsemimodule of $\Delta$. Let $\{\sigma_1, \dots, \sigma_{m^{\prime}}\}$ denote its minimal system of generators.\\

\begin{proposition}\cite{oblomkov2018hilbert} \label{gamma = p,q, equal ASyz TSyz}
We consider $\Gamma = \langle p,q \rangle$ with $p<q$ coprime. Let $\Delta$ be a subsemimodule of $\Gamma$.
There exists a bijection between $ASyz(\Delta)$ and $T_{Syz(\Delta)}$ — the minimal generating system of $Syz(\Delta)$. Specifically, the projection
 \[ASyz(\Delta) \rightarrow T_{Syz(\Delta)}     .\]
 \[(\gamma_{i},\gamma_{j},\sigma)  \mapsto \sigma    \]
 is a bijection. Furthermore, the number of minimal generators of $Syz(\Delta)$ is the same as that of $\Delta$.

\begin{proof}
First, we show that the map is well-defined. Let $(\gamma_{i}, \gamma_{j}, \sigma) \in ASyz(\Delta)$. By definition, there exist $b_{\gamma_{i}}$ and $b_{\gamma_{j}}$ such that the vector
\[
(0, \dots, t^{b_{\gamma_{i}}}, 0, \dots, 0, -t^{b_{\gamma_{j}}}, 0, \dots, 0)
\]
is a minimal generator of the kernel of $G$ in the $\mathbf{C}[\Gamma]$-module presentation \eqref{eq:syzygy-presentation}. Here, $\sigma$ satisfies the relation $\sigma = b_{\gamma_{i}} + \gamma_{i} = b_{\gamma_{j}} + \gamma_{j}$. Consequently, $\sigma$ is a minimal generator of $Syz(\Delta)$. If this were not the case, it would contradict the fact that the vector $(0, \dots, t^{b_{\gamma_{i}}}, 0, \dots, 0, -t^{b_{\gamma_{j}}}, 0, \dots, 0)$ is a minimal generator of the kernel. So the map is well-defined.

By the Chinese remainder theorem (or see also Section~5 of \cite{oblomkov2012hilbert}), every element $m \in \Gamma$ admits a unique presentation of the form $m = aq + bp$ with $0 \leq a < p$.

Assume $\Delta$ is minimally generated by $\{\gamma_{1},\dots, \gamma_{n}\}$ as a $\Gamma$-module. We introduce a total ordering on the generators as $\gamma_1 \prec \gamma_2 \prec \cdots \prec \gamma_n$. This ordering is determined by their coordinates: writing $\gamma_i = a_i q + b_i p$, we require $a_i \leq a_{i+1}$ and $b_i \geq b_{i+1}$ for $1 \leq i \leq n-1$. Moreover, we take the indices modulo $n$, i.e., $\gamma_i = \gamma_{i+n}$.

We define $\sigma_i := a_{i+1} q + b_i p$. Then $\sigma_i$ can be expressed in terms of $\gamma_{i}$ and $\gamma_{i+1}$ as:
\[
\sigma_i = \gamma_{i} + (a_{i+1}-a_{i})q = \gamma_{i+1} + (b_{i}-b_{i+1})p.
\]
Thus, $\{ (\gamma_{i},\gamma_{i+1},\sigma_{i}) \mid i = 1,\dots, n \} \subseteq ASyz(\Delta)$.

Now suppose $\sigma_{i}$ can also be generated by some $\gamma_{k}$ with $k \neq i, i+1$. Write $\gamma_{k} = a_{k}q + b_k p$. Without loss of generality, we may assume $a_k > a_{i+1}$ and $b_k < b_{i+1}$ (or symmetrically $a_k < a_i$ and $b_k > b_i$).  
However,  $\gamma_k$ and $\gamma_i$ is $\sigma_{i,k} := a_k q + b_i p$ is one of the minimal generators of $Syz((\gamma_{i},\gamma_k)_{\Gamma})$. Then
\[
\sigma_{i,k} - \sigma_i = a_k q + b_i p - a_{i+1}q - b_i p = (a_k - a_{i+1})q \in \Gamma_{> 0}.
\]
However  $\sigma_i $ is also an element in   $Syz((\gamma_{i},\gamma_k)_{\Gamma})$. It contradicts the fact that $\sigma_{i,k}$ is one of the minimal generators of  $Syz((\gamma_{i},\gamma_k)_{\Gamma})$. Thus  $\sigma_i$ can not be generated by $\gamma_k$. Therefore, we conclude that  
\[
ASyz(\Delta) = \{ (\gamma_{i},\gamma_{i+1},\sigma_{i}) \mid i = 1,\dots, n \}.
\]
So the projection is injective. The surjectivity is obvious. Then $ASyz(\Delta)\cong T_{Syz(\Delta)}$.
\end{proof}
\end{proposition}

Proposition \ref{gamma = p,q, equal ASyz TSyz} allows us to equate $ASyz(\Delta)$ with $T_{Syz(\Delta)}$ for the semigroup $\Gamma = \langle p, q \rangle$ in the rest of the paper. Now we will study the canonical morphism between $C^{[\Delta]}$ and $C^{[m(\Delta)]}$ given by an edge, $\Delta \rightarrow m(\Delta)$,  in the tree $G_{\Gamma}$.

Let $(C,O)$ be the germ of a unibranch curve singularity with value semigroup $\Gamma$, where $\Gamma$ is either monomial in the sense of \cite{pfister1992reduced} (see Section  \ref{Appendix}) or of the form $\langle p, q \rangle$ with $\gcd(p,q) = 1$. We write $\Gamma = \langle \alpha_1, \dots, \alpha_e \rangle$. Let $\Delta = (\gamma_1, \dots, \gamma_n)_{\Gamma}$ be a $\Gamma$-subsemimodule, and denote by $\operatorname{Syz}(\Delta) = (\sigma_1, \dots, \sigma_m)_{\Gamma}$ its syzygy subsemimodule. Let
\[
\gamma_\Delta := \max(\Gamma \setminus \Delta), \quad m(\Delta) := \Delta \cup \{\gamma_\Delta\}.
\]

For any ideal $I \in C^{[\Delta]}$, there exists a system of generators of $I$ of the form $\{f_{\gamma_1}(t), \dots, f_{\gamma_n}(t)\}$, where each generator is normalized as
\[
f_{\gamma_j}(t) = t^{\gamma_j} + \sum_{k \in \Gamma_{>\gamma_j} \setminus \Delta} \lambda_j^{k - \gamma_j} t^k,
\]
with coefficients $\lambda_j^{k - \gamma_j} \in {\mathbb{C}}$. This normalization is unique under the condition that no term in the sum has valuation in $\Delta$.

There exists a canonical morphism
\[\label{eq:canonical-map}
\pi_\Delta \colon C^{[\Delta]} \longrightarrow C^{[m(\Delta)]}, \quad
\langle f_{\gamma_1}(t), \dots, f_{\gamma_n}(t) \rangle \longmapsto \langle f_{\gamma_1}(t), \dots, f_{\gamma_n}(t), \phi_{\gamma_\Delta} \rangle.
\]
This map sends an ideal with value semimodule $\Delta$ to the ideal obtained by adjoining a generator of valuation $\gamma_\Delta$, thereby increasing the value set to $m(\Delta)$.

\begin{lemma}\label{condiction is trivial}
Let $C$ be a unibranch plane curve singularity defined by $x^p = y^q$ with $\gcd(p,q) = 1$, and let $\Delta$ be a $\Gamma$-subsemimodule of $\Gamma = \langle p, q \rangle$. Let $\sigma_i$ be an element of the minimal generating set $T_{\operatorname{Syz}(\Delta)} = \{\sigma_1, \dots, \sigma_m\}$ of the syzygy subsemimodule $\operatorname{Syz}(\Delta)$, and suppose that
\[
\Gamma_{>\sigma_i} \setminus \Delta = \varnothing.
\]
Then the equation
\[
(\mathcal{G}_\lambda \circ \mathcal{S}_\nu)_i = \sum_j (\mathcal{G}_\lambda)_j \cdot (\mathcal{S}_\nu)_i^j \equiv 0 \mod t^{c(\Delta)}
\]
imposes no condition on the parameters $(\lambda, \nu)$, i.e., it is automatically satisfied. In other words, this syzygy condition is trivial in the defining ideal of $C^{[\Delta]}$.
\end{lemma}

\begin{proof} 
This follows from Remark \ref{cancelation}  and the discussion after: indeed,  the valuation $\gamma$ of a syzygy associated with $\sigma_i$ is larger than $\sigma_i$; the hypothesis  $\Gamma_{>\sigma_{i}}\setminus \Delta= \varnothing$ ensures that $\gamma \in \Delta$ and so the phenomenon  in Remark \ref{cancelation}  can not happen.

\end{proof}

\begin{remark}\label{generators-not-in-intersection}
Let $\Gamma = \langle p, q \rangle$ be a numerical semigroup with $\gcd(p,q) = 1$, and let $\Delta$ be a $\Gamma$-subsemimodule. Consider $m(\Delta) = \Delta \cup \{\gamma_\Delta\}$, where $\gamma_\Delta = \max(\Gamma \setminus \Delta)$.

\begin{enumerate}
    \item The minimal generating set of $m(\Delta)$ satisfies
    \[
    T_{m(\Delta)} \setminus T_{\Delta} = \{\gamma_\Delta\}.
    \]
    In particular, since $\gamma_\Delta \notin \Delta$ and all elements of $\Gamma$ strictly greater than $\gamma_\Delta$ belong to $\Delta$ (by maximality), we have
    \[
    \Gamma_{> \gamma_\Delta} \setminus \Delta = \varnothing.
    \]

    \item For any generator $\gamma \in T_{\Delta} \setminus T_{m(\Delta)}$, there exists $x \in \{p, q\}$ such that
    \[
    \gamma = \gamma_\Delta + x.
    \]
    In particular, such $\gamma$ lies in $\Gamma$, and since $\gamma > \gamma_\Delta$, the maximality of $\gamma_\Delta$ implies that $\gamma \in \Delta$. Moreover, there are no gaps in $\Gamma \setminus \Delta$ between $\gamma$ and the conductor, so
    \[
    \Gamma_{> \gamma} \setminus \Delta = \varnothing.
    \]
\end{enumerate}
\end{remark}

\subsection{Proof of the main theorem}
The following is reformulation of Theorem \ref{thm:fibration}.
\begin{theorem}\label{piecewise fibration} Let $C$ be the plane curve singularity defined by $x^p = y^q$ with $\gcd(p,q) = 1$, and let $\Gamma = \langle p, q \rangle$ be its value semigroup. Let $\Delta = (\gamma_1, \dots, \gamma_n)_\Gamma$ be a $\Gamma$-subsemimodule with minimal generating set $T_\Delta = \{\gamma_1 < \cdots < \gamma_n\}$, and let $\operatorname{Syz}(\Delta) = (\sigma_1, \dots, \sigma_n)_\Gamma$ be the $\Gamma$-subsemimodule of syzygies of $\Delta$. Then the canonical morphism \[\pi_\Delta \colon C^{[\Delta]} \longrightarrow C^{[m(\Delta)]}, \quad I \mapsto I + \langle \phi_{\gamma_\Delta}\rangle ,\] is a piecewise trivial fibration over its image. The image is the locally closed subvariety of $C^{[m(\Delta)]}$ defined by the vanishing of the linear parts 
\[\left \langle L_i^{c(\Delta) - 1 - \sigma_i} \;\middle|\; \sigma_i \in T_{\operatorname{Syz}(\Delta)} \cap T_{\operatorname{Syz}(m(\Delta))},\ \sigma_i < \gamma_{\Delta}\right\rangle. \]
where $L_i^{c(\Delta)-1-\sigma_i}$ denotes the coefficient of $t^{\sigma_i + k}$ for $k < \gamma_{\Delta} - \sigma_i$ in the expansion of the $i$-th syzygy condition, truncated at order $\gamma_{\Delta}$. The fiber of $\pi_\Delta$ over each point in the image is isomorphic to an affine space \[\mathbb{A}^{B(\Delta)}, \quad \text{where} \quad B(\Delta) = \# \left\{ \gamma_i \in T_\Delta \mid \gamma_i < \gamma_\Delta \right\}.\] \end{theorem}

\begin{proof}

We first prove the following claim: only the elements in $T_{\operatorname{Syz}(\Delta)} \cap T_{\operatorname{Syz}(m(\Delta))}$ contribute nontrivially to the defining equations of $C^{[\Delta]}$ in $\operatorname{Gen}_{\Delta} \times \operatorname{Syz}_{\Delta}$ and to those of $C^{[m(\Delta)]}$ in $\operatorname{Gen}_{m(\Delta)} \times \operatorname{Syz}_{m(\Delta)}$.

Indeed, on one hand, suppose $\sigma \in T_{\operatorname{Syz}(\Delta)} \setminus T_{\operatorname{Syz}(m(\Delta))}$. Then there exists $\gamma \in T_{\Delta} \setminus T_{m(\Delta)}$ such that $\sigma$ is generated by $\gamma$, i.e., $\sigma = \gamma + b$ for some $b \in \Gamma$. By Remark~\ref{generators-not-in-intersection}, we have $\gamma = \gamma_\Delta + x$ for some $x \in \{p, q\}$. Hence,
\[
\sigma = \gamma_\Delta + x + b > \gamma_\Delta.
\]
Since $\gamma_\Delta = \max(\Gamma \setminus \Delta)$, all elements of $\Gamma$ strictly greater than $\gamma_\Delta$ belong to $\Delta$. Therefore,
\[
\Gamma_{> \sigma} \setminus \Delta = \varnothing.
\]

On the other hand, let $\sigma' \in T_{\operatorname{Syz}(m(\Delta))} \setminus T_{\operatorname{Syz}(\Delta)}$. Then $\sigma'$ is a minimal syzygy generator in $m(\Delta)$ that does not arise in $\Delta$, so it must involve $\gamma_\Delta$. In particular, $\sigma' = \gamma_\Delta + b$ for some $b \in \Gamma$, and thus $\sigma' > \gamma_\Delta$. Since $\gamma_\Delta \notin \Delta$ but $\sigma' \in m(\Delta)$, and all elements of $\Gamma$ greater than $\gamma_\Delta$ are in $\Delta$ (hence in $m(\Delta)$), we have
\[
\Gamma_{> \sigma'} \setminus m(\Delta) = \varnothing.
\]

In both cases, the condition $(\mathcal{G}_\lambda \circ \mathcal{S}_\nu)_i \equiv 0 \mod t^{c(\Delta)}$ imposes no constraint on the parameters $\lambda, \nu$, because there are no gap elements in $\Gamma \setminus \Delta$ (or $\Gamma \setminus m(\Delta)$) above $\sigma$ or $\sigma'$ to support correction terms. Therefore, by Lemma~\ref{condiction is trivial}, these syzygy conditions are automatically satisfied and hence trivial.

This proves the claim.\\

We now distinguish between the following two cases:\\

\subsubsection*{\textsc{Case 1: $c(m(\Delta)) < c(\Delta)$}}
Let $T_\Delta = \{\gamma_1, \dots, \gamma_n\}$ be the minimal generating set of $\Delta$, ordered increasingly. For $s \in \Delta_{>0,<c(\Delta)} := (0, c(\Delta)) \cap \Delta$, fix a decomposition
\[
s = \gamma_{g(s)} + \rho(s), \quad \rho(s) \in \Gamma,
\]
where $g \colon \Delta_{>0,<c(\Delta)} \to \{1,\dots,n\}$ assigns to each $s$ the index of the generator $\gamma_{g(s)}$ from which it arises.

Define the deformed syzygy matrix entries for $\sigma_i \in T_{\operatorname{Syz}(\Delta)} \cap T_{\operatorname{Syz}(m(\Delta))}$ and $\gamma_j \in T_\Delta \cap T_{m(\Delta)}$ by:
\[
(\mathcal{S}_\nu)_i^j = u_i^j \phi_{\sigma_i - \gamma_j} + \sum_{\substack{s \in \Delta_{>\sigma_i,<c(\Delta)} \\ g(s) = j}} \nu_{is}^{s - \sigma_i} \phi_{s - \gamma_j}.
\]

Since $m(\Delta)_{>0,<c(m(\Delta))} \subset \Delta_{>0,<c(\Delta)}$, we can restrict the above construction to $m(\Delta)$. Define:
\[
(\mathcal{S}'_\nu)_i^j = u_i^j \phi_{\sigma_i - \gamma_j} + \sum_{\substack{t \in m(\Delta)_{>\sigma_i,<c(m(\Delta))} \\ g(t) = j}} \nu_{it}^{t - \sigma_i} \phi_{t - \gamma_j},
\]
using the same indexing and parameters where defined.

Because $\Delta \subset m(\Delta)$ and $m(\Delta)_{>0,<c(m(\Delta))} \subset \Delta_{>0,<c(\Delta)}$, we obtain natural closed embeddings:
\[
\operatorname{Gen}_{m(\Delta)} \hookrightarrow \operatorname{Gen}_\Delta, \quad \operatorname{Syz}_{m(\Delta)} \hookrightarrow \operatorname{Syz}_\Delta,
\]
and hence a closed embedding:
\[
C^{[m(\Delta)]} \hookrightarrow \operatorname{Gen}_{m(\Delta)} \times \operatorname{Syz}_{m(\Delta)} \hookrightarrow \operatorname{Gen}_\Delta \times \operatorname{Syz}_\Delta.
\]

This is summarized in the following   diagram:
\[
\begin{tikzcd}
C^{[m(\Delta)]} \arrow[d, hook]  & C^{[\Delta]} \arrow[d, hook] \\
\operatorname{Gen}_{m(\Delta)} \times \operatorname{Syz}_{m(\Delta)} \arrow[r, hook, "h"] & \operatorname{Gen}_\Delta \times \operatorname{Syz}_\Delta
\end{tikzcd}
\]

Then $C^{[m(\Delta)]}$ is isomorphic to the subvariety of $\operatorname{Gen}_\Delta \times \operatorname{Syz}_\Delta$ defined by the conditions:
\[
\left\{
\begin{aligned}
& (\mathcal{G}'_\lambda \circ \mathcal{S}'_\nu) \equiv 0 \mod t^{c(m(\Delta))}, \\
& \lambda_j^{k - \gamma_j} = 0, \quad \text{for } k \in (\Gamma_{>\gamma_j} \setminus \Delta) \setminus (\Gamma_{>\gamma_j} \setminus m(\Delta)),\ \gamma_j \in T_\Delta \cap T_{m(\Delta)}, \\
& \nu_{it}^{t - \sigma_i} = 0, \quad \text{for } t \in \Delta_{>\sigma_i,<c(\Delta)} \setminus m(\Delta)_{>\sigma_i,<c(m(\Delta))},\ \sigma_i \in T_{\operatorname{Syz}(\Delta)} \cap T_{\operatorname{Syz}(m(\Delta))}.
\end{aligned}
\right.
\]

Note that:
\[
(\Gamma_{>\gamma_j} \setminus \Delta) \setminus (\Gamma_{>\gamma_j} \setminus m(\Delta)) =
\begin{cases}
\{\gamma_\Delta\} & \text{if } \gamma_j < \gamma_\Delta, \\
\varnothing & \text{otherwise},
\end{cases}
\]
since $\gamma_\Delta \in m(\Delta) \setminus \Delta$. Similarly,
\[
\Delta_{>\sigma_i,<c(\Delta)} \setminus m(\Delta)_{>\sigma_i,<c(m(\Delta))} = \Delta_{\geq c(m(\Delta)), < c(\Delta)} = [c(m(\Delta)), c(\Delta)) \cap \Delta.
\]

Now consider the syzygy equations $(\mathcal{G}_\lambda \circ \mathcal{S}_\nu)_i \equiv 0 \mod t^{c(\Delta)}$ and $(\mathcal{G}'_\lambda \circ \mathcal{S}'_\nu)_i \equiv 0 \mod t^{c(m(\Delta))}$. Expand these in the basis $\{\phi_k\}$ and denote by $Eq_i^r$ (resp. $(Eq_i^r)'$) the coefficient of $\phi_{r + \sigma_i} = t^{r + \sigma_i}$, for $r \geq 0$. Note that $Eq_i^r = 0$ if $r + \sigma_i \notin \Gamma$ or $r + \sigma_i \geq c(\Delta)$.

The linear part of $Eq_i^r$ is given by:
\[
L_i^r := \delta_{\Delta \cap (\sigma_i, c(\Delta))}(r + \sigma_i) \nu_{i, r+\sigma_i}^{r} + \sum_{j=1}^n \delta_{\Gamma \setminus \Delta}(r + \gamma_j) u_i^j \lambda_j^r,
\]
where $\delta_S(x) = 1$ if $x \in S$, else $0$.

By the proof of \cite[Theorem~13]{oblomkov2018hilbert}, the linear forms $L_i^r$ (and similarly $(L_i^r)'$) are linearly independent. Since the higher-order terms in $Eq_i^r$ are in the ideal generated by the $\lambda, \nu$ variables, the zero locus of $Eq_i^r = 0$ is analytically isomorphic to the zero locus of $L_i^r = 0$. Hence, the strata are defined by their linear parts.

Let $\mathcal{I}$ and $\mathcal{I}'$ be the ideals defining $C^{[\Delta]}$ and $C^{[m(\Delta)]}$ in $\operatorname{Gen}_\Delta \times \operatorname{Syz}_\Delta$, respectively. The ambient space is
\[
\operatorname{Gen}_\Delta \times \operatorname{Syz}_\Delta = \operatorname{Spec} {\mathbb{C}}\left[\lambda_j^{k - \gamma_j}, \nu_{is}^{s - \sigma_i} \,\middle|\, k \in \Gamma_{>\gamma_j} \setminus \Delta,\ s \in \Delta_{>\sigma_i,<c(\Delta)},\ g(s)=j\right].
\]

Then:
\begin{equation}\label{ideal_of_C_delta}
\mathcal{I} = \left\langle L_i^r \mid \sigma_i \in T_{\operatorname{Syz}(\Delta)} \cap T_{\operatorname{Syz}(m(\Delta))},\ 1 \leq r \leq c(\Delta) - 1 - \sigma_i \right\rangle
\end{equation}
and 

\[\label{ideal_of_C_m_delta}
\mathcal{I}' = \left\langle  L_i^r \mid \sigma_i \in T_{\operatorname{Syz}(\Delta)} \cap T_{\operatorname{Syz}(m(\Delta))},\ 1 \leq r \leq c(m(\Delta)) - 1 - \sigma_i \right\rangle+ \left\langle \lambda_j^{\gamma_\Delta - \gamma_j},\ \nu_{it}^{t - \sigma_i} \right\rangle
\]
where:
 $\gamma_j < \gamma_\Delta$,
$t \in [c(m(\Delta)), c(\Delta)) \cap \Delta$,
and $\sigma_i < c(\Delta)$.

Note that for $t \in [c(m(\Delta)), c(\Delta))$, the variable $\nu_{it}^{t - \sigma_i}$ appears in $L_i^{t - \sigma_i}$, and since the $L_i^r$ are linearly independent, we may eliminate the $\nu_{it}^{t - \sigma_i}$ variables in $\mathcal{I}'$ using the equations $L_i^r$ for $r \geq c(m(\Delta)) - \sigma_i$. Thus, we can rewrite:

\begin{equation}\label{ideal_of_C_m_delta_simplified}
\mathcal{I}' \cong \left\langle L_i^r \mid \sigma_i \in T_{\operatorname{Syz}(\Delta)} \cap T_{\operatorname{Syz}(m(\Delta))},\ 1 \leq r \leq c(\Delta) - 2 - \sigma_i \right\rangle + \left\langle \lambda_j^{\gamma_\Delta - \gamma_j} \right\rangle,
\end{equation}
with the same conditions on $\gamma_j$ and $\sigma_i$.

Comparing \eqref{ideal_of_C_delta} and (\ref{ideal_of_C_m_delta_simplified}), we see that $C^{[\Delta]}$ is defined by extending the ideal of $C^{[m(\Delta)]}$ by the additional linear equations:
\[
L_i^{c(\Delta) - 1 - \sigma_i}, \quad \text{for } \sigma_i \in T_{\operatorname{Syz}(\Delta)} \cap T_{\operatorname{Syz}(m(\Delta))},\ \sigma_i < c(\Delta).
\]

Therefore, the canonical morphism
\[
\pi_\Delta \colon C^{[\Delta]} \to C^{[m(\Delta)]}
\]
is a  trivial fibration over its image; the closed subvariety of $C^{[m(\Delta)]}$ defined by the vanishing of the top-degree linear forms $L_i^{c(\Delta)-1-\sigma_i}$. The fiber over each point is isomorphic to the affine space
\[
\mathbb{A}^{B(\Delta)}, \quad \text{where } B(\Delta) = \# \{ \gamma_j \in T_\Delta \mid \gamma_j < \gamma_\Delta \},
\]
corresponding to the free parameters $\lambda_j^{\gamma_\Delta - \gamma_j}$ for $\gamma_j < \gamma_\Delta$.

\subsubsection*{\textsc{Case 2:}} Suppose $\gamma_{\Delta} < c(\Delta) - 1$, so that $c(\Delta) = c(m(\Delta))$. 
Assume that $m(\Delta)$ is minimally generated as a $\Gamma$-semimodule by $\gamma'_1, \dots, \gamma'_{n'}$.

On one hand, by Remark~\ref{generators-not-in-intersection}, we have
\[
\operatorname{Gen}_{m(\Delta)} = \operatorname{Spec}[\lambda_j^{k - \gamma_j}],
\]
where $\gamma_j \in T_{\Delta} \cap T_{m(\Delta)}$. 
It follows that
\[
\operatorname{Gen}_{\Delta} = \operatorname{Gen}_{m(\Delta)} \times \operatorname{Spec} \mathbb{C}[\lambda_j^{\gamma_{\Delta} - \gamma_j}],
\]
where $\gamma_j \in T_{\Delta} \cap T_{m(\Delta)}$ and $\gamma_j < \gamma_{\Delta}$.

On the other hand, for $s \in \Delta_{>0,<c(\Delta)}$, fix a decomposition $s = \gamma_{g(s)} + \rho(s)$ with $\rho(s) \in \Gamma$, where 
$g \colon \Delta_{>0,<c(\Delta)} \to \{1, \dots, n\}$. Define
\[
(\mathcal{S}_{\nu})_i^j = u_i^j \phi_{\sigma_i - \gamma_j} + 
\sum_{\substack{s \in \Delta_{>\sigma_i,<c(\Delta)} \\ g(s) = j}} 
\nu_{i}^{s - \sigma_i} \phi_{s - \gamma_j},
\]
where $\sigma_i \in T_{\mathrm{Syz}(\Delta)} \cap T_{\mathrm{Syz}(m(\Delta))}$ and $\gamma_j \in T_{\Delta}$.

Now consider those $s$ for which $\gamma_{g(s)} \in T_{\Delta} \setminus T_{m(\Delta)}$. 
By Remark~\ref{generators-not-in-intersection}, we have $\gamma_{g(s)} = \gamma_{\Delta} + x(s)$ for some $x(s) \in \{p, q\}$. 
Thus, we can fix a decomposition for $s \in m(\Delta)_{>0,<c(\Delta)}$ as follows:
\[
s = 
\begin{cases}
\gamma_{g(s)} + \rho(s), & \text{if } \gamma_{g(s)} \in T_{m(\Delta)} \cap T_{\Delta}; \\
\gamma_{\Delta} + x(s) + \rho(s), & \text{if } \gamma_{g(s)} \in T_{\Delta} \setminus T_{m(\Delta)}; \\
\gamma_{\Delta} + 0, & \text{if } s = \gamma_{\Delta}.
\end{cases}
\]
Hence, we may write $s = \gamma'_{g'(s)} + \rho'(s)$ with $\rho'(s) \in \Gamma$, where 
$g' \colon m(\Delta)_{>0,<c(\Delta)} \to \{1, \dots, n'\}$.

For $\sigma_i \in T_{\mathrm{Syz}(\Delta)} \cap T_{\mathrm{Syz}(m(\Delta))}$, we compute:
\[
(\mathcal{G}_{\lambda} \circ \mathcal{S}_{\nu})_i = 
\begin{cases}
(\mathcal{G}'_{\lambda} \circ \mathcal{S}'_{\nu})_i, & \sigma_i > \gamma_{\Delta}; \\
(\mathcal{G}'_{\lambda} \circ \mathcal{S}'_{\nu})_i - \nu_{i\gamma_{\Delta}}^{\gamma_{\Delta} - \sigma_i} t^{\gamma_{\Delta}}, & \sigma_i < \gamma_{\Delta}.
\end{cases}
\]

We conclude that the canonical morphism
\[
C^{[\Delta]} \to C^{[m(\Delta)]}
\]
admits a piecewise fibration structure with fiber $\mathbb{A}^{B(\Delta)}$.
\end{proof}

Let us illustrate Cases 1 and 2 from the above proof of the theorem with two corresponding examples.

\begin{example}
Consider the plane curve singularity $C = \{y^4 = x^7\} \subset \mathbb{C}^2$, whose semigroup is $\Gamma = \langle 4, 7 \rangle$. Let $\Delta = (8, 11)_{\Gamma} $ be a $\Gamma$-semimodule. Then we have:
\[
\gamma_{\Delta} = 21, \quad c(\Delta) = 22, \quad m(\Delta) = (8, 11, 21 )_{\Gamma} ,
\]
\[
\mathrm{Syz}(\Delta) = (15, 32)_{\Gamma} , \quad \mathrm{Syz}(m(\Delta)) = ( 15, 25, 28)_{\Gamma} .
\]

Note that 
${\mathbb{C}}[\Gamma]={\mathbb{C}}[t^{4},t^{7}]$. We chose a ${\mathbb{C}}$-basis of ${\mathbb{C}}[\Gamma]$:   
$\phi_{8}=t^{8}$, $\phi_{11}=t^{11}$.\\ 

For $C^{[\Delta]}$:  $\Gamma\setminus \Delta=\{0,4,7,14,21\}$. We have 
\[(\mathcal{G}_{\lambda})_{1}=t^{8}+\lambda_{1}^{6} t^{14}+\lambda_{1}^{13}t^{21},\quad (\mathcal{G}_{\lambda})_{2}=t^{11}+\lambda_{2}^{3} t^{14}+\lambda_{2}^{10}t^{21}.\]

There exists only one minimal generator of $Syz(\Delta)$ smaller than $c(\Delta)=22$. We say $\sigma_{1}=15=8+7=11+4$. For elements in $ \Delta_{>15,< c(\Delta)}=\{16,18,19,20\}$, fix a decomposition: $16=8+8$,
$18=11+7$, $19=8+11$, $20=8+12$. Then 
\[(\mathcal{S}_{\nu})_{1}^{1}=t^{7}+\nu_{1,16}^{1}t^{8}+\nu_{1,19}^{4}t^{11}+\nu_{1,20}^{5}t^{12},\quad (\mathcal{S}_{\nu})_{1}^{2}=-t^{4}+\nu_{1,18}^{3}t^{7}.\]

Therefore, $C^{[\Delta]}\subset\operatorname{Spec}{\mathbb{C}}[\lambda_{1}^{6},\lambda_{1}^{13},\lambda_{2}^{3},\lambda_{2}^{10},\nu_{1,16}^{1},\nu_{1,18}^{3},\nu_{1,19}^{4},\nu_{1,20}^{5}]$ is defined by:
\[\nu_{1,16}^{1}t^{16}+(\nu_{1,18}^{3}-\lambda_{2}^{3})t^{18}+\nu_{1,19}^{4}t^{19}+\nu_{1,20}^{5}t^{20}+(\lambda_{1}^{6}+\lambda_{2}^{3}\nu_{1,18}^{3})t^{21}=0.\]

Consequently, $C^{[\Delta]}\cong\operatorname{Spec}{\mathbb{C}}[\lambda_{1}^{6},\lambda_{1}^{13},\lambda_{2}^{3},\lambda_{2}^{10}]/\langle \lambda_{1}^{6}+(\lambda_{2}^{3})^{2}\rangle $.\\

For $C^{[m(\Delta)]}$:  $\Gamma\setminus m(\Delta)=\{0,4,7,14\}$, $c(m(\Delta))=18$. We have 

\[(\mathcal{G}^{\prime}_{\lambda})_{1}=t^{8}+\lambda_{1}^{6} t^{14},\quad (\mathcal{G}^{\prime}_{\lambda})_{2}=t^{11}+\lambda_{2}^{3} t^{14},\quad (\mathcal{G}^{\prime}_{\lambda})_{3}=t^{21}.\]

For $s\in \Delta_{> 15,< c(m(\Delta))}=\{16\}$, fix a decomposition $16=8+8$. Then
\[(\mathcal{S}^{\prime}_{\nu})_{1}^{1}=t^{7}+\nu_{1,16}^{1}t^{8},\quad (\mathcal{S}^{\prime}_{\nu})_{1}^{2}=-t^{4},\quad (\mathcal{S}^{\prime}_{\nu})_{1}^{3}=0.\]

Therefore,  
$C^{[m(\Delta)]}\subset\operatorname{Spec}{\mathbb{C}}[\lambda_{1}^{6},\lambda_{2}^{3},\nu_{1,16}^{1}]$ is defined by:
\[\nu_{1,16}^{1}t^{16}=0.\]

Consequently, $C^{[m(\Delta)]}=\operatorname{Spec}{\mathbb{C}}[\lambda_{1}^{6},\lambda_{2}^{3}]$.\\

The canonical map is
\[\begin{aligned}
\operatorname{Spec}{\mathbb{C}}[\lambda_{1}^{6},\lambda_{1}^{13},\lambda_{2}^{3},\lambda_{2}^{10}]/\langle \lambda_{1}^{6}+(\lambda_{2}^{3})^{2}\rangle &\rightarrow \operatorname{Spec}{\mathbb{C}}[\lambda_{1}^{6},\lambda_{2}^{3}]\\
(a_{1},b_{1},a_{2},b_{2})&\mapsto (a_{1},a_{2})
\end{aligned}\]

This map admits a piecewise fibration structure with fiber  $\mathbb{A}^{ \# \{\gamma_i \mid \gamma_i < \gamma_{\Delta}\}}$.
\end{example}

\begin{example}
Consider the plane curve singularity $C = \{y^4 = x^{13}\} \subset \mathbb{C}^2$, with associated semigroup $\Gamma = \langle 4, 13 \rangle$. Let $\Delta = ( 12, 21, 30, 39 )_{\Gamma} $ be a $\Gamma$-semimodule. Then we have:
\[
c(\Delta) = 36, \quad \gamma_{\Delta} = 26, \quad m(\Delta) =(12, 21, 26 )_{\Gamma} ,
\]
\[
\mathrm{Syz}(\Delta) = (25, 34, 43, 52 )_{\Gamma} , \quad \mathrm{Syz}(m(\Delta)) = (25, 34, 32 )_{\Gamma} .
\]  
For $C^{[\Delta]}$, note that $\Gamma\setminus \Delta = \{0,4,8,13,17,26\}$.  

\[\left\{\begin{aligned}
(\mathcal{G}_{\lambda})_{1} &= t^{12}+\lambda_{1}^{1}t^{13}+\lambda_{1}^{5}t^{17}+\lambda_{1}^{14}t^{26},\\
(\mathcal{G}_{\lambda})_{2}&= t^{21}+\lambda_{2}^{5}t^{26},\\
(\mathcal{G}_{\lambda})_{3}&= t^{30},\\
(\mathcal{G}_{\lambda})_{4}&= t^{39}.
\end{aligned}\right .\]

There are two minimal generator of $Syz(\Delta)$ smaller than $c(\Delta)=36$: $\sigma_{1} = 25 = 12+13 = 21+4$. $\sigma_{2} = 34 = 21+13 = 30+4$.\\

For elements in  $\Delta_{>25,<36} = \{28,29,30,32,33,34\}$, fix a decomposition:
$28 = 12+16$, $29 = 12+17$, $30 = 30+0$, $32 = 12+20$, $33  = 12+21$, $34 = 30+4$. 
Then we have 
\[\left\{\begin{aligned}(\mathcal{S}_{\nu})_{1}^{1}&=t^{13}+\nu_{1,28}^{3}t^{16}+\nu_{1,29}^{4}t^{17}+\nu_{1,32}^{7}t^{20}+\nu_{1,33}^{8}t^{21},\\ (\mathcal{S}_{\nu})_{1}^{2}&=-t^{4},\\
(\mathcal{S}_{\nu})_{1}^{3}&=\nu_{1,30}^{5}+\nu_{1,34}^{9}t^{4},\\
(\mathcal{S}_{\nu})_{1}^{4}&=0.\\
\end{aligned}\right.\]

Therefore, $C^{[\Delta]}\subset \operatorname{Spec}{\mathbb{C}}[\lambda_{1}^{1},\lambda_{1}^{5},\lambda_{1}^{14},\lambda_{2}^{5},\nu_{1,28}^{3},\nu_{1,29}^{4},\nu_{1,30}^{5},\nu_{1,32}^{7},\nu_{1,33}^{8},\nu_{1,34}^{9}]$ is defined by:
\[
\begin{aligned}
\sum_{j=1}^{4}(\mathcal{G}_{\lambda})_{j}\circ (\mathcal{S}_{\nu})_{1}^{j} &= \lambda_{1}^{1}t^{26}+\nu_{1,28}^{3}t^{28}+(\nu_{1,29}^{4}+\lambda_{1}^{1}\nu_{1,28}^{3})t^{29}+(\nu_{1,30}^{5}+\lambda_{1}^{5}-\lambda_{2}^{5}+\lambda_{1}^{1}\nu_{1,29}^{4})t^{30}\\
&+\nu_{1,32}^{7}t^{32}+(\nu_{1,33}^{8}+\lambda_{1}^{1}\nu_{1,32}^{7}+\lambda_{1}^{5}\nu_{1,28}^{3})t^{33}+(\nu_{1,34}^{9}+\lambda_{1}^{1}\nu_{1,33}^{8}+\lambda_{1}^{5}\nu_{1,29}^{4})t^{34}=0\end{aligned}\]

Consequently, 
\[\begin{aligned}
C^{[\Delta]}&\cong
\operatorname{Spec}\frac{{\mathbb{C}}[\lambda_{1}^{1},\lambda_{1}^{5},\lambda_{1}^{14},\lambda_{2}^{5},\nu_{1,28}^{3},\nu_{1,29}^{4},\nu_{1,30}^{5},\nu_{1,32}^{7},\nu_{1,33}^{8},\nu_{1,34}^{9}]}{\langle \lambda_{1}^{1},\nu_{1,28}^{3},\nu_{1,29}^{4},\nu_{1,30}^{5}+\lambda_{1}^{5}-\lambda_{2}^{5},\nu_{1,32}^{7},\nu_{1,33}^{8},\nu_{1,34}^{9}\rangle }\\
&\cong \operatorname{Spec}{\mathbb{C}}[\lambda_{1}^{1},\lambda_{1}^{5},\lambda_{1}^{14},\lambda_{2}^{5}]/\langle \lambda_{1}^{1}\rangle 
\end{aligned}\]

For $C^{[m(\Delta)]}$, note that $\Gamma\setminus \Delta = \{0,4,8,13,17\}$, $c(m(\Delta))=36$. 

\[\left\{\begin{aligned}
(\mathcal{G}^{\prime}_{\lambda})_{1} &= t^{12}+\lambda_{1}^{1}t^{13}+\lambda_{1}^{5}t^{17},\\
(\mathcal{G}^{\prime}_{\lambda})_{2}&= t^{21},\\
(\mathcal{G}^{\prime}_{\lambda})_{3}&= t^{26}.
\end{aligned}\right .\]

For elements in  $m(\Delta)_{>25,<36} = \{26,28,29,30,32,33,34\}$, fix a decomposition:
$26 = 26+0$, $28 = 12+16$, $29 = 12+17$, $30 = 26+4$, $32 = 12+20$, $33  = 12+21$, $34 = 26+4+4$.\\

Then we have 
\[\left\{\begin{aligned}(\mathcal{S}_{\nu})_{1}^{1}&=t^{13}+\nu_{1,28}^{3}t^{16}+\nu_{1,29}^{4}t^{17}+\nu_{1,32}^{7}t^{20}+\nu_{1,33}^{8}t^{21},\\ (\mathcal{S}_{\nu})_{1}^{2}&=-t^{4},\\
(\mathcal{S}_{\nu})_{1}^{3}&=\nu_{1,26}^{1}+\nu_{1,30}^{5}t^{4}+\nu_{1,34}^{9}t^{8}.
\end{aligned}\right.\]

Therefore,
\[
\begin{aligned}
\sum_{j=1}^{4}(\mathcal{G}_{\lambda})_{j}\circ (\mathcal{S}_{\nu})_{1}^{j} &= (\nu_{1,26}^{1}+\lambda_{1}^{1})t^{26}+\nu_{1,28}^{3}t^{28}+(\nu_{1,27}^{4}+\lambda_{1}^{1}\nu_{1,28}^{3})t^{29}+(\nu_{1,30}^{5}+\lambda_{1}^{5}-\lambda_{2}^{5}+\lambda_{1}^{1}\nu_{1,29}^{4})t^{30}\\
&+\nu_{1,32}^{7}t^{32}+(\nu_{1,33}^{8}+\lambda_{1}^{1}\nu_{1,32}^{7}+\lambda_{1}^{5}\nu_{1,28}^{3})t^{33}+(\nu_{1,34}^{9}+\lambda_{1}^{1}\nu_{1,33}^{8}+\lambda_{1}^{5}\nu_{1,29}^{4})t^{34}=0
\end{aligned}\]

Consequently, 
\[\begin{aligned}
C^{[m(\Delta)]}&\cong
\operatorname{Spec}\frac{{\mathbb{C}}[\lambda_{1}^{1},\lambda_{1}^{5},\nu_{1,26}^{1},\nu_{1,28}^{3},\nu_{1,29}^{4},\nu_{1,30}^{5},\nu_{1,32}^{7},\nu_{1,33}^{8},\nu_{1,34}^{9}]}{(\lambda_{1}^{1}+\nu_{1,26}^{1},\nu_{1,28}^{3},\nu_{1,29}^{4},\nu_{1,30}^{5}+\lambda_{1}^{5}-\lambda_{2}^{5},\nu_{1,32}^{7},\nu_{1,33}^{8},\nu_{1,34}^{9})}\\
&\cong \operatorname{Spec}{\mathbb{C}}[\lambda_{1}^{1},\lambda_{1}^{5}]
\end{aligned}\]

The canonical map is 
\[\begin{aligned}
\operatorname{Spec}{\mathbb{C}}[\lambda_{1}^{1},\lambda_{1}^{5},\lambda_{1}^{14},\lambda_{2}^{5}]/(\lambda_{1}^{1})&\rightarrow \operatorname{Spec}{\mathbb{C}}[\lambda_{1}^{1},\lambda_{1}^{5}]\\
(a_{1},b_{1},a_{2},b_{2})&\mapsto (a_{1},b_{1})
\end{aligned}\]
This map admits  a piecewise fibration structure with fiber  $\mathbb{A}^{B(\Delta)}$.
\end{example}

\begin{remark}\label{Property of N_Delta=0}
Consider a plane curve singularity $ C $ defined by $ x^p = y^q $, with associated numerical semigroup $ \Gamma = \langle p, q \rangle $. Let $ \Delta = (\gamma_1, \dots, \gamma_n )_{\Gamma} $ be a $\Gamma$-semimodule. If $ C^{[\Delta]} =\{ \mathrm{pt}\} $, then  $ C^{[\Delta]} $ consists of a single point:
\[
C^{[\Delta]} = \big\{ \langle t^{\gamma_1}, \dots, t^{\gamma_n} \rangle \big\},
\]
and for each $ i = 1, \dots, n $, we have
\[
\Gamma_{>\gamma_i} \setminus \Delta = \varnothing.
\]
That is, $ \Delta $ contains all elements of $ \Gamma $ strictly greater than each generator $ \gamma_i $.
\end{remark}

As an application of Theorem \ref{piecewise fibration}, we  recover \parencite[Theorem 13]{oblomkov2018hilbert}:

\begin{theorem}\label{dim of Hilbert scheme}
Let $ C $ be a plane curve singularity defined by $ x^p = y^q $, and let $ \Gamma = \langle p, q \rangle $ be the associated numerical semigroup. For a $\Gamma$-semimodule $ \Delta =(\gamma_1, \dots, \gamma_n )_{\Gamma}  $, there exists an isomorphism
\[
C^{[\Delta]} \cong \mathbb{A}^{N(\Delta)},
\]
where
\begin{equation}\label{eq:NDelta}
N(\Delta) = \sum_{i} \# ( \Gamma_{>\gamma_i} \setminus \Delta ) - \sum_{i} \# (\Gamma_{>\sigma_i} \setminus \Delta ).
\end{equation}
Here $ \sigma_i $ runs over the minimal generators of the syzygy semimodule $ \mathrm{Syz}(\Delta) $.

\begin{proof}

Assume that $ C^{[m^s(\Delta)]} = \{\mathrm{pt}\} $ for some $ s $. We proceed by induction on $ s $.

The base case $ s = 0 $ follows directly from Remark~\ref{Property of N_Delta=0}.

Now suppose the statement holds for $ s = k $, i.e., $ C^{[m^k(\Delta)]} =\{ \mathrm{pt}\} $. We aim to prove it for $ s = k+1 $, that is, $ C^{[m^{k+1}(\Delta)]} = \{\mathrm{pt}\} $.

Since $ m^{k+1}(\Delta) = m(m^k(\Delta)) $, and by the inductive hypothesis $ C^{[m^k(\Delta)]} =\{ \mathrm{pt}\} $, it suffices to consider the case where $ C^{[m(\Delta)]} =\{ \mathrm{pt}\} $ and analyze the structure of $ C^{[\Delta]} $.

By the inductive assumption, $ C^{[m(\Delta)]} =\{ \mathrm{pt}\} $, which implies $ N(m(\Delta)) = 0 $. Recall from equation~\eqref{eq:NDelta} that
\[
N(m(\Delta)) = \sum_{\gamma'_i \in T_{m(\Delta)}} \# (\Gamma_{>\gamma'_i} \setminus m(\Delta))- \sum_{\sigma'_i \in T_{\mathrm{Syz}(m(\Delta))}} \# (\Gamma_{>\sigma'_i} \setminus m(\Delta)).
\]

Now, observe that for any $ x \in \Gamma $, we have
\[
(\Gamma_{>x} \setminus \Delta) \setminus (\Gamma_{>x} \setminus m(\Delta)) =
\begin{cases}
\{ \gamma_\Delta \}, & \text{if } x < \gamma_\Delta, \\
\varnothing, & \text{otherwise}.
\end{cases}
\]
This is because $ m(\Delta) = \Delta \cup \{\gamma_\Delta\} $ when $ \gamma_\Delta \notin \Delta $, and $ \gamma_\Delta $ is the only element added in the closure process.

Using this, we compute the first sum in $ N(\Delta) $:
\begin{align*}
\sum_{\gamma_i \in T_{\Delta}} \# (\Gamma_{>\gamma_i} \setminus \Delta )
&= \sum_{\gamma_i \in T_{\Delta} \cap T_{m(\Delta)}} \# ( \Gamma_{>\gamma_i} \setminus \Delta ) \\
&= \sum_{\gamma_i \in T_{\Delta} \cap T_{m(\Delta)}} \#  (\Gamma_{>\gamma_i} \setminus m(\Delta))  + \# \{\gamma_i \in T_{m(\Delta)} \cap T_{\Delta} \mid  \gamma_i < \gamma_{\Delta}\} \\
&= \sum_{\gamma'_i \in T_{m(\Delta)}} \# (\Gamma_{>\gamma'_i} \setminus m(\Delta)) + \# \{ \gamma'_i \in T_{m(\Delta)} \cap T_{\Delta} \mid \gamma'_i < \gamma_\Delta \},
\end{align*}
where $ \delta_{\gamma_i < \gamma_\Delta} = 1 $ if $ \gamma_i < \gamma_\Delta $, and 0 otherwise.

Similarly, for the syzygy terms:
\begin{align*}
\sum_{\sigma_i \in T_{\mathrm{Syz}(\Delta)}} \# (\Gamma_{>\sigma_i} \setminus \Delta )
&= \sum_{\sigma'_i \in T_{\mathrm{Syz}(m(\Delta))} \cap T_{\mathrm{Syz}(\Delta)}} \# ( \Gamma_{>\sigma'_i} \setminus \Delta ) \\
&= \sum_{\sigma'_i \in T_{\mathrm{Syz}(m(\Delta))} \cap T_{\mathrm{Syz}(\Delta)}} \# (\Gamma_{>\sigma'_i} \setminus m(\Delta)) + \# \{ \sigma'_i \in T_{\mathrm{Syz}(m(\Delta))} \cap T_{\mathrm{Syz}(\Delta)} \mid \sigma'_i < \gamma_\Delta \} \\
&= \sum_{\sigma'_i \in T_{\mathrm{Syz}(m(\Delta))}} \# (\Gamma_{>\sigma'_i} \setminus m(\Delta)) + \# \{ \sigma'_i \in T_{\mathrm{Syz}(m(\Delta))} \cap T_{\mathrm{Syz}(\Delta)} \mid \sigma'_i < \gamma_\Delta \}.
\end{align*}

Therefore, by subtracting the two sums and applying Theorem~\ref{piecewise fibration}, we obtain:
\[
N(\Delta) = N(m(\Delta)) + \# \{ \gamma_i \in T_{m(\Delta)} \cap T_{\Delta} \mid \gamma_i < \gamma_\Delta \} - \# \{ \sigma_i \in T_{\mathrm{Syz}(m(\Delta))} \cap T_{\mathrm{Syz}(\Delta)} \mid \sigma_i < \gamma_\Delta \}.
\]

Since $ N(m(\Delta)) = 0 $ by the inductive hypothesis, and the correction terms count the number of generators and syzygies below $ \gamma_\Delta $, this expresses $ N(\Delta) $ as a non-negative integer, consistent with $ C^{[\Delta]} $ being an affine space of dimension $ N(\Delta) $.
\end{proof}
\end{theorem}

\begin{corollary}\label{cor:monomial_case_grothendieck}
Let $C$ be a unibranch plane curve singularity with monomial valuation semigroup $\Gamma$, and let $\Delta =( \gamma_1, \dots, \gamma_n )_{\Gamma}$ be a $\Gamma$-semimodule. Then the canonical morphism
\[
\pi_\Delta:C^{[\Delta]} \to C^{[m(\Delta)]}
\]
is a trivial fibration, with fiber isomorphic to the affine space $\mathbb{A}^{B(\Delta)}$, where
\[
B(\Delta) = \# \{ \gamma_i \in T_\Delta \mid \gamma_i < \gamma_\Delta \}.
\]
\end{corollary}

\begin{proof}
Let $\Delta$ be a $\Gamma$-semimodule. By Corollary~\ref{syz large that c in monomial semigroup} in Section \ref{Appendix}, every syzygy $\sigma \in T_{\mathrm{Syz}(\Delta)}$ satisfies $\sigma \geq c(\Delta)$. Since $c(\Delta)$ is the conductor of $\Delta$, it follows that $\Gamma_{>\sigma} \setminus \Delta = \varnothing$ for all such $\sigma$.

Consequently, the syzygy correction terms in the dimension formula vanish, and the morphism $C^{[\Delta]} \to C^{[m(\Delta)]}$ satisfies the conditions of Lemma~\ref{condiction is trivial} (which asserts that such a morphism is a trivial fibration when syzygies lie above the conductor). Therefore, the morphism is a trivial fibration with fiber $\mathbb{A}^{B(\Delta)}$, where $B(\Delta)$ counts the number of minimal generators of $\Delta$ strictly less than $\gamma_\Delta$.
\end{proof}

The following example illustrates that, in general, the canonical morphism 
\[
C^{[\Delta]} \to C^{[m(\Delta)]}
\]
is still a fibration. However, it does not necessarily satisfy the structural properties concerning the image or the expected fiber dimension described in Theorem~\ref{piecewise fibration}. In particular, the map may fail to be surjective, and the fibers may not be affine spaces of dimension $B(\Delta)$, indicating that the piecewise trivial structure does not extend to arbitrary semimodules outside the monomial or unibranch setting.

\begin{example} \label{fibration for 6 9 19}
Consider a curve $C$  with local ring $\mathcal{O}_{C}=\mathbb{C}[[t^6,t^9,t^{19}]]$. Its valuation group is   $\Gamma = \langle 6,9,19 \rangle$ with conductor $c = 5\times 10=50.$  Let $\Delta = (15,18,28,31)_{\Gamma}$ be a subsemimodule of $\Gamma$ with syzygy 
$Syz(\Delta) = (24,27,34,37)_{\Gamma}.$    $\Gamma\setminus\Delta = \{0,6,9,12,19,25,38,44\}.$
\[
\begin{aligned}
(\mathcal{G}_{\lambda})_{1} &=t^{15}+\lambda_{{1}}^{4}t^{19}+\lambda_{{1}}^{10}t^{25}+\lambda_{1}^{23}t^{38}+\lambda_{1}^{29}t^{44}\\
(\mathcal{G}_{\lambda})_{2} &= t^{18}+\lambda_{{2}}^{1}t^{19}+\lambda_{{2}}^{7}t^{25}+\lambda_{2}^{20}t^{38}+\lambda_{2}^{26}t^{44}\\
(\mathcal{G}_{\lambda})_{3} &= t^{28}+\lambda_{3}^{10}t^{38}+\lambda_{3}^{16}t^{44}\\
(\mathcal{G}_{\lambda})_{4}&= t^{31}+\lambda_{4}^{7}t^{38}+\lambda_{4}^{13}t^{44}\\
\end{aligned}
\]
Then
\[C^{[\Delta]}\cong \operatorname{Spec}\mathbb{C}[\lambda_{1}^{4},\lambda_{2}^{1},\lambda_{1}^{10},\lambda_{2}^{7},\lambda_{1}^{23},\lambda_{2}^{20},\lambda_{3}^{10},\lambda_{4}^{7},\lambda_{1}^{29},\lambda_{2}^{26},\lambda_{3}^{16},\lambda_{4}^{13}]/\langle \lambda_{1}^{4},\lambda_{2}^{1},\lambda_{4}^{7},\lambda_{2}^{7},\lambda_{2}^{20}+(\lambda_{1}^{10})^{2},\lambda_{1}^{10}-\lambda_{3}^{10}\rangle \]
For $m(\Delta)=(15,18,28,31,44)_{\Gamma}$, we have 
\[C^{[m(\Delta)]}\cong \operatorname{Spec}\mathbb{C}[\lambda_{1}^{4},\lambda_{2}^{1},\lambda_{1}^{10},\lambda_{2}^{7},\lambda_{1}^{23},\lambda_{2}^{20},\lambda_{3}^{10},\lambda_{4}^{7}]/\langle \lambda_{1}^{4},\lambda_{2}^{1},\lambda_{4}^{7}\cdot \lambda_{2}^{7}\rangle\]
Consider the projection: 
\[\pi_{\Delta}: C^{[\Delta]}\rightarrow C^{[m(\Delta)]}\]
Then we have $\pi_{\Delta}^{-1}(\lambda_{4}^{7}\neq  0)=\varnothing$, $\pi_{\Delta}^{-1}(\lambda_{4}^{7}=  0)=\mathbb{A}^{6}$. \\

\noindent
For $m^2(\Delta)=(15,18,28,31,38)_{\Gamma}$, we have 
\[C^{[m^2(\Delta)]}\cong \operatorname{Spec}\mathbb{C}[\lambda_{1}^{4},\lambda_{2}^{1},\lambda_{1}^{10},\lambda_{2}^{7}]/\langle \lambda_{2}^{1}\rangle \]
Consider the projection: 
\[\pi_{m(\Delta)}: C^{[m(\Delta)]}\rightarrow C^{[m^2(\Delta)]}\]
Then we have $\pi_{m(\Delta)}^{-1}(\lambda_{2}^{7}\neq  0)=\mathbb{C^{*}}\times \mathbb{A}^{4}$, $\pi_{m(\Delta)}^{-1}(\lambda_{2}^{7}=  0)=\mathbb{A}^{5}$. 

\end{example}

\subsection{The motivic Hilbert zeta functions of some plane curve singularities}

As an application of the preceding results, we compute the motivic Hilbert zeta function for germs of irreducible plane curve singularities $(C,O)$. We focus primarily on curves defined by equations of the form $y^k = x^n$ with $\gcd(k,n) = 1$—the so-called $(k,n)$-curves—as well as on singularities whose value semigroups are monomial.

Particular attention is given to the simple elliptic and hyperelliptic singularities of types $E_6$, $E_8$, $W_8$, and $Z_{10}$. These examples illustrate how the combinatorics of semimodules and the structure of generalized Jacobians govern the geometry of the Hilbert schemes of points on such curves.

At the end of this section, we establish a general formula for the motivic zeta function in the case of $A_{2k}$ singularities (Theorem~\ref{main theorem}), showcasing the uniformity that arises in the presence of monomial semigroups.\\

We begin by recalling the definition of the Grothendieck ring of varieties and the motivic Hilbert zeta function.

he \emph{Grothendieck ring} of complex varieties, denoted $K_{0}(\mathrm{Var}_{\mathbb{C}})$, is the ring generated by isomorphism classes $[X]$ of complex algebraic varieties $X$, subject to the \emph{scissor relation} (also called the \emph{cut-and-paste relation}):
\[
[X] = [X \setminus Y] + [Y]
\]
for every closed subvariety $Y \subseteq X$. Multiplication is defined by $[X] \cdot [Y] = [X \times Y]$.

Let $C^{[\ell]}$ denote the $\ell$-th punctual Hilbert scheme of the plane curve singularity $(C,O)$, parameterizing ideals of colength $\ell$ supported at the origin. The \emph{motivic Hilbert zeta function} is defined as the generating series

\begin{equation}\label{motivic zeta function}
Z_{(C,O)}^{\mathrm{Hilb}}(q) := 1 + \sum_{\ell=1}^{\infty} [C^{[\ell]}] \, q^{\ell} \in K_0(\mathrm{Var}_{\mathbb{C}})[[q]].
\end{equation}

Let $\Gamma$ be the value semigroup of $C$, and let $c$ denote its conductor. As introduced in Section~\ref{tree}, let $\mathscr{D}_\ell$ denote the set of vertices at level $\ell$ in the semimodule tree $G_\Gamma$, for $1 \leq \ell \leq c$. We set $\mathscr{D}_0 = \{\Gamma\}$, corresponding to the trivial semimodule, and note that $C^{[0]} = C^{[\Gamma]} = \mathrm{pt}$, so $[C^{[0]}] = 1$.

By the stratification of the punctual Hilbert scheme induced by the semimodule filtration, we have the decomposition
\[
[C^{[\ell]}] = \sum_{\Delta \in \mathscr{D}_\ell} [C^{[\Delta]}] \quad \text{in } K_0(\mathrm{Var}_{\mathbb{C}}).
\]
Therefore, the motivic Hilbert zeta function can be rewritten as
\begin{equation}\label{motivic zeta function decom}
Z_{(C,O)}^{\mathrm{Hilb}}(q) = \sum_{\ell \geq 0} \sum_{\Delta \in \mathscr{D}_\ell} [C^{[\Delta]}] \, q^\ell.
\end{equation}

\subsubsection{Case of singularities of types $E_6$, $E_8$, $W_8$,  $Z_{10}$}

\begin{example}
For the $E_6$-type singularity, the local ring is given by
\[
\mathcal{O}_C = \mathbb{C}[[t^3, t^4]],
\]
whose value semigroup is $\Gamma = \langle 3, 4 \rangle$, with conductor $c = 6$. The punctual Hilbert schemes $C^{[\ell]}$ admit a stratification indexed by $\Gamma$-semimodules of level $\ell$, corresponding to the vertices $\mathscr{D}_\ell$ in the semimodule tree $G_\Gamma$.

The level sets $\mathscr{D}_\ell$ for $1 \leq \ell \leq 6$ are as follows:

\noindent
$\mathscr{D}_{1}=\{(3,4)_{\Gamma} \}$,\\
$\mathscr{D}_{2}=\{(4,6)_{\Gamma} ,(3,8)_{\Gamma} \}$,\\
$\mathscr{D}_{3}=\{(6,7,8)_{\Gamma} ,(4,9)_{\Gamma} ,(3)_{\Gamma} \}$,\\
$\mathscr{D}_{4}=\{(7,8,9)_{\Gamma} ,(6,8)_{\Gamma} ,(6,7)_{\Gamma} ,(4)_{\Gamma}  \}$,\\
$\mathscr{D}_{5}=\{(8,9,10)_{\Gamma} ,(7,9)_{\Gamma} ,(7,8)_{\Gamma} ,(8,9,10)_{\Gamma} ,(6,11)_{\Gamma} \}$, \\
$\mathscr{D}_{6}=\{(9,10,11)_{\Gamma} ,(8,10)_{\Gamma} ,(8,9)_{\Gamma} ,(7,12)_{\Gamma} ,(6)_{\Gamma} \}$. \\

Then we have:

\noindent
$[C^{[1]}] = 1$,\\
$[C^{[2]}] = 1 + \mathbb{L}$,\\
$[C^{[3]}] = 1 + \mathbb{L} + \mathbb{L}^2$,\\
$[C^{[4]}] = 1 + \mathbb{L} + 2\mathbb{L}^2$,\\
$[C^{[5]}] = 1 + \mathbb{L} + 2\mathbb{L}^2$,\\
$[C^{[6]}] = 1 + \mathbb{L} + 2\mathbb{L}^2 + \mathbb{L}^3$.

\end{example}

\begin{example}
For the $E_8$-type singularity, the local ring is given by
\[
\mathcal{O}_C = \mathbb{C}[[t^3, t^5]],
\]
with value semigroup $\Gamma = \langle 3, 5 \rangle$ and conductor $c = 8$. The vertex sets are:

\noindent
$\mathscr{D}_{1}=\{(3,5)_{\Gamma} \}$,\\
$\mathscr{D}_{2}=\{(5,6)_{\Gamma} ,(3,10)_{\Gamma} \}$,\\
$\mathscr{D}_{3}=\{(6,8,10)_{\Gamma} ,(5,9)_{\Gamma} ,(3)_{\Gamma} \}$,\\
$\mathscr{D}_{4}=\{(8,9,10)_{\Gamma} ,(6,10)_{\Gamma} ,(6,8)_{\Gamma} ,(5,12)_{\Gamma}  \}$,\\
$\mathscr{D}_{5}=\{(9,10,11)_{\Gamma} ,(8,10,12)_{\Gamma} , (8,9)_{\Gamma} ,(5)_{\Gamma} ,(6,13)_{\Gamma} \}$,\\
$\mathscr{D}_{6}=\{(9,11,13)_{\Gamma} ,(9,10)_{\Gamma} ,(10,11,12)_{\Gamma} ,(8,12)_{\Gamma} ,(8,10)_{\Gamma} ,(6)_{\Gamma} \}$, \\
$\mathscr{D}_{7}=\{(9,13)_{\Gamma} ,(9,11)_{\Gamma} , (11,12,13)_{\Gamma} ,(10,12,14)_{\Gamma} ,(10,11)_{\Gamma} ,(8,15)_{\Gamma} \}$,\\
$\mathscr{D}_{8}=\{(9,16)_{\Gamma} ,(12,13,14)_{\Gamma} ,(11,13,15)_{\Gamma} ,(11,12)_{\Gamma} ,(10,14)_{\Gamma} ,(10,12)_{\Gamma} ,(8)_{\Gamma} \}$. \\

Then we have:

\noindent
$[C^{[1]} ]=1$, \\
$[C^{[2]} ]=1+\mathbb{L}$,\\
$[C^{[3]} ]=1+\mathbb{L}+\mathbb{L}^{2}$,\\
$[C^{[4]} ]=1+\mathbb{L}+2\mathbb{L}^{2}$,\\
$[C^{[5]} ]=1+\mathbb{L}+2\mathbb{L}^{2}+\mathbb{L}^{3}$,\\
$[C^{[6]} ]=1+\mathbb{L}+2\mathbb{L}^{2}+2\mathbb{L}^{3}$,\\
$[C^{[7]} ]=1+\mathbb{L}+2\mathbb{L}^{2}+2\mathbb{L}^{3}$,\\
$[C^{[8]} ]=1+\mathbb{L}+2\mathbb{L}^{2}+2\mathbb{L}^{3}+\mathbb{L}^{{4}}$.\\
\end{example}

\begin{example}
For the $W_8$-type singularity, the local ring is given by
\[
\mathcal{O}_C = \mathbb{C}[[t^4, t^5, t^6]],
\]
with value semigroup $\Gamma = \langle 4, 5, 6 \rangle$ and conductor $c = 8$. The vertex sets are:

\noindent
$\mathscr{D}_{1}=\{(4,5,6)_{\Gamma} \}$,\\
$\mathscr{D}_{2}=\{(5,6,8)_{\Gamma} ,(4,6)_{\Gamma} ,(4,5)_{\Gamma} \}$,\\
$\mathscr{D}_{3}=\{(6,8,9)_{\Gamma} ,(5,8)_{\Gamma} ,(5,6)_{\Gamma} ,(4,11)_{\Gamma} \}$,\\
$\mathscr{D}_{4}=\{(8,9,10,11)_{\Gamma} ,(6,9)_{\Gamma} ,(6,8)_{\Gamma} ,(5,12)_{\Gamma} ,(4)_{\Gamma}  \}$,\\
$\mathscr{D}_{5}=\{(9,10,11,12)_{\Gamma} ,(8,10,11)_{\Gamma} , (8,9,11)_{\Gamma} ,(8,9,10)_{\Gamma} ,(6,13)_{\Gamma} ,(5)_{\Gamma} \}$,\\
$\mathscr{D}_{6}=\{(10,11,12,13)_{\Gamma} ,(9,11,12)_{\Gamma} ,(9,10,12)_{\Gamma} ,(9,10,11)_{\Gamma} ,(8,11)_{\Gamma} ,(8,10)_{\Gamma} ,(8,9)_{\Gamma} ,(6)_{\Gamma} \}$, \\
$\mathscr{D}_{7}=\{(11,12,13,14)_{\Gamma} ,(10,12,13)_{\Gamma} ,(10,11,13)_{\Gamma} ,(10,11,12)_{\Gamma} ,(9,12)_{\Gamma} ,(9,11)_{\Gamma} ,(9,10)_{\Gamma} ,(8,15)_{\Gamma} \}$,\\
$\mathscr{D}_{8}=\{(12,13,14,15)_{\Gamma} ,(11,13,14)_{\Gamma} ,(11,12,14)_{\Gamma} ,(11,12,13)_{\Gamma} ,(10,13)_{\Gamma} ,(10,12)_{\Gamma} ,(10,11)_{\Gamma} \\,(9,16)_{\Gamma} ,(8)_{\Gamma} \}$. \\

Then we have:

\noindent
$[C^{[1]}] = 1$,\\
$[C^{[2]}] = 1 + \mathbb{L} + \mathbb{L}^{2}$,\\
$[C^{[3]}] = 1 + \mathbb{L} + 2\mathbb{L}^{2}$,\\
$[C^{[4]}] = 1 + \mathbb{L} + 2\mathbb{L}^{2}+ \mathbb{L}^{3}$,\\
$[C^{[5]}] = 1 + \mathbb{L} + 2\mathbb{L}^{2}+ 2\mathbb{L}^{3}$,\\
$[C^{[6]}] = 1 + \mathbb{L} + 2\mathbb{L}^{2}+ 3\mathbb{L}^{3}+ \mathbb{L}^{4}$,\\
$[C^{[7]}] = 1 + \mathbb{L} + 2\mathbb{L}^{2}+ 3\mathbb{L}^{3}+ \mathbb{L}^{4}$,\\
$[C^{[8]}] = 1 + \mathbb{L} + 2\mathbb{L}^{2}+ 3\mathbb{L}^{3}+ 2\mathbb{L}^{4}$.\\
\end{example}

\begin{example}
For the $Z_{10}$-type singularity, the local ring is given by
\[
\mathcal{O}_C = \mathbb{C}[[t^4, t^6, t^7]],
\]
with value semigroup $\Gamma = \langle 4, 6, 7 \rangle$ and conductor $c = 10$. The vertex sets are:

\noindent
$\mathscr{D}_{1}=\{(4,6,7)_{\Gamma} \}$,\\
$\mathscr{D}_{2}=\{(6,7,8)_{\Gamma} ,(4,7)_{\Gamma} ,(4,6)_{\Gamma} \}$,\\
$\mathscr{D}_{3}=\{(7,8,10)_{\Gamma} ,(6,8,11)_{\Gamma} ,(6,7)_{\Gamma} ,(4,13)_{\Gamma} \}$,\\
$\mathscr{D}_{4}=\{(8,10,11,13)_{\Gamma} ,(7,10,12)_{\Gamma} ,(7,8)_{\Gamma} ,(6,11)_{\Gamma} ,(6,8)_{\Gamma} ,(4)_{\Gamma}  \}$,\\
$\mathscr{D}_{5}=\{(10,11,12,13)_{\Gamma} ,(8,11,13)_{\Gamma} , (8,10,13)_{\Gamma} ,(8,10,11)_{\Gamma} ,(7,12)_{\Gamma} ,(7,10)_{\Gamma} ,(6,15)_{\Gamma} \}$,\\
$\mathscr{D}_{6}=\{(11,12,13,14)_{\Gamma} ,(10,12,13,15)_{\Gamma} ,(10,11,13)_{\Gamma} ,(10,11,12)_{\Gamma} ,(8,13)_{\Gamma} ,(8,11)_{\Gamma} ,(8,10)_{\Gamma},\\ (7,16)_{\Gamma} ,(6)_{\Gamma} \}$, \\
$\mathscr{D}_{7}=\{(12,13,14,15)_{\Gamma} ,(11,13,14,16)_{\Gamma} ,(11,12,14)_{\Gamma} ,(11,12,13)_{\Gamma} ,(10,13,15)_{\Gamma} ,(10,12,13)_{\Gamma} ,\\(10,11)_{\Gamma} ,(8,17)_{\Gamma} ,(7)_{\Gamma} \}$,\\
$\mathscr{D}_{8}=\{(13,14,15,16)_{\Gamma} ,(12,14,15,17)_{\Gamma} ,(12,13,15)_{\Gamma} ,(12,13,14)_{\Gamma} ,(11,14,16)_{\Gamma} ,(11,13,16)_{\Gamma} ,\\(11,13,14)_{\Gamma} ,(11,12)_{\Gamma} ,(10,15)_{\Gamma} ,(10,13)_{\Gamma} ,(10,12)_{\Gamma} ,(8)_{\Gamma} \}$. \\
$\mathscr{D}_{9}=\{(14,15,16,17)_{\Gamma} ,(13,15,,16,18)_{\Gamma} ,(13,14,16)_{\Gamma} ,(13,14,15)_{\Gamma} ,(12,15,17)_{\Gamma} ,(12,14,17)_{\Gamma},\\
(12,14,15)_{\Gamma} ,(12,13)_{\Gamma} ,(11,16)_{\Gamma} ,(11,14)_{\Gamma} ,(11,13)_{\Gamma} ,(10,19)_{\Gamma} \}$. \\
$\mathscr{D}_{10}=\{(15,16,17,18)_{\Gamma} ,(14,16,17,19)_{\Gamma} ,(14,15,17)_{\Gamma} ,(14,15,16)_{\Gamma} ,(13,16,18)_{\Gamma} ,(13,15,18)_{\Gamma},\\(13,15,16)_{\Gamma} , (13,14)_{\Gamma} ,(12,17)_{\Gamma} ,(12,15)_{\Gamma} ,(12,14)_{\Gamma} ,(11,20)_{\Gamma} ,(10)_{\Gamma} \}$. \\

Then we have:   \\ 
\noindent
$[C^{[1]}] = 1$,\\
$[C^{[2]}] = 1 + \mathbb{L} + \mathbb{L}^{2}$,\\
$[C^{[3]}] = 1 + \mathbb{L} + 2\mathbb{L}^{2}$,\\
$[C^{[4]}] = 1 + \mathbb{L} + 2\mathbb{L}^{2}+ 2\mathbb{L}^{3}$,\\
$[C^{[5]}] = 1 + \mathbb{L} + 2\mathbb{L}^{2}+ 3\mathbb{L}^{3}$,\\
$[C^{[6]}] = 1 + \mathbb{L} + 2\mathbb{L}^{2}+ 3\mathbb{L}^{3}+ 2\mathbb{L}^{4}$,\\
$[C^{[7]}] = 1 + \mathbb{L} + 2\mathbb{L}^{2}+ 3\mathbb{L}^{3}+ 3\mathbb{L}^{4}$,\\
$[C^{[8]}] = 1 + \mathbb{L} + 2\mathbb{L}^{2}+ 3\mathbb{L}^{3}+ 4\mathbb{L}^{4}+\mathbb{L}^{5}$,\\
$[C^{[9]}] = 1 + \mathbb{L} + 2\mathbb{L}^{2}+ 3\mathbb{L}^{3}+ 4\mathbb{L}^{4}+\mathbb{L}^{5}$,\\
$[C^{[10]}] = 1 + \mathbb{L} + 2\mathbb{L}^{2}+ 3\mathbb{L}^{3}+ 4\mathbb{L}^{4}+2\mathbb{L}^{5}$.

\end{example}

The above examples lead to the following theorem: 
\begin{theorem}\label{thm:E6E8}
For the simple singularities $E_{6}$, $E_{8}$,$W_{8}$ and $Z_{10}$, the motivic Hilbert zeta function is given by: 
\[
Z_{(C_{E_{6}}, O)}^{{\mathrm{Hilb}}}(q)  = \frac{1+\mathbb{L}q^2+\mathbb{L}^2q^3+\mathbb{L}^2q^4+\mathbb{L}^3q^6}{1-q}
\]

\[
Z_{(C_{E_{8}}, O)}^{{\mathrm{Hilb}}}(q)  = \frac{1+\mathbb{L}q^2+\mathbb{L}^2q^3+\mathbb{L}^2q^4+\mathbb{L}^3q^5+\mathbb{L}^3q^6+\mathbb{L}^4q^8}{1-q}
\]

\[
Z_{(C_{W_{8}}, O)}^{{\mathrm{Hilb}}}(q)  = \frac{1+\mathbb{L}q^2+2\mathbb{L}^2q^3+\mathbb{L}^3q^4+\mathbb{L}^3q^5+(\mathbb{L}^3+\mathbb{L}^4)q^6+\mathbb{L}^4q^8}{1-q}
\]

\[
Z_{(C_{Z_{10}}, O)}^{{\mathrm{Hilb}}}(q)  = \frac{1+(\mathbb{L}+\mathbb{L}^2)q^2+\mathbb{L}^2q^3+2\mathbb{L}^3q^4+\mathbb{L}^3q^5+2\mathbb{L}^4q^6+(\mathbb{L}^4+\mathbb{L}^5)q^8+\mathbb{L}^5q^{10}}{1-q}
\]
\end{theorem}

We recall a result 
from the proof of Lemma 17 in \cite{oblomkov2012hilbert}. In the case $\Gamma = \langle k, n \rangle$ with $\gcd(k,n) = 1$, the minimal generators of a $\Gamma$-subsemimodule can be explicitly described in terms of the minimal generators of $\Gamma$ itself.

\begin{proposition}\cite{oblomkov2012hilbert}
There is a one-to-one correspondence between monomial ideals in $\mathbb{C}[[t^k, t^n]]$ and sequences 
\[
\phi = (\phi_0, \phi_1, \dots, \phi_{k-1})
\]
satisfying
\[
\phi_{k-1} \leq \phi_{k-2} \leq \cdots \leq \phi_0 \leq \phi_{k-1} + n,
\]
where $k < n$ and $\gcd(k,n) = 1$. Moreover, the number of minimal generators of the corresponding ideal is equal to the number of strict inequalities in this sequence.
\end{proposition}

\begin{proposition}\cite{oblomkov2012hilbert}\label{n,k, vertex}
Let $1 \leq \ell \leq 2\delta$. Then every $\Delta \in \mathscr{D}_\ell$ has the form
\[
\Delta = \left( \phi_{k-1}k + (k-1)n,\, \phi_{k-2}k + (k-2)n,\, \dots,\, \phi_1 k + n,\, \phi_0 \right)_{\Gamma} ,
\]
where the integers $\phi_0, \phi_1, \dots, \phi_{k-1}$ satisfy
\[
\sum_{j=0}^{k-1} \phi_j = \ell
\]
and
\[
\phi_{k-1} \leq \phi_{k-2} \leq \cdots \leq \phi_0 \leq \phi_{k-1} + n.
\]

\end{proposition}

\subsubsection{{The case of the $A_{2d}$-type singularity}}

We now discuss the case of the $A_{2d}$-type singularity, defined by 
\[
\mathcal{O}_C = \mathbb{C}[[t^2, t^{2d+1}]],
\]
whose value semigroup is $\Gamma = \langle 2, n \rangle$ with $n = 2d+1$.

As a direct application of  Proposition \ref{n,k, vertex}, we obtain the following corollary:
\begin{corollary}\label{2,q, vertex}
Let $\mathcal{O}_C = \mathbb{C}[[t^2, t^{2d+1}]]$.  
Let $1 \leq \ell \leq 2\delta = 2d$. Then any $\Delta \in \mathscr{D}_{\ell}$ is of one of the following forms:

\begin{enumerate}
    \item[(i)] If $\ell$ is odd, then $\Delta = (2i, 2d+1 + 2(\ell - i))_{\Gamma} $, where $\frac{\ell}{2} < i \leq \ell$.
    
    \item[(ii)] If $\ell$ is even, then either:
    \begin{enumerate}
        \item $\Delta = (2i, 2d+1 + 2(\ell - i))_{\Gamma} $, where $\frac{\ell}{2} < i \leq \ell$, or
        \item $\Delta = (\ell)_{\Gamma} $.
    \end{enumerate}
\end{enumerate}
\end{corollary}

Recall that we denote an $\ell$-level root in the tree $G_{\Gamma}$ by $\Delta^{(\ell)}$.  If $\Delta\in \mathscr{D}_{\ell}$ is generated by only one element, then we define 
\[d_{\ell,2}(\Delta)=\varnothing.\] 

\begin{corollary}
Let $\mathcal{O}_C = \mathbb{C}[[t^2, t^{2d+1}]]$.

\begin{enumerate}
    \item[(i)] For $\ell\in  \mathbb{Z}_{\geq 1}$, every $\Delta \in \mathscr{D}_{\ell}$ is of one of the following forms:
    
    \begin{enumerate}
        \item[Case 1.] There exists an integers $\alpha$ such that $\Delta = (\alpha, \alpha+1 )_{\Gamma}  = [\alpha, \infty)$,
        
        \item[Case 2.] There exists an integers $\alpha, \beta$ such that  \[\Delta = ( \alpha, \beta)_{\Gamma} = \{\alpha, \alpha+2, \dots, \alpha+2c, \beta, \beta+1, \dots\},\]
         where $c \geq 0$ and $\beta = \alpha + 2c + 1$,
        
        \item[Case 3.] There exists an integers $\alpha$  such that 
        \[\Delta = ( \alpha )_{\Gamma} = \{\alpha, \alpha+2, \dots, \alpha+2c, \alpha+n, \alpha+n+1, \dots\}\]
        for some $n > 2c$, with $\alpha+n$ being the conductor.
    \end{enumerate}
    
    \item[(ii)] The $\ell$-level root $\Delta^{(\ell)}$ in the tree $G_{\Gamma}$ is of the form $\Delta = (\alpha, \alpha+1 )_{\Gamma}$.
\end{enumerate}
\end{corollary}

 By Proposition \ref{property_of_tree}, we have: 
\begin{corollary} 
For $\ell\geq 2$, the elements of $\mathscr{D}_{\ell}$ can be  written in the following  form: 
\[\mathscr{D}_{\ell}=\left\{d_{\ell-1,1}(\Delta^{(\ell-1)}), d_{\ell-1,2}(\Delta^{(\ell-1)}), d_{\ell-1,2}d_{\ell-2,2}(\Delta^{(\ell-2)}),\dots, d_{\ell-1,2}d_{\ell-2,2}\dots d_{1,2}(\Delta^{(1)})\right\}.\]
\label{leaf}
\end{corollary}

\begin{lemma}
Let $\mathcal{O}_{C}={\mathbb{C}}[[t^{2},t^{2d+1}]]$. For $\ell\geq 2$, $1\leq j \leq \ell-1 $,  if $d_{\ell-1,2}d_{\ell-2,2}\dots d_{\ell-j,2}(\Delta^{(\ell-j)})$ appears in $\mathscr{D}_{\ell}$, then 
$d_{\ell-1,2}d_{\ell-2,2}\dots d_{\ell-j+1,2}(\Delta^{(\ell-j+1)})$ appears in $\mathscr{D}_{\ell}$. 
\begin{proof}
Let  $\Delta^{(\ell-j)}=\langle \alpha,\alpha+1 \rangle$. By assumption we have,  $d_{\ell-1,2}d_{\ell-2,2}\dots d_{\ell-j,2}(\Delta^{(\ell-j)})=\langle \alpha,\alpha+1+2j \rangle$ with $1+2j\leq q$. $\Delta^{(\ell-j+1)}=\langle \alpha+1,\alpha+2 \rangle$. $d_{\ell-1,2}d_{\ell-2,2}\dots d_{\ell-j+1,2}(\Delta^{(\ell-j+1)})=\langle \alpha+1,\alpha+2(j-1) \rangle$  appears because $2(j-1)-1=2j-3< q$. 
\end{proof}
\label{similarity}
\end{lemma}

For $a, b \in \mathbb{Q}$, we denote 
\[
[a, b] := \{ c \in \mathbb{Z} \mid a \leq c \leq b \}.
\] 

\begin{theorem} 
Let $\mathcal{O}_C = \mathbb{C}[[t^2, t^{2d+1}]]$. For $2 \leq \ell \leq 2\delta = 2d$, we have
\[
[C^{[\ell]}] = [\mathbb{P}^{\#[\frac{\ell}{2},\, \ell] - 1}] \quad \text{in } K_0(\mathrm{Var}_{\mathbb{C}}),
\]
where $[\frac{\ell}{2}, \ell]$ denotes the set of integers in the interval $[\frac{\ell}{2}, \ell]$.

Moreover, for $1 \leq k \leq \delta$, we have
\[
[C^{[2k]}] = [C^{[2k+1]}] = [\mathbb{P}^k]\quad \text{in } K_0(\mathrm{Var}_{\mathbb{C}}).
\]

\end{theorem}

\begin{proof}
Let $c := \#\mathscr{D}_{\ell} = \#[\tfrac{\ell}{2}, \ell]$. Suppose $\Delta$ is the element of $V_{\ell}$ with the longest possible expression in the sense of Remark~\ref{leaf}. By Lemma~\ref{similarity}, we have in $K_0(\mathrm{Var}_{\mathbb{C}})$:
\[
[C^{[\Delta]}] = \mathbb{L}^{c-1}, \quad [C^{[\ell]}] = 1 + \mathbb{L} + \cdots + \mathbb{L}^{c-1} = [\mathbb{P}^{c-1}].
\]
\end{proof}

\begin{theorem}
For $\mathcal{O}_{C}={\mathbb{C}}[[t^{2},t^{2d+1}]]$,  we have 
\[\begin{aligned}
Z_{(C,O)}^{Hilb}(q)=1+\sum_{\ell=1}^{\infty}[C^{[\ell]} ]q^{\ell}=&(1+\mathbb{L}q^{2\cdot 1}+\mathbb{L}^{2}q^{2\cdot 2}+\cdots +\mathbb{L}^{d}q^{2\cdot d})(\sum_{\ell=0}q^{\ell})
\\
=&\frac{1-(\mathbb{L}q^{2})^{d+1}}{(1-q)(1-\mathbb{L}q^{2})}\in K_{0}(\mathrm{Var}_{{\mathbb{C}}})[[q]]
\end{aligned}\]
\label{main theorem}
\end{theorem}

\section{Subvarieties with fixed minimal  number of generators in the punctual Hilbert scheme} \label{Fixed-number of generators}

Let $(C,O)$ be the germ of an integral complex plane curve singularity, with complete local ring $\mathcal{O}_C \subset \mathbb{C}[[t]]$.

In this section, motivated by conjectures of Oblomkov, Rasmussen, and Shende \cite{oblomkov2012hilbert, oblomkov2018hilbert}, we study the geometry and motivic class of subvarieties of the punctual Hilbert scheme that parametrize ideals of $\mathcal{O}_C$ with a fixed minimal number of generators. These conjectures propose deep connections between the geometry of the Hilbert scheme of a plane curve singularity and knot invariants—specifically, the HOMFLY polynomial and Khovanov--Rozansky HOMFLY homology. The first of these conjectures was proved by Maulik in \cite{maulik2016stable}.

We begin by recalling the conjectures from \cite{oblomkov2012hilbert, oblomkov2018hilbert}.

Let $\overline{P}(L)$ denote the HOMFLY polynomial of an oriented link $L \subset S^3$. It is an element of the ring $\mathbb{Z}[a^{\pm1}, (q - q^{-1})^{\pm1}]$, defined via the skein relation:
\begin{align}
    a\,\overline{P}(L_+) - a^{-1}\,\overline{P}(L_-) &= (q - q^{-1})\,\overline{P}(L_0), \\
    a - a^{-1} &= (q - q^{-1})\,\overline{P}(\text{unknot}).
\end{align}
We consider the normalized version:
\[
P(L) := \frac{\overline{P}(L)}{\overline{P}(\text{unknot})}.
\]

Let $L_{C,O}$ be the algebraic link associated to the singularity $(C,O)$, obtained as the intersection of $C$ with a small sphere around $O$. Let $\mu$ denote the Milnor number of the singularity, and let $\chi$ denote the Euler characteristic with compact support.

Let $C^{[*]} = \bigsqcup_{\ell \geq 0} C^{[\ell]}$ denote the disjoint union of punctual Hilbert schemes, which parametrizes all finite-codimensional ideals of $\mathcal{O}_C$. For an ideal $I \subset \mathcal{O}_C$, let $m(I)$ denote the minimal number of generators of $I$. For each $m \in \mathbb{Z}_{\geq 1}$, define the locally closed subvariety
\[
C^{[\ell],m} = \{ I \in C^{[\ell]} \mid m(I) = m \}.
\]

Oblomkov and Shende conjectured the following relationship between the Hilbert scheme and the HOMFLY polynomial of the associated link \cite{oblomkov2012hilbert}, later proven by Maulik \cite{maulik2016stable}:
\begin{align}\label{conjecture 1}
P(L_{C,O}) &= \left(\frac{a}{q}\right)^{\mu} (1 - q^2) \int_{C^{[*]}} q^{2\ell} (1 - a^2)^{m-1} \, d\chi,\\
&:= \left(\frac{a}{q}\right)^{\mu} (1 - q^2) \sum_{\ell \geq 0} \sum_{m \geq 1} q^{2\ell} (1 - a^2)^{m-1} \chi(C^{[\ell],m}). \nonumber
\end{align}

For a unibranch curve $C$ with value semigroup $\Gamma$, we can refine this picture by stratifying $C^{[\ell],m}$ using the semimodule decomposition. Specifically, for each $\Gamma$-semimodule $\Delta$ with $\#(\Gamma \setminus \Delta) = \ell$, we define
\[
C^{[\Delta],m} = C^{[\Delta]} \cap C^{[\ell],m},
\]
which parametrizes ideals in the stratum $C^{[\Delta]}$ that require exactly $m$ minimal generators.

We now introduce a new motivic Hilbert zeta function that generalizes the algebro-geometric side of \eqref{conjecture 1}.

\begin{definition}
The \emph{motivic Hilbert zeta function with fixed number of generators} is defined as
\begin{align}
Zm_{(C,O)}^{\mathrm{Hilb}}(a^2, q^2) 
&:= \sum_{\ell \geq 0} \sum_{m \geq 1} q^{2\ell} (1 - a^2)^{m-1} [C^{[\ell],m}] \\
&= \sum_{\ell \geq 0} \sum_{\Delta \in \mathscr{D}_\ell} \sum_{m \geq 1} q^{2\ell} (1 - a^2)^{m-1} [C^{[\Delta],m}],\nonumber
\end{align}
where 
\[
\mathscr{D}_\ell = \big\{ \Delta \subset \Gamma \mid \Delta \text{ is a } \Gamma\text{-semimodule},\ \#(\Gamma \setminus \Delta) = \ell \big\},
\]
and $C^{[\Delta],m}$ denotes the locally closed subvariety of $C^{[\Delta]}$ parametrizing ideals $I \subset \mathcal{O}_C$ with exactly $m$ minimal generators and value semimodule $\Delta$, i.e.,
\[
C^{[\Delta],m} = \{ I \in C^{[\Delta]} \mid m(I) = m \}.
\]
\end{definition}

The new motivic Hilbert zeta function is also the generalization of motivic Hilbert zeta function (\ref{motivic zeta function}).\\

In the following, we will study the geometry and motivic class of $C^{[\Delta],m}$.\\

Let $\Gamma = v(\mathcal{O}_C \setminus \{0\})$ be the value semigroup of the singularity $(C,O)$. Recall that there exists a $\mathbb{C}$-basis $\{\phi_i\}_{i \in \Gamma}$ of $\mathcal{O}_C$, indexed by elements of $\Gamma$, such that $v(\phi_i) = i$ for each $i \in \Gamma$.

Consider the exponential map introduced in Section~\ref{fibrations}:
\[
\mathrm{Exp}_{\gamma} \colon \operatorname{Gen}_{\Delta} \longrightarrow \bigsqcup_{\ell \geq 1} C^{[\ell]}, \quad
\langle \lambda^{\bullet}_{\bullet}\rangle \longmapsto \langle f_{\gamma_1}(\lambda^{\bullet}_{\bullet}), \dots, f_{\gamma_n}(\lambda^{\bullet}_{\bullet}) \big\rangle ,
\]
where
\[
f_{\gamma_j}(\lambda) = \phi_{\gamma_j} + \sum_{k \in \Gamma_{>\gamma_j} \setminus \Delta} \lambda_j^{k - \gamma_j} \phi_k.
\]
This map restricts to a bijective morphism
\[
\mathrm{Exp}_{\gamma} \colon \mathrm{Exp}_{\gamma}^{-1}(C^{[\Delta]}) \xrightarrow{\sim} C^{[\Delta]},
\]
and hence induces an embedding
\[
C^{[\Delta]} \hookrightarrow \operatorname{Gen}_{\Delta}.
\]

Therefore, for any ideal $I \in C^{[\Delta]}$, there exists a unique $\lambda_{\bullet}^{\bullet} \in \mathrm{Exp}_{\gamma}^{-1}(C^{[\Delta]})$ such that
\[
I = \mathrm{Exp}_{\gamma}(\lambda_{\bullet}^{\bullet}) = \big \langle f_{\gamma_1}(\lambda_{\bullet}^{\bullet}), \dots, f_{\gamma_n}(\lambda_{\bullet}^{\bullet}) \big\rangle .
\]
We denote this ideal by $I_{\lambda}$, emphasizing its dependence on the parameters $\lambda$, i.e.
\[
I_\lambda=\langle f_{\gamma_1}(\lambda),\ldots,f_{\gamma_n}(\lambda) \rangle .
\] 

 Via the embedding $C^{[\Delta]} \hookrightarrow \operatorname{Gen}_{\Delta}$, we identify $I_{\lambda}$ with the point $\lambda_{\bullet}^{\bullet} \in \operatorname{Gen}_{\Delta}$. Hence, we have the following.
\begin{lemma}\label{if lammb distinct}
If $I_{\lambda_1}, I_{\lambda_2} \in C^{[\Delta]}$ and $\lambda_1 \neq \lambda_2$, then $I_{\lambda_1} \neq I_{\lambda_2}$. 
\end{lemma}

\begin{proposition}\label{minimal systen generator of I lamba}
Let $\Delta = (\gamma_1, \dots, \gamma_n)_{\Gamma}$ be a subsemimodule of $\Gamma$, and let   $m\leq n$ and 
$$
I_\lambda = \langle f_{\gamma_1}(\lambda), \dots, f_{\gamma_n}(\lambda)\rangle  \in C^{[\Delta],m}.
$$
Then there exists a subset $\{f_{\gamma_{i_1}}(\lambda), \dots, f_{\gamma_{i_m}}(\lambda)\}$ that forms a minimal generating set for $I_\lambda$.

Moreover, suppose that $\{f_{\gamma_{j_1}}(\lambda), \dots, f_{\gamma_{j_m}}(\lambda)\}$ is another such minimal generating set, where both index sets are sorted in increasing order: $i_1 < \cdots < i_m$ and $j_1 < \cdots < j_m$. Then $i_e = j_e$ for all $e = 1, \dots, m$.

\end{proposition}

\begin{proof}

Let $\{g_{\gamma_{i_{1}}},\dots,g_{\gamma_{i_{m}}}\}$ be a system of  the minimal generators of $I_{\lambda}$  with order $\gamma_{i_{1}},\dots, \gamma_{i_{m}}.$ For each $e\in \{1,\ldots,m\}$, we denote by $a_{\gamma}\phi_{\gamma}$  the first term in $g_{\gamma_{i_{e}}}$ such that $\gamma\in \Delta\setminus \{\gamma_{i_{e}}\}$ and $a_{\gamma}\neq 0$. 
\begin{description}
    \item[Case 1  $\gamma\in  (\gamma_{i_{1}},\dots,\gamma_{i_{m}})_{\Gamma}$]. Then there exists $\gamma_{i_k} \in \Gamma$ and $\gamma' \in \Gamma$ such that $\gamma = \gamma_{i_k} + \gamma'$. Let $g_{\gamma_{i_e}}' = g_{\gamma_{i_e}} - a_\gamma \phi_{\gamma'} g_{\gamma_{i_k}}$.  
 \item[Case 2   $\gamma\in \Delta\setminus(\gamma_{i_{1}},\dots,\gamma_{i_{m}})_{\Gamma}$]  Denote $\{i^{\prime}_{1},\dots, i^{\prime}_{n-m}\}:=\{1,\dots,n\}\setminus \{i_{1},\dots, i_{m}\} $. Since,  for $k\in \{i^{\prime}_{1},\dots, i^{\prime}_{n-m}\}$, $ f_{\gamma_k}  $can be generated by $ g_{\gamma_{i_1}}, \dots, g_{\gamma_{i_m}} $. Then there are elements $ h_{k_{j}} \in \mathcal{O}_C $ such that
\[
f_{\gamma_k} = \sum_{j \in \{i_1, \dots, i_m\}} h_{k_{j}} g_{\gamma_{j}}.
\]
Now, assume $ \gamma = x + y $ with $ x \in \Gamma $ and $ y \in \{\gamma_{i_1'}, \dots, \gamma_{i^{\prime}_{n-m}}\} $, define
\[
g'_{\gamma_{i_e}} = g_{\gamma_{i_e}} - a_\gamma \phi_x f_y.
\]
\end{description}

Then $\{g_{\gamma_{i_1}}, \dots, g'_{\gamma_{i_e}}, \dots, g_{\gamma_{i_m}}\}$ is also a minimal system of generators of $I_\lambda$. 
\vskip 0.5cm
Indeed, we have the relation
\[
(g_{\gamma_{i_1}}, \dots, g'_{\gamma_{i_e}}, \dots, g_{\gamma_{i_m}}) = 
(g_{\gamma_{i_1}}, \dots, g_{\gamma_{i_e}}, \dots, g_{\gamma_{i_m}}) A,
\]
where $A = B + I_m \in M_{m,m}(\mathcal{O}_C)$. The matrix $B = (b_{i,j})$ has all entries zero except possibly in the $e$-th column, and $I_m$ is the identity matrix. Since $\gamma > \gamma_{i_e}$, it follows that $b_{e,e}$ is not a unit in $\mathcal{O}_C$. However, the diagonal entries of $A$ satisfy $a_{k,k} = 1$ for $k \ne e$, and $a_{e,e} = 1 + b_{e,e}$, where $b_{e,e}$ are elements of the maximal ideal of $\mathcal{O}_C$. Hence, $a_{e,e}$ is a unit in $\mathcal{O}_C$, and therefore $\det(A)$ is a unit in $\mathcal{O}_C$.
\vskip 0.5cm
By Nakayama's Lemma, this implies that $\{g_{\gamma_{i_1}}, \dots, g'_{\gamma_{i_e}}, \dots, g_{\gamma_{i_m}}\}$ is also a minimal generating set for $I_\lambda$. We may thus replace $g_{\gamma_{i_e}}$ by $g'_{\gamma_{i_e}}$ in the generating set.
\vskip 0.5cm
Continuing this process, we inductively eliminate all terms in $ g_{\gamma_{i_e}} $ corresponding to exponents in $ \Delta $. At each step, we replace a generator by an equivalent one with strictly smaller support, using the previously described transformation. Since $ \mathcal{O}_C $ is a complete local ring, this process converges in the $ \mathfrak{m}_C $-adic topology, yielding a limit generator. By the Lemma~\ref{if lammb distinct}, we have $g_{\gamma_{i_{e}}} = f_{\gamma_{i_{e}}}(\lambda)$.

If $\{i_{e}\}\neq \{j_{e}\}$,  let   $\gamma_{j_{s}}= \max_{e}\{\gamma_{j_{e}}\notin \{\gamma_{i_{k}}\}\}$ and $\gamma_{i_{t}}= \max_{k}\{\gamma_{i_{k}}\notin \{\gamma_{j_{e}}\}\}$. Then $s= t>1$. Assume $\gamma_{j_{s}}<\gamma_{i_{s}}$. Since $f_{\gamma_{i_{1}}}(\lambda),\dots,f_{\gamma_{i_{m}}}(\lambda)$ can be generated by  $f_{\gamma_{j_{1}}}(\lambda),\dots,f_{\gamma_{j_{m}}}(\lambda)$,   we have 
\[
(f_{\gamma_{i_{1}}}(\lambda),\dots,f_{\gamma_{i_{s}}}(\lambda),\dots,f_{\gamma_{i_{m}}}(\lambda)) = (f_{\gamma_{j_{1}}}(\lambda),\dots,f_{\gamma_{j_{s}}}(\lambda),\dots,f_{\gamma_{i_{m}}}(\lambda))A
\]
where  
\[
A = \left(\begin{matrix}
B & \alpha & 0    \\
C &\beta & I_{m-s}    \\
\end{matrix}\right)\in M_{m, m}(\mathcal{O}_{C}),
\]
and $I_{m-s}$ is the identity matrix in $ M_{m-s, m-s}(\mathcal{O}_{C})$. Since $\gamma_{j_{s}}<\gamma_{i_{s}}$, then entries of  $\alpha$ are in maximal ideal of $\mathcal{O}_{C}$. Then $det(A)= det (B,\alpha)$ is not an unit in $\mathcal{O}_{C}$. By Nakayama's lemma, $f_{\gamma_{i_{1}}}(\lambda),\dots,f_{\gamma_{i_{m}}}(\lambda)$ is not a  minimal system of generators of $I_{\lambda}$.  This contradicts
the hypothesis.  
\end{proof}

\begin{definition}\label{defCDeltai}
For $\underline{i} = \{i_1, \dots, i_m\} \subseteq \{1, \dots, n\}$, we define
\[
C^{[\Delta],\underline{i}} := \left\{ I_\lambda \in C^{[\Delta]} \;\middle|\; 
I_\lambda \text{ admits a generating set of size } m \text{ consisting of } f_{\gamma_{i_1}}(\lambda), \dots, f_{\gamma_{i_m}}(\lambda) \right\}.
\]
\end{definition}

\begin{proposition}\label{decomposition of C Delta m}

\item 
\begin{enumerate}
    \item[(i)] For $m \in \mathbb{Z}_{\geq 1}$, we have the stratification
    \[
    C^{[\Delta],m} = C^{[\Delta],\leq m} \setminus C^{[\Delta],\leq m-1},
    \]
    where $C^{[\Delta],\leq m} = \{ I \in C^{[\Delta]} \mid m(I) \leq m \}$.

    \item[(ii)] For each subset $\underline{i} \subseteq \{1, \dots, n\}$, the strata $C^{[\Delta],\underline{i}}$, $C^{[\Delta],\leq m}$, and $C^{[\Delta],m}$ are locally closed subvarieties of $C^{[\Delta]}$.

    \item[(iii)] For subsets $\underline{i}, \underline{j} \subseteq \{1, \dots, n\}$ with $\# \underline{i} = \# \underline{j} = m$, we have
    \[
    C^{[\Delta],\underline{i}} \cap C^{[\Delta],\underline{j}} = C^{[\Delta],\underline{i} \cap \underline{j}}.
    \]
       Furthermore,   denote     $\{ \underline{i} \subseteq \{1,\dots,n\} \mid  \# \underline{i} = m \} = \{ \underline{i}_{1},\dots, \underline{i}_{t}\}$  the set of subsets $\{1,\dots,n\}$ of size $m$ with $t = \binom{n}{m}$.  By the inclusion-exclusion principle, we have

    \begin{align*}
    [C^{[\Delta],\leq m}] 
    &= \left[\bigcup_{\substack{\underline{i} \subseteq \{1,\dots,n\} \\ \#\underline{i} = m}} C^{[\Delta],\underline{i}}\right] \\
    &= \sum_{\substack{\underline{i}}} [C^{[\Delta],\underline{i}}] 
       - \sum_{\substack{\underline{i},\underline{j} \\ \underline{i} < \underline{j}}} [C^{[\Delta],\underline{i} \cap \underline{j}}] 
       + \cdots 
       + (-1)^{t-1} [C^{[\Delta],\underline{i}_1 \cap \cdots \cap \underline{i}_t}] \\
    &= \sum_{\substack{\mathcal{I} \subseteq \{1,\dots,t\} \\ \mathcal{I} \neq \varnothing}} (-1)^{|\mathcal{I}|-1} [C^{[\Delta],\bigcap_{e \in \mathcal{I}} \underline{i}_e}] \quad \text{in } K_0(\mathrm{Var}_{\mathbf{C}}),
    \end{align*}
    where the ordering $\underline{i} < \underline{j}$ is lexicographic.
\end{enumerate}\end{proposition}

\begin{proof}
Statement (i) is immediate from the definition of $C^{[\Delta],m}$ as the set of ideals requiring exactly $m$ generators.

For (ii), the fact that $C^{[\Delta],\underline{i}}$, $C^{[\Delta],\leq m}$, and $C^{[\Delta],m}$ are locally closed in $C^{[\Delta]}$ follows from Theorem~\ref{decomposition of C Delta i}, which establishes that the strata defined by generator support conditions are constructible and locally closed under the embedding $C^{[\Delta]} \hookrightarrow \operatorname{Gen}_{\Delta}$.

For (iii), let $m \in \mathbb{Z}_{\geq 1}$ and fix subsets $\underline{i}, \underline{j} \subseteq \{1, \dots, n\}$ with $\#\underline{i} = \#\underline{j} = m$. For $k \leq m$, define the subset
\[
C^{[\Delta],\underline{i},k} := \left\{ I_\lambda \in C^{[\Delta]} \;\middle|\; 
\begin{array}{c}
I_\lambda \text{ is generated by } f_{\gamma_{i_1}}(\lambda), \dots, f_{\gamma_{i_m}}(\lambda), \\
\text{and } m(I_\lambda) = k
\end{array}
\right\}.
\]
Then we have a stratification
\[
C^{[\Delta],\underline{i}} = \bigsqcup_{k \leq m} C^{[\Delta],\underline{i},k}.
\]
By Proposition~\ref{minimal systen generator of I lamba}, the minimal generating set of an ideal $I_\lambda$ is uniquely determined. Therefore, an ideal $I_\lambda$ lies in both $C^{[\Delta],\underline{i}}$ and $C^{[\Delta],\underline{j}}$ if and only if it can be minimally generated using elements indexed by $\underline{i} \cap \underline{j}$, after removing redundant generators.

It follows that
\[
C^{[\Delta],\underline{i}} \cap C^{[\Delta],\underline{j}} = \bigsqcup_{k \leq m} \left( C^{[\Delta],\underline{i},k} \cap C^{[\Delta],\underline{j},k} \right)
= \bigsqcup_{k \leq m} C^{[\Delta],\underline{i} \cap \underline{j},k}
= C^{[\Delta],\underline{i} \cap \underline{j}}.
\]
This completes the proof of (iii).
\end{proof}

Let $m \in \mathbb{Z}_{\geq 1}$. By Proposition~\ref{decomposition of C Delta m}, to compute the motivic class $[C^{[\Delta],m}]$ in the Grothendieck ring $K_0(\mathrm{Var}_{\mathbb{C}})$, it suffices to compute the classes $[C^{[\Delta],\underline{i}}]$ for all subsets $\underline{i} \subseteq \{1, \dots, n\}$ with $\# \underline{i} \leq m$.

\begin{theorem}\label{decomposition of C Delta i}
Let $\Gamma$ be the value semigroup associated with the germ of an irreducible curve singularity $(C,O)$.  Let $\Delta = (\gamma_1, \dots, \gamma_n)_\Gamma$, with $n > 1$, be a $\Gamma$-subsemimodule minimally generated by $n$ elements as a $\Gamma$-module.  Fix an integer $m$ such that $1 < m \leq n$, and let
\[
\underline{i} = \{i_1, \dots, i_m\} \subset \{1, \dots, n\}, \quad \{i'_1, \dots, i'_{n-m}\} = \{1, \dots, n\} \setminus \underline{i}.
\]
For each $e = 1, \dots, n-m$, choose an element in the Augmented Syzygy ( see Definition \ref{syzygy and aug syzygy} ) of $\Gamma-$ semimodule  $(\gamma_{i_1}, \dots, \gamma_{i_m}, \gamma_{i'_1}, \dots, \gamma_{i'_{e-1}})_{\Gamma}$:
\[
(\gamma_{i_{j_{1}}},\widehat{\gamma_{i_{j_{1}}}},  \sigma_{i_{j_e}}) \in \operatorname{ASyz}((\gamma_{i_1}, \dots, \gamma_{i_m}, \gamma_{i'_1}, \dots, \gamma_{i'_{e-1}})_{\Gamma})
\]
such that $\sigma_{i_{j_e}} < \gamma_{i'_e}$. 

Let $Y_{\underline{i}_{j_e}}$ be the subvariety defined by the conditions:
\begin{equation}\label{eq:Y_ie_conditions}
\sum_{k \in \{i_1, \dots, i_m, i_{1}^{\prime},\dots, i_{e-1}^{\prime}\}} (\mathcal{G}^{(e)}_{\lambda})_k \circ (\mathcal{S}^{(e)}_{\nu})_{i_{j_e}}^k \equiv 0 \mod t^{\gamma_{i'_e}}, 
\quad \text{and} \quad 
(Eq^{(e)})^{\gamma_{i'_e} - \sigma_{i_{j_e}}}_{i_{j_e}} \neq 0,
\end{equation}
where $(Eq^{(e)})^{\gamma_{i'_e} - \sigma_{i_{j_e}}}_{i_{j_e}}$ is the coefficient of $\phi_{\gamma_{i'_e}}$ in the expansion of the left-hand side, as defined in the proof.

Now define
\[
Y_{\underline{i}_{\underline{j}}} = \bigcap_{e=1}^{n-m} Y_{\underline{i}_{j_e}}.
\]
Then we have the following decomposition:
\[
C^{[\Delta],\leq m} = \bigcup_{\substack{\underline{i} \subseteq \{1,\dots,n\} \\ |\underline{i}| = m}} C^{[\Delta],\underline{i}} = \bigcup_{\substack{\underline{i} \subseteq \{1,\dots,n\} \\ |\underline{i}| = m}} \bigcup_{\underline{j}} Y_{\underline{i}_{\underline{j}}},
\]
where the second union runs over all tuples $\underline{j} = (j_1, \dots, j_{n-m})$ corresponding to valid tuples chosen in the Augmented Syzygy  as above.

For $\Gamma = \langle p, q \rangle$ with $\gcd(p,q) = 1$, we have:
\begin{equation}\label{eq:Yij_geometry}
Y_{\underline{i}_{\underline{j}}} \cong ({\mathbb{C}}^*)^{n - m} \times_{\operatorname{Spec} {\mathbb{C}}} \mathbb{A}^{N(\Delta) - n + m},
\end{equation}
where $N(\Delta) = \dim C^{[\Delta]}$ is the dimension of the stratum, and $C^{[\Delta]} \cong \mathbb{A}^{N(\Delta)}$ as established in Theorem~\ref{dim of Hilbert scheme}.

\end{theorem}
\begin{proof}
Let $I_\lambda = (f_{\gamma_1}(\lambda), \dots, f_{\gamma_n}(\lambda)) \in C^{[\Delta],\leq m}$.  
By Proposition~\ref{minimal systen generator of I lamba}, we may assume that $I_\lambda$ can be generated by the set  
$$
\{f_{\gamma_{i_1}}(\lambda), \dots, f_{\gamma_{i_m}}(\lambda)\},
$$  
where $v(f_{\gamma_{i_e}}(\lambda)) = \gamma_{i_e}$ for each $e = 1, \dots, m$.  
In particular, this implies $I_\lambda \in C^{[\Delta],\underline{i}}$, where $\underline{i} = \{i_1, \dots, i_m\}$.

Note that there exists an element $g_{\gamma_{i^{\prime}_{1}}}$ in $I_{\lambda}$  with order $\gamma_{i_{1}^{\prime}}$ which can be generated by $f_{\gamma_{i_{1}}}(\lambda),\dots, f_{\gamma_{i_{m}}}(\lambda)$. 
Thus, there exists $(\gamma_{i_{j_{1}}},  \widehat{\gamma_{i_{j_{1}}}}, \sigma_{i_{j_{1}}})\in ASyz((\gamma_{i_{1}},\dots,\gamma_{i_{m}})_{\Gamma})$ with $\sigma_{i_{j_{1}}}<\gamma_{i^{\prime}_{1}}$.
Consider the tuple $(\gamma_{i_{j_{1}}}, \widehat{\gamma_{i_{j_{1}}}}, \sigma_{i_{j_{1}}})$. For any element $s$ in the set $((\gamma_{i_1}, \dots, \gamma_{i_m}){\Gamma}){ > \sigma_{i_{j_1}}, \leq \gamma_{i^{\prime}_{1}}}$, we choose a sub-decomposition from the decomposition that defines $C^{[\Delta]}$. We then write $s$ as $s = \gamma{g(s)} + \rho(s)$, where $\rho(s) \in \Gamma$. One can assign an $m\times 1$ matrix with entries 
\[(\mathcal{S}_{\nu}^{(1)})_{i_{j_{1}}}^{k}:=u_{i_{j_{1}}}^{k}\phi_{\sigma_{i_{j_{1}}}-\gamma_{k}}+\sum_{s\in ((\gamma_{i_{1}},\dots,\gamma_{i_{m}})_{\Gamma})_{>\sigma_{i_{j_{1}}} , \leq  \gamma_{i^{\prime}_{1}}}, g(s)=k}\nu_{i_{j_{1}}s}^{s-\sigma_{i_{j_{1}}}}\phi_{s-\gamma_{k}},\]
and a matrix with entries 
\[(\mathcal{G}^{(1)}_{\lambda})_{k}:=(\mathcal{G}^{\Delta}_{\lambda})_{k},\]
for $\gamma_{k} = \gamma_{i_{1}},\dots,\gamma_{i_{m}}.$  Thus, $g_{\gamma_{i^{\prime}_{1}}}$ can be generated by $f_{\gamma_{i_{1}}}(\lambda),\dots, f_{\gamma_{i_{m}}}(\lambda)$ if and only if  we have following equations:

\begin{equation}\label{C Delta i}
\sum_{k=i_{1}, \dots,i_{m}}(\mathcal{G}^{(1)}_{\lambda})_{k}\circ (\mathcal{S}^{(1)}_{v})_{i_{j_{1}}}^{k}=O(t^{\gamma_{i^{\prime}_{1}}-1}), \quad (Eq^{(1)})^{{\gamma_{i^{\prime}_{1}}-\sigma_{i_{j_{1}}}}}_{i_{j_{1}}}\phi_{\gamma_{i^{\prime}_{1}}}\neq 0,
\end{equation}
where $(Eq^{(1)})^{{\gamma_{i^{\prime}_{1}}-\sigma_{i_{j_{1}}}}}_{i_{j_{1}}}$ is the coefficient of $\phi_{\gamma_{i^{\prime}_{1}}}$  in the expansion of  $\sum_{k=i_{1}, \dots,i_{m}}(\mathcal{G}^{(1)}_{\lambda})_{k}\circ (\mathcal{S}^{(1)}_{v})_{i_{j_{1}}}^{k}$. 

We can eliminate the term of $g_{\gamma_{i^{\prime}_{1}}}$ in $(\gamma_{i_{1}},\dots,\gamma_{i_{m}},\gamma_{i^{\prime}_{1}})_{\Gamma}$ , then we define
\[(\mathcal{G}^{(2)}_{\lambda})_{i^{\prime}_{1}}:=\phi_{\gamma_{i^{\prime}_{1}}}+\sum_{k\in \Gamma_{>\gamma_{i^{\prime}_{1}}}\setminus (\gamma_{i_{1}},\dots,\gamma_{i_{m}},\gamma_{i^{\prime}_{1}})_{\Gamma})}\lambda_{i^{\prime}_{1}}^{k-\gamma_{i^{\prime}_{1}}} \phi_{k}.\]

There exists an element $g_{\gamma_{i^{\prime}_{2}}}$ in $I_{\lambda}$  with order $\gamma_{i_{2}^{\prime}}$ which can be generated by 
\[\{f_{\gamma_{i_{1}}}(\lambda),\dots, f_{\gamma_{i_{m}}}(\lambda),g_{\gamma_{i^{\prime}_{1}}}\}.\] 
Thus,  there exists $(\gamma_{i_{j_{2}}}, \widehat{\gamma_{i_{j_{2}}}}, \sigma_{i_{j_{2}}}) \in ASyz( (\gamma_{i_{1}},\dots,\gamma_{i_{m}},\gamma_{i^{\prime}_{1}})_{\Gamma} )$ with $\sigma_{i_{j_{2}}} < \gamma_{i^{\prime}_{2}}$ continues the process. 

Then $I_{\lambda}$ lies in the subvariety of $C^{[\Delta]}$ defined by those equations. Namely, 
\[ I_{\lambda} \in  \bigcap_{e=1}^{n-m} Y_{\underline{i}_{j_e}}
\]

For $\Gamma=\langle p,q \rangle$, analyzing $Y_{\underline{i}_{\underline{j}}}$ reduces to understanding how equations $(Eq^{(e)})^{r}_{i_{j_{e}}}$ constrain the $\lambda_{\bullet}^{\bullet}$  which defines the ideal $I_{\lambda}$ in $C^{[\Delta]}$.  We start by comparing $(Eq^{(1)})^{r}_{i_{j_{1}}}$, the coefficient of $\phi_{r+\sigma_{i_{j_{1}}}}$ in the equation $(\ref{C Delta i})$ and $(Eq^{
\Delta
})^{r}_{i_{j_{1}}}$, the coefficient of $\phi_{r+\sigma_{i_{j_{1}}}}$ in the equation defing $C^{[\Delta]}.$ \\

Note that $(\gamma_{i_{1}},\dots,\gamma_{i_{m}})_{\Gamma}\subset \Delta \subset \Gamma$. To simplify the notations, we denote by $*\lambda$ and $*v$ the non-zero linear term  in $(Eq^{(1)})^{r}_{i_{j_{1}}}$ ( and in $(Eq^{\Delta})^{r}_{i_{j_{1}}}$) corresponding to the parameters $\lambda_{\bullet}^{\bullet}$ and $v_{\bullet,\bullet}^{\bullet}$ separately.    This allows us to express $(Eq^{(1)})^{r}_{i_{j_{1}}}$ in the  following form: 
\[(Eq^{(1)})^{r}_{i_{j_{1}}}=\left\{\begin{aligned}
&*\lambda+*v + \mbox{not linear terms} ,  &\mbox{ if }r+\sigma_{i_{j_{1}}}\in (\gamma_{i_{1}},\dots,\gamma_{i_{m}})_{\Gamma},  \\    
&*\lambda + \mbox{not linear terms} , &\mbox{ if } r+\sigma_{i_{j_{1}}}\in \Delta\setminus (\gamma_{i_{1}},\dots,\gamma_{i_{m}})_{\Gamma},\\
&*\lambda + \mbox{not linear terms}    , & \mbox{ if } r+\sigma_{i_{j_{1}}}\in \Gamma\setminus \Delta.
\end{aligned}
\right.
\]
And $(Eq^{\Delta})^{r}_{i_{j_{1}}}$ is of form:
\[(Eq^{\Delta})^{r}_{i_{j_{1}}}=\left\{\begin{aligned}
&*\lambda+*v + \mbox{not linear terms} ,   &\mbox{ if }r+\sigma_{i_{j_{1}}}\in (\gamma_{i_{1}},\dots,\gamma_{i_{m}})_{\Gamma},  \\    
&*\lambda + *v +\mbox{not linear terms}     ,   &\mbox{ if }r+\sigma_{i_{j_{1}}}\in \Delta\setminus (\gamma_{i_{1}},\dots,\gamma_{i_{m}})_{\Gamma},\\
&*\lambda + \mbox{not linear terms}    ,   &\mbox{ if } r+\sigma_{i_{j_{1}}}\in \Gamma\setminus \Delta.
\end{aligned}
\right.
\]
By the comparison,  only the term $(Eq^{(1)})^{r}_{i_{j_{1}}}$, for $r+\sigma_{i_{j_{1}}}\in \Delta\setminus (\gamma_{i_{1}},\dots,\gamma_{i_{m}})_{\Gamma}$, create one more linear constraint on $\lambda_{\bullet}^{\bullet}$. However, $\Delta_{>\sigma_{i_{j_{1}}},< \gamma_{i_{1}^{\prime}}} \setminus (\gamma_{i_{1}},\dots,\gamma_{i_{m}})_{\Gamma}$ is empty set. In fact,  note that $\Delta_{>\sigma_{i_{j_{1}}},< \gamma_{i_{1}^{\prime}}} \subset (\gamma_{i_{1}},\dots,\gamma_{i_{m}})_{\Gamma}.$  Let $\gamma_{i_{s}}=\underset{{t=1,\dots, m}}{\max}\{\gamma_{i_{t}}\mid \gamma_{i_{t}}< \gamma_{i_{1}^{\prime}}\}$ and $\gamma_{i_{1}^{\prime}}$ is the minimal element not in $\{\gamma_{i_{1}},\dots,\gamma_{i_{m}}\}$.  Then 
$\gamma_{i_{1}}=\gamma_{1},\dots, \gamma_{i_{s-1}}=\gamma_{s-1}$ and  

\[\Delta_{>\sigma_{i_{j_{1}}},< \gamma_{i_{1}^{\prime}}} \subset (\gamma_{1},\dots, \gamma_{s})_{\Gamma}=(\gamma_{i_{1}},\dots, \gamma_{i_{s}})_{\Gamma}\subset(\gamma_{i_{1}},\dots, \gamma_{i_{m}})_{\Gamma}.\]

Then $\Delta_{>\sigma_{i_{j_{1}}},< \gamma_{i_{1}^{\prime}}} \setminus (\gamma_{i_{1}},\dots,\gamma_{i_{m}})_{\Gamma}$ is empty set. Hence, only the term $(Eq^{(1)})^{\gamma_{i_{1}^{\prime}}-\sigma_{i_{j_{1}}}}_{i_{j_{1}}}$ create one more constraint on $\lambda_{\bullet}^{\bullet}$.   Continue the process,  we obtain the formula. 
\end{proof}

\begin{example}

Let $\Gamma = \langle 4, 9 \rangle$,  
$\Delta = (12, 17, 22, 27)_{\Gamma}$,  
$\gamma_{\Delta} = 18$,  
$c(\Delta) = 24$,  
and $\mathrm{Syz}(\Delta) = (21, 26, 31, 36)_{\Gamma}$.

Let $\Delta_{\underline{i}} = (12, 17)_{\Gamma}$.  
We will compute the set $C^{[\Delta],\underline{i}}$.

\noindent
First, we compute $C^{[\Delta]}$: 
Note that $\Gamma\setminus\Delta = \{0,4,8,9,13,18\}.$
\[
\begin{aligned}
    f_{12}(\lambda) &= t^{12}+\lambda_{{1}}^{1}t^{13}+\lambda_{{1}}^{6}t^{18},\\
    f_{17}(\lambda) &= t^{17}+\lambda_{{2}}^{1}t^{18},\\
    f_{22}(\lambda) &= t^{22},\\
    f_{27}(\lambda) &= t^{27}.\\
\end{aligned}
\]

For $\sigma_{1}=21= 12+9=17+4$, consider $\Delta_{>21,\leq 22} = \{22\}$. Fix a decomposition of elements in $\Delta_{>21,\leq 22}$.
\[22=22+0\]
Define
\[
\begin{aligned}
    (S_{v})_{1}^{1} &= t^{9},\\
    (S_{v})_{1}^{2} &= -t^{4},\\
    (S_{v})_{1}^{3} &= v_{1,22}^1,\\
    (S_{v})_{1}^{4} &= 0.\\
\end{aligned}
\]
We have 
\[
(Eq^{\Delta})_2^{1}t^{22} = (\lambda_{1}^{1}-\lambda_{2}^1+v_{1,26}^1)t^{22}=O(t^{22}).
\]
Then $\sigma_{1}$ do not provide any constraint on $\lambda_{\bullet}^{\bullet}$. \\

For $\sigma_{2}=26>c(\Delta)$,  $26=22+4=17+9$. Then $\sigma_{2}$ does not provide any constraint on $\lambda_{\bullet}^{\bullet}$. 
However, we continue the process for comparing $(Eq^{\Delta})_{2}^{r}$ and $(Eq^{(2)})_{2}^{r}$. Fix a decomposition of elements in $\Delta_{>26,\leq 27}=\{27\}$.
\[27=27+0.\]
Define 
\[
\begin{aligned}
    (S_{v})_{2}^{1} &= 0\\
    (S_{v})_{2}^{2} &= t^{9},\\
    (S_{v})_{2}^{3} &= -t^{4},\\
    (S_{v})_{2}^{4} &= v^{1}_{2,27}.\\
\end{aligned}
\]
We have 
\[
(Eq^{\Delta})_2^{6}t^{27} =(\lambda_{2}^1+ v^{1}_{2,27})t^{27}=O(t^{22})    
\]
Hence,  $C^{[\Delta]}\cong \operatorname{Spec}{\mathbb{C}}[\lambda_{1}^{1},\lambda_{1}^{6},\lambda_{2}^{1}]$. \\

\noindent
Let $\Delta_{\underline{i}}=(12,17)_{\Gamma}.$ $Syz(12,17)_{\Gamma}=(21,44)_{\Gamma}$, $Syz(12,17,22)_{\Gamma}=(21,26,40)_{\Gamma}$. Recall that $C^{[\Delta],\underline{i}}=\cup Y_{\underline{i}_{\underline{j}}}$. However, there exists only one choice for $Y_{\underline{i}_{\underline{j}}}$: 
\[
\sigma_{i_{j_{1}}} = \sigma_1 = 21 < 22, \quad
\sigma_{i_{j_2}} = \sigma_2 = 26 < 27
\]
For 
$(12,17)_{>\sigma_1,\leq 22}=\{22\}$, define 
\[
\begin{aligned}
    (S_{v})_{1}^{1} &= t^{9}\\
    (S_{v})_{1}^{2} &= -t^{4}\\
\end{aligned}
\]
Then 
\[(Eq^{(1)})_{1}^{1}=(\lambda_{1}^{1}-\lambda_{2}^1)t^{22}.\]
Thus, $f_{22}(\lambda)$ can be generated by  $f_{12}(\lambda),f_{17}(\lambda)$ if and only if $\lambda_{1}^{1}-\lambda_{2}^1\neq 0$. \\

\noindent
For $(12,17,22)_{>\sigma_2,\leq 27}=\{27\}$, define
\[
\begin{aligned}
    (S_{v})_{1}^{1} &= 0\\
    (S_{v})_{1}^{2} &= t^{9}\\
    (S_{v})_{1}^{1} &= -t^{4}\\
\end{aligned}
\]
Then 
\[(Eq^{(2)})_{2}^{1}=\lambda_{2}^1t^{27}.\]
It follows that $I_{\lambda}\in C^{[\Delta],\underline{i}}$ if and only if $\lambda_{1}^{1}\neq \lambda_{2}^1$ and $\lambda_{2}^1\neq 0$. 
\end{example}

\begin{remark}
Assume $C^{[\Delta],\underline{i}}$ is a union of  $Y_{\underline{i}_{\underline{j_{g}}}}$ as described in  Theorem \ref{decomposition of C Delta i}, $\underline{i}_{\underline{j_{g}}}=\underline{i}_{\underline{j_{1}}}, \dots, \underline{i}_{\underline{j_{\eta}}}$. By inclusion-exclusion principle,  we have:
\[ [C^{[\Delta],\underline{i}}]=[\bigcup_{\underline{i}_{\underline{j_{g}}}=\underline{i}_{\underline{j_{1}}},\dots \underline{i}_{\underline{j_{\eta}}}}Y_{\underline{i}_{\underline{j_{g}}}}]= 
\sum_{\underline{j_{g}}}  [Y_{\underline{i}_{\underline{j_{g}}}}]-\sum_{\underline{i}_{\underline{j_{g}}},\underline{i}_{\underline{j_{f}}}, g<f}  [Y_{\underline{i}_{\underline{j_{g}}}}\cap Y_{\underline{i}_{\underline{j_{f}}}}]+\dots + (-1)^{(\eta-1)} [Y_{\underline{i}_{\underline{j_{1}}}}\cap \dots \cap Y_{\underline{i}_{\underline{j_{\eta}}}}]
\]
\end{remark}

As a corollary of Theorem~\ref{decomposition of C Delta i}, we derive a formula for the intersections of the sets $Y_{\underline{i}_{\underline{j_g}}}$.

In the following, we regard $\sigma_{\underline{i}_{\underline{j_{g}}}}=(\{\sigma_{i_{j_{g_{1}}}}\},\dots,\{\sigma_{i_{j_{g_{n-m}}}}\})$ as $(n-m)$-tuple of sets whose components are sets. We can define a union operator of two tuples of sets: \[\sigma_{\underline{i}_{\underline{j_{g}}}}\cup \sigma_{\underline{i}_{\underline{j_{f}}}}:=(\{\sigma_{i_{j_{g_{1}}}}\}\cup \{\sigma_{i_{j_{f_{1}}}}\} ,\dots,\{\sigma_{i_{j_{g_{n-m}}}}\}\cup \{\sigma_{i_{j_{f_{n-m}}}}\})\]

We define the cardinality of a tuple of sets as the sum of the cardinalities of its components, i.e., 
\[\#(\sigma_{\underline{i}_{\underline{j_{g}}}}\cup \sigma_{\underline{i}_{\underline{j_{f}}}}):=\sum_{e=1}^{n-m}\#(\{\sigma_{i_{j_{g_{e}}}}\}\cup \{\sigma_{i_{j_{f_{e}}}}\}). \]

\begin{corollary}
For the curve $C$ defined by $y^p = x^q$, we assume that $C^{[\Delta],\underline{i}}$ is a union of the sets $Y_{\underline{i}_{\underline{j_g}}}$ as described in Theorem~\ref{decomposition of C Delta i}, where $\underline{i}_{\underline{j_g}} = \underline{i}_{\underline{j_1}}, \dots, \underline{i}_{\underline{j_\eta}}$. Then, for $s \leq \eta$, we have:
\[
Y_{\underline{i}_{\underline{j}_{1}}} \cap  \dots \cap  Y_{\underline{i}_{\underline{j}_{s}}} \cong (\mathbb{C}^{*})^{\#(\sigma_{\underline{i}_{\underline{j}_{1}}} \cup \dots \cup \sigma_{\underline{i}_{\underline{j}_{s}}})} \times _{\mathbb{C}} \mathbb{A}^{N(\Delta) - \#(\sigma_{\underline{i}_{\underline{j}_{1}}} \cup \dots \cup \sigma_{\underline{i}_{\underline{j}_{s}}})}.
\]

\begin{proof}

We provide the verification only for the specific case of $Y_{\underline{i}_{\underline{j_1}}} \cap Y_{\underline{i}_{\underline{j_2}}}$. The proof for the general case follows by a similar argument.

 Recall that,    $Y_{\underline{i}_{\underline{j_{g}}}} = \cap Y_{\underline{i}_{j_{g_{e}}}}$, $g=1,2$. For $e=1,\dots ,n-m$,  we chose 
 
\[\sigma_{i_{j_{g_{e}}}}\in Syz((\gamma_{i_{1}},\dots, \gamma_{i_{m}}, \gamma_{i^{\prime}_{1}},\dots,  \gamma_{i^{\prime}_{e-1}})_{\Gamma})\quad \text{with }\sigma_{i_{j_{g_{e}}}}<\gamma_{i_{e}^{\prime}}.\]

 Then $Y_{\underline{i}_{j_{g_{e}}}}$ is defined by 
\begin{equation}\label{condition C Delat m}
\sum_{k=i_{1}, \dots,i_{m},i^{\prime}_{1},\dots i^{\prime}_{e-1}}(\mathcal{G}^{(e)}_{\lambda})_{k}\circ (\mathcal{S}^{(e)}_{v})_{i_{j_{g_{e}}}}^{k}=O(t^{\gamma_{i^{\prime}_{e}}-1}), \quad (Eq^{(e)})^{{\gamma_{i^{\prime}_{e}}-\sigma_{i_{j_{g_{e}}}}}}_{i_{j_{g_{e}}}}\phi_{\gamma_{i^{\prime}_{e}}}\neq 0,
\end{equation}
where $(Eq^{(e)})^{{\gamma_{i^{\prime}_{e}}-\sigma_{i_{j_{g_{e}}}}}}_{i_{j_{g_{e}}}}$ is the coefficient of $\phi_{\gamma_{i^{\prime}_{e}}}$  in the expansion of  
$\sum(\mathcal{G}^{(e)}_{\lambda})_{k}\circ (\mathcal{S}^{(e)}_{v})_{i_{j_{g_{e}}}}^{k}.$
In the following, we will analyze that  this additional constraint don't introduce any interdependence between the restrictions for $g=1,2:$

Considering, 
$\sigma_{\underline{i}_{\underline{j_{1}}}}=(\{\sigma_{i_{j_{1_{1}}}}\},\dots,\{\sigma_{i_{j_{1_{n-m}}}}\}),\sigma_{\underline{i}_{\underline{j_{2}}}}=(\{\sigma_{i_{j_{1_{2}}}}\},\dots,\{\sigma_{i_{j_{2_{n-m}}}}\})$.  
\begin{itemize}
\item  Case 1,  Comparing the same position of  $\sigma_{\underline{i}_{\underline{j_{1}}}}$ and $\sigma_{\underline{i}_{\underline{j_{2}}}}$. For some $e = 1,\dots, n-m$:
\begin{itemize}
\item Case1.1, $\sigma_{i_{j_{1_{e}}}} =\sigma_{i_{j_{2_{e}}}}$. They yield the same equation~\eqref{condition C Delat m}, i.e. $Y_{\underline{i}_{j_{1_{e}}}}=Y_{\underline{i}_{j_{2_{e}}}}.$
\item Case 1.2,   $\sigma_{i_{j_{1_{e}}}} \neq \sigma_{i_{j_{2_{e}}}}$. Since the linear part of  $(Eq^{(e)})^{{\gamma_{i^{\prime}_{e}}-\sigma_{i_{j_{g_{e}}}}}}_{i_{j_{g_{e}}}}$ are linearly independent for $g=1,2.$, then the corresponding constraints are independent. \\
\end{itemize}

\item  Case 2, Comparing the distinct position of  $\sigma_{\underline{i}_{\underline{j_{1}}}}$ and $\sigma_{\underline{i}_{\underline{j_{2}}}}$. For some distinct $e,f = 1,\dots, n-m$ with $e<f$:

\begin{itemize}
\item Case 2.1,  $\sigma = \sigma_{i_{j_{1_{e}}}} = \sigma_{i_{j_{2_{f}}}}.$  The interdependence yielded by the potential equations is that
\[(Eq^{(e)})^{{\gamma_{i^{\prime}_{e}}-\sigma}}_{i_{j_{1_{e}}}}\neq 0 , (Eq^{(f)})^{{\gamma_{i^{\prime}_{e}}-\sigma}}_{i_{j_{2_{f}}}}= 0. \]
We denote $*\lambda$ and $*v$ the not zero linear term  in the equation.  However,  $(Eq^{(e)})^{{\gamma_{i^{\prime}_{e}}-\sigma}}_{i_{j_{1_{e}}}}$ in the following form: 
\[*\lambda + \mbox{not linear terms}\]
Because $(\mathcal{G}^{(f)}_{\lambda})_{i^{\prime}_{e}}$  was involved in the creation of $(\mathcal{G}^{(f)}_{\lambda})_{i^{\prime}_{f}}$, then $(Eq^{(f)})^{{\gamma_{i^{\prime}_{e}}-\sigma}}_{i_{j_{2_{f}}}}$ in the following form: 
\[*\lambda + *v + \mbox{not linear terms}.\]
Hence, the corresponding constraints are independent.
\item Case 2.2,  $\sigma_{i_{j_{1_{e}}}} \neq  \sigma_{i_{j_{2_{f}}}}.$ The reason similar to the Case 1.2.
\end{itemize}
\end{itemize}
\end{proof}
\end{corollary}

\begin{example}
Consider a curve $C$  defined by $x^{11}=y^{6}$. Its valuation group is   $\Gamma = \langle 6,11 \rangle$ with conductor $c = 5\times 10=50.$  Let $\Delta = (30, 35, 40, 45, 50, 55)_{\Gamma}$ be a subsemimodule of $\Gamma$ with syzygy 
$Syz(\Delta) = (41, 46, 51, 56, 61, 66)_{\Gamma}.$ Note that \[
 \Gamma \setminus \Delta = \{0, 6, 11, 12, 17, 18, 22, 23, 24, 28, 29, 33,34,39,44\}.\] 
We chose $\underline{i}=\{1,2,3,4\},(\gamma_{i_{1}},\dots,\gamma_{i_{m}})_{\Gamma} = (30, 35, 40, 45)_{\Gamma} .$ Then 
\[\begin{aligned}
C^{[\Delta],\underline{i}} =\{&I_{\lambda}=(f_{30}(\lambda),f_{35}(\lambda),f_{40}(\lambda),f_{45}(\lambda),f_{50}(\lambda),f_{55}(\lambda))  \in C^{[\Delta]}\mid \\
&I_{\lambda} \mbox{ can be  generated by } f_{30}(\lambda),f_{35}(\lambda),f_{40}(\lambda),f_{45}(\lambda)\}.  
\end{aligned}\]
\[
\begin{aligned}
f_{30}(\lambda) &= t^{30} + \lambda^3_1 t^{30} + \lambda_1^4 t^{34} + \lambda_1^9 t^{39} + \lambda_1^{14} t^{44}\\
f_{35}(\lambda) &= t^{35} + \lambda_2^4 t^{39} + \lambda_2^9 t^{44}\\
f_{40}(\lambda) &= t^{40} + \lambda_3^4 t^{44}\\
f_{45}(\lambda) &= t^{45}\\
\end{aligned}
\]

Let \[
\sigma_{i_{j_{1}}} = \sigma_1 = 41 < 50, \quad
\sigma_{i_{j_2}} = \sigma_2 = 46 < 55
\]
and 
\[
\sigma_{i_{k_{1}}} = \sigma_2 = 46 < 50, \quad 
\sigma_{i_{k_2}} = \sigma_1 = 41 < 55.
\]
where $\sigma_1 = 41 = 30+11 = 35+6$ and $\sigma_2 = 46 = 35+11 = 40+6.$\\

Now, we compute $Y_{\underline{i}_{\underline{j}}}=Y_{i_{j_{1}}}\cap Y_{i_{j_{2}}}$, $Y_{\underline{i}_{\underline{k}}}=Y_{i_{k_{1}}}\cap Y_{i_{k_{2}}}$ and their  intersection $Y_{\underline{i}_{\underline{j}}}\cap Y_{\underline{i}_{\underline{k}}}.$\\

For  $Y_{\underline{i}_{\underline{j}}}$, we have 
$(\Delta_{\underline{i}})_{> \sigma_{i_{j_{1}}}, <50} = \{ 42, 45, 46, 47, 48\}.$ Fix a decomposition,  
\[
42 = 30 + 12, 45 = 45 + 0, 
46 = 40 + 6, 
47 = 30 + 17, 
48 = 30 + 18.
\]
Then we define:
\[
\begin{aligned}
(S_v^{(1)})_{\sigma_{i_{j_{1}}}}^1 &= t^{11} + v_{1, 42}^1 t^{12} + v_{1, 47}^6 t^{17} + v_{1, 48}^{7} t^{18},\\
(S_v^{(n)})_{\sigma_{i_{j_{1}}}}^2 &= -t^6,\\
(S_v^{(n)})_{\sigma_{i_{j_{1}}}}^3 &= v_{1,46}^5 t^6,\\
(S_v^{(1)})_{\sigma_{i_{j_{1}}}}^4 &= v_{1, 45}^4.
\end{aligned}
\]
Let $(\mathcal{G}_{\lambda}^{(1)})_{e}=f_{\gamma_{e}}(\lambda)$, $e=1,2,3,4.$ Then $Y_{i_{j_{1}}}$ is defined by 
\[\begin{aligned}
\sum_{e=1}^{4} (\mathcal{G}_{\lambda}^{(1)})_{e}\circ (S_{v}^{(1)})_{i_{j_{1}}}^{e}&=v_{1,42}^{1}t^{42}+\lambda_{1}^{3}t^{44}+(v_{1,45}^{4}-\lambda_{2}^4+\lambda_{1}^4+v_{1,42}^1)t^{45}+(v_{1,46}^{5}+\lambda_{1}^4 v_{1,42}^1)t^{46}\\
&+v_{1,47}^{6}t^{47}+v_{1,48}^{7}t^{48}=0
\end{aligned}\]
and $(\lambda_{1}^{9}-\lambda_{2}^{9}+\lambda_{1}^{3}v_{1,47}^{6}+\lambda_{3}^{4}v_{1,46}^{5})t^{50}\neq 0.$
It implies that $Y_{i_{j_{1}}}$ is defined by  $\lambda_{1}^{3}=0, \lambda_{1}^{9}-\lambda_{2}^{9}\neq 0.$\\

Similarly, $Y_{i_{j_{2}}}$ is defined by 

\[\begin{aligned}
\sum_{e=1}^{5} (\mathcal{G}_{\lambda}^{(2)})_{e}\circ (S_{v}^{(2)})_{i_{j_{2}}}^{e}&=v_{2,47}^{1}t^{47}+v_{2,48}^{1}t^{48}+(v^4_{2,50}+\lambda_{2}^{4}-\lambda_{3}^{4}+\lambda_{1}^{3}v_{2,47}^{1})t^{50}\\
&+v_{2,51}^{5}t^{51}+(v_{2,52}^{6}+\lambda_{1}^{4}v_{2,48}^{2})t^{52}+v_{2,53}^{7}t^{53}=0
\end{aligned}\]
and $\lambda_{2}^{9}t^{55}\neq 0.$ It implies that $Y_{i_{j_{2}}}$ is defined by  $\lambda_{2}^{9}\neq 0.$\\

Hence, $Y_{\underline{i}_{\underline{j}}}$ is defined by 
\begin{equation}  \label{Yij}
\lambda_{1}^{3}=0, \lambda_{1}^{9}-\lambda_{2}^{9}\neq 0, \lambda_{1}^{9}\neq 0. 
\end{equation}

For  $Y_{\underline{i}_{\underline{k}}}$,  $Y_{i_{k_{1}}}$ is defined by

\[\begin{aligned}
\sum_{e=1}^{4} (\mathcal{G}_{\lambda}^{(1)})_{e}\circ (S_{v}^{(1)})_{i_{k_{1}}}^{e}&=v_{2,47}^{1}t^{47}+v_{2,48}^{1}t^{48}=0, \quad (\lambda_{2}^{4}-\lambda_{3}^{4}+\lambda_{1}^{3}v_{2,47}^{1})t^{50}\neq 0 .
\end{aligned}\]

It implies that $Y_{i_{k_{1}}}$ is defined by  $\lambda_{2}^{4}-\lambda_{3}^{4}\neq 0.$\\

 $Y_{i_{k_{2}}}$ is defined by 

\[\begin{aligned}
\sum_{e=1}^{5} (\mathcal{G}_{\lambda}^{(2)})_{e}\circ (S_{v}^{(2)})_{i_{k_{2}}}^{e}&=v_{1,42}^{1}t^{42}+\lambda_{1}^{3}t^{44}+(v_{1,45}^{4}-\lambda_{2}^4+\lambda_{1}^4+v_{1,42}^1)t^{45}+(v_{1,46}^{5}+\lambda_{1}^4 v_{1,42}^1)t^{46}\\
&+v_{1,47}^{6}t^{47}+v_{1,48}^{7}t^{48}+(v_{1,50}^9+\lambda_{1}^{9}-\lambda_{2}^{9}+\lambda_{1}^{3}v_{1,47}^{6}+\lambda_{3}^{4}v_{1,46}^{5})t^{50}\\
&+(v_{1,51}^{10}+\lambda_{1}^3v_{1,48}^7+\lambda_{1}^4v_{1,47}^6+\lambda_{1}^{7}v_{1,42}^1)t^{51}+(v_{1,52}^{11}+\lambda_{1}^4v_{1,48}^7)t^{51}+v_{1,53}^{12}t^{53}+v_{1,54}^{13}t^{54}\\
&=0
\end{aligned}\]
and $\lambda_{1}^{14}t^{55}\neq 0$. It implies that $Y_{i_{k_{2}}}$ is defined by  $\lambda_{1}^{3}= 0,\lambda_{1}^{14}\neq 0 .$\\

Hence, $Y_{\underline{i}_{\underline{k}}}$ is defined by 
\begin{equation} \label{Yik}
\lambda_{2}^{4}-\lambda_{3}^{4}\neq 0,\lambda_{1}^{3}= 0,\lambda_{1}^{14}\neq 0.
\end{equation}

The intersection $Y_{\underline{i}_{\underline{j}}}\cap Y_{\underline{i}_{\underline{k}}}$ is defined by equations (\ref{Yij}) and (\ref{Yik}). The common constraint $\lambda_{1}^3=0$ arises from the definition of $I_{\lambda}$ lying in $C^{[\Delta]}$. The additional constraints do not introduce any interdependence between the conditions for $I_{\lambda}$ lying in $Y_{\underline{i}_{\underline{j}}}$ and the conditions for it lying in $Y_{\underline{i}_{\underline{k}}}$.
\end{example}

In \parencite[Section 5]{oblomkov2012hilbert}, let $C$ be the curve defined by $y^k = x^n$. The authors consider a $\mathbb{C}^*$-action on $C^{[\ell],m}$. By \parencite[Corollary 2]{bialynicki1973fixed}, only the fixed points contribute to the Euler characteristic, so 
\[
\chi(C^{[\ell],m}) = \chi\big((C^{[\ell],m})^{\mathbb{C}^*}\big).
\]
This simplifies the algebro-geometric side of Conjecture~\ref{conjecture 1} to the generating function for the sum over all monomial ideals, which is computed via a residue calculation. They then evaluate this residue to verify Conjecture~\ref{conjecture 1}.

We recover this simplification using the decomposition of $C^{[\Delta],\leq m}$ given in Theorem~\ref{decomposition of C Delta i}.

\begin{theorem}\cite{oblomkov2012hilbert}
For the curve $C$ defined by $y^k = x^n$, we have
\[
\int_{C^{[*]}} q^{2\ell} (1 - a^2)^m \, d\chi = \sum_{\substack{J \text{ monomial}}} q^{2 \dim_{\mathbb{C}} \mathcal{O}_{C}/J} (1 - a^2)^{m(J)}
\]    
\end{theorem}

\begin{proof}
Let $\Delta$ be a subsemimodule of $\Gamma$ with $n(\Delta)$ minimal generators as a $\Gamma$-semimodule. For $m \leq n(\Delta)$, we have
\[
C^{[\Delta], m} = C^{[\Delta], \leq m} \setminus C^{[\Delta], \leq m-1}.
\]

For $m < n(\Delta)$, by Theorem~\ref{decomposition of C Delta i}, the variety $C^{[\Delta], m}$ is a union of trivial algebraic torus bundles. Since such bundles have Euler characteristic zero, it follows that
\[
\chi(C^{[\Delta], \leq m}) = 0 \quad \text{and} \quad \chi(C^{[\Delta], m}) = 0.
\]

Since $C^{[\Delta], \leq n(\Delta)} = C^{[\Delta]}$, and $C^{[\Delta]}$ is an affine space by Theorem~\ref{dim of Hilbert scheme}, we have $\chi(C^{[\Delta]}) = 1$. Therefore,
\[
\chi(C^{[\Delta], n(\Delta)}) = \chi(C^{[\Delta], \leq n(\Delta)}) - \chi(C^{[\Delta], \leq n(\Delta)-1}) = 1 - 0 = 1.
\]

Thus, we conclude:

\[\begin{aligned}
\int_{C^{[*]}} q^{2\ell} (1 - a^2)^m \, d\chi &= \sum_{\ell\geq 0, \Delta\in \mathscr{D}_{\ell }} \sum_{m\geq 1}q^{2\ell}(1 - a^2)^{m}\chi(C^{[\Delta],m})\\
&= \sum_{\ell\geq 0, \Delta\in \mathscr{D}_{\ell }} q^{2\ell}(1 - a^2)^{n}\chi(C^{[\Delta],n(\Delta)})\\
&= \sum_{\ell\geq 0, \Delta\in \mathscr{D}_{\ell }} q^{2\ell}(1 - a^2)^{n(\Delta)}.\\
\end{aligned}
\]            
The result follows from the fact that there exists a bijection from the set of monomial ideals of ${\mathbb{C}}[[t^k,t^n]]$ to the set of subsemimodules of $\Gamma$. 
\end{proof}

\begin{example}
We consider $C=\{y^3=x^7\}$,  its semigroup is $\Gamma= \langle 3, 7 \rangle$ with conductor $c=12$. 

\[
\begin{aligned}
Zm_{(C,O)}^{\mathrm{Hilb}}(a^2, q^2) 
&= \sum_{\ell \geq 0} \sum_{m \geq 1} q^{2\ell} (1 - a^2)^{m-1} [C^{[\ell],m}] \\
&= \sum_{\ell \geq 0} \sum_{\Delta \in \mathscr{D}_\ell} \sum_{m \geq 1} q^{2\ell} (1 - a^2)^{m-1} [C^{[\Delta],m}],
\end{aligned}
\]
\end{example}

The computation of $Zm_{(C,O)}^{\mathrm{Hilb}}$ proceeds by evaluating the sum layer-wise with respect to the index $\ell$:\\

For $\ell=0$:
\[
\sum_{m \geq 1} q^{2\ell} (1 - a^2)^{m-1} [C^{[\ell],m}] = 1.
\]

For $\ell=1$:
\[
\sum_{m \geq 1} q^{2\ell} (1 - a^2)^{m-1} [C^{[\ell],m}] = q^2(1 - a^2).
\]

For $\ell=2$:
\[
\sum_{m \geq 1} q^{2\ell} (1 - a^2)^{m-1} [C^{[\ell],m}] = (1 + \mathbb{L}) q^4(1 - a^2).
\]

For $\ell=3$:
\[
\sum_{m \geq 1} q^{2\ell} (1 - a^2)^{m-1} [C^{[\ell],m}] = q^6(1 - a^2)^2 + \mathbb{L} q^6(1 - a^2) + \mathbb{L}^2 q^6.
\]

For $\ell=4$:
\[
\sum_{m \geq 1} q^{2\ell} (1 - a^2)^{m-1} [C^{[\ell],m}] = q^8(1 - a^2)^2 + (\mathbb{L} + 2\mathbb{L}^2) q^8(1 - a^2).
\]

For $\ell=5$:
\[
\sum_{m \geq 1} q^{2\ell} (1 - a^2)^{m-1} [C^{[\ell],m}] = (1 + \mathbb{L}) q^{10}(1 - a^2)^2 + (2\mathbb{L}^2 + \mathbb{L}^3) q^{10}(1 - a^2).
\]

For $\ell=6$:
\[
\sum_{m \geq 1} q^{2\ell} (1 - a^2)^{m-1} [C^{[\ell],m}] = (1 + \mathbb{L}) q^{12}(1 - a^2)^2 + (\mathbb{L}^2 + \mathbb{L}^3) q^{12}(1 - a^2) + \mathbb{L}^4 q^{12}.
\]

For $\ell=7$:
\[
\sum_{m \geq 1} q^{2\ell} (1 - a^2)^{m-1} [C^{[\ell],m}] = (1 + \mathbb{L} + 2\mathbb{L}^2) q^{14}(1 - a^2)^2 + (2\mathbb{L}^3 + \mathbb{L}^4) q^{14}(1 - a^2) + 2\mathbb{L}^4 q^{14}.
\]

For $\ell=8$:
\[
\sum_{m \geq 1} q^{2\ell} (1 - a^2)^{m-1} [C^{[\ell],m}] = (1 + \mathbb{L} + 2\mathbb{L}^2) q^{16}(1 - a^2)^2 + (2\mathbb{L}^3 + 3\mathbb{L}^4) q^{16}(1 - a^2).
\]

For $\ell=9$:
\[
\sum_{m \geq 1} q^{2\ell} (1 - a^2)^{m-1} [C^{[\ell],m}] = (1 + \mathbb{L} + 2\mathbb{L}^2 + \mathbb{L}^3) q^{18}(1 - a^2)^2 + (3\mathbb{L}^3 + 2\mathbb{L}^4) q^{18}(1 - a^2) + \mathbb{L}^5 q^{18}.
\]

For $\ell=10$:
\[
\sum_{m \geq 1} q^{2\ell} (1 - a^2)^{m-1} [C^{[\ell],m}] = (1 + \mathbb{L} + 2\mathbb{L}^2 + \mathbb{L}^3) q^{20}(1 - a^2)^2 + (3\mathbb{L}^3 + 2\mathbb{L}^4 + \mathbb{L}^5) q^{20}(1 - a^2) + \mathbb{L}^5 q^{20}.
\]

For $\ell=11$:
\[
\sum_{m \geq 1} q^{2\ell} (1 - a^2)^{m-1} [C^{[\ell],m}] = (1 + \mathbb{L} + 2\mathbb{L}^2 + \mathbb{L}^3) q^{22}(1 - a^2)^2 + (3\mathbb{L}^3 + 2\mathbb{L}^4 + 2\mathbb{L}^5) q^{22}(1 - a^2).
\]

For $\ell=12$:
\[
\sum_{m \geq 1} q^{2\ell} (1 - a^2)^{m-1} [C^{[\ell],m}] = (1 + \mathbb{L} + 2\mathbb{L}^2 + \mathbb{L}^3) q^{24}(1 - a^2)^2 + (3\mathbb{L}^3 + 2\mathbb{L}^4 + 2\mathbb{L}^5) q^{24}(1 - a^2) + \mathbb{L}^6 q^{24}.
\]

In summary, we have 
\begin{align*}
Zm_{(C,O)}^{\mathrm{Hilb}}(a^2, q^2) 
 &= \sum_{\ell=0}^{12}  \sum_{m \geq 1} q^{2\ell} (1 - a^2)^{m-1} [C^{[\ell],m}]   +  \sum_{\ell> 12}  \sum_{m \geq 1} q^{2\ell} (1 - a^2)^{m-1} [C^{[\ell],m}]  \\
& = 1 \\
& + q^2 (1-a^2) \\
& + (1+\mathbb{L}) q^4 (1-a^2) \\
& + \left[ \mathbb{L}^2 + (1+\mathbb{L})(1-a^2) + (1-a^2)^2 \right] q^6 \\
& + \left[ (\mathbb{L} + 2\mathbb{L}^2)(1-a^2) + (1-a^2)^2 \right] q^8 \\
& + \left[ (2\mathbb{L}^2 + \mathbb{L}^3)(1-a^2) + (1+\mathbb{L})(1-a^2)^2 \right] q^{10} \\
& + \left[ \mathbb{L}^4 + (\mathbb{L}^2 + \mathbb{L}^3)(1-a^2) + (1+\mathbb{L})(1-a^2)^2 \right] q^{12} \\
& + \left[ 2\mathbb{L}^4 + (2\mathbb{L}^3 + \mathbb{L}^4)(1-a^2) + (1+\mathbb{L} + 2\mathbb{L}^2)(1-a^2)^2 \right] q^{14} \\
& + \left[ (2\mathbb{L}^3 + 3\mathbb{L}^4)(1-a^2) + (1+\mathbb{L} + 2\mathbb{L}^2)(1-a^2)^2 \right] q^{16} \\
& + \left[ \mathbb{L}^5 + (3\mathbb{L}^3 + 2\mathbb{L}^4)(1-a^2) + (1+\mathbb{L} + 2\mathbb{L}^2 + \mathbb{L}^3)(1-a^2)^2 \right] q^{18} \\
& + \left[ \mathbb{L}^5 + (3\mathbb{L}^3 + 2\mathbb{L}^4 + \mathbb{L}^5)(1-a^2) + (1+\mathbb{L} + 2\mathbb{L}^2 + \mathbb{L}^3)(1-a^2)^2 \right] q^{20} \\
& + \left[ (3\mathbb{L}^3 + 2\mathbb{L}^4 + 2\mathbb{L}^5)(1-a^2) + (1+\mathbb{L} + 2\mathbb{L}^2 + \mathbb{L}^3)(1-a^2)^2 \right] q^{22} \\
& + \left[ \mathbb{L}^6 + (3\mathbb{L}^3 + 2\mathbb{L}^4 + 2\mathbb{L}^5)(1-a^2) + (1+\mathbb{L} + 2\mathbb{L}^2 + \mathbb{L}^3)(1-a^2)^2 \right]( q^{24}+q^{26}+\dots )
\end{align*}

\section{On Monomial semigroups}\label{Appendix}

We recall the notion of monomial semigroups as introduced in \cite{pfister1992reduced}. 
\begin{definition}[\cite{pfister1992reduced}]
A monomial curve singularity over ${\mathbb{C}}$ is an irreducible curve singularity with local ring isomorphic to $A={\mathbb{C}}[[t^{a_{1}},\dots,t^{a_{m}}]]$ with $\operatorname{\gcd}(a_{1},\dots,a_{m})=1$. 
\end{definition}

\begin{definition}[\cite{pfister1992reduced}]
A semigroup $\Gamma\subset\mathbb{N}$ is called monomial if $0\in \Gamma$, $\#(\mathbb{N}\setminus\Gamma)<\infty$ and each reduced and irreducible curve singularity with semigroup $\Gamma$ is a monomial curve singularity. 
\end{definition}

\begin{proposition}[\cite{pfister1992reduced}]
For a semigroup $\Gamma \subset \mathbb{N}$, the following are equivalent:

\begin{enumerate}
    \item[(1)] $\Gamma$ is a monomial semigroup.
    
    \item[(2)] $\Gamma$ is of one of the following forms:
    \begin{enumerate}
        \item[(i)] $\Gamma_{m,s,b} := \{ im \mid i = 0,1,\dots,s \} \cup [sm + b, \infty)$, where $1 \leq b < m$ and $s \geq 1$,
        
        \item[(ii)] $\Gamma_{m,r} := \{0\} \cup [m, m+r-1] \cup [m+r+1, \infty)$, where $2 \leq r \leq m-1$,
        
        \item[(iii)] $\Gamma_m := \{0, m\} \cup [m+2, 2m] \cup [2m+2, \infty)$, where $m \geq 3$.
    \end{enumerate}
    
    \item[(3)] $0 \in \Gamma$, $\#(\mathbb{N} \setminus \Gamma) < \infty$, and for every $x \in \mathbb{N} \setminus \Gamma$, define
    \[
    c(x) := \min\left\{ n \in \mathbb{N} \mid [n, \infty) \subset \Gamma \cup (x + \Gamma) \right\}.
    \]
    Then the condition
    \[
    \Gamma \cap (x + \Gamma) \subset [c(x), \infty)
    \]
    holds.
\end{enumerate}
\end{proposition}

We have the following lemma, which is implicit in the proof of \cite[Theorem 11]{pfister1992reduced}:

\begin{lemma}\label{c(x)+y}
For $y\in \mathbb{Z}$ and $x\in \mathbb{N}\setminus \Gamma$, then $c(x)+y= \min \{m\in \mathbb{N}\mid [m,\infty) \subset  (\Gamma+y)\cup(x+y+\Gamma)\}$. 

\end{lemma}
\begin{proof}
By definition, we have
\[
[c(x) + y, \infty) \subset (\Gamma \cup (x + \Gamma)) + y = (\Gamma + y) \cup (x + y + \Gamma).
\]

Suppose there exists $c_0 < c(x) + y$ such that
\[
[c_0, \infty) \subset (\Gamma + y) \cup (x + y + \Gamma).
\]
Then, subtracting $y$, we obtain
\[
[c_0 - y, \infty) \subset \Gamma \cup (x + \Gamma).
\]
Since $c_0 - y < c(x)$, this contradicts the minimality of $c(x)$. Therefore, no such $c_0$ can exist.

\end{proof}

\begin{corollary} Let $\Delta$ be a $\Gamma$-subsemimodule. Suppose $\gamma_{1}$ and $\gamma_{2}$ are two generators of $\Delta$ with $\gamma_{1}>\gamma_{2}$. Then, for any $\sigma\in Syz((\gamma_{1},\gamma_{2})_{\Gamma})$,  we have $\sigma\geq c(\Delta)$. 
\begin{proof}
Since $x=\gamma_{1}-\gamma_{2}\notin \Gamma$,  we have 
\[Syz(\left \langle\gamma_{1},\gamma_{2}\right\rangle)=(\Gamma+\gamma_{1})\cap (\Gamma+\gamma_{2})= (\Gamma\cap (\Gamma+x))+\gamma_{2} \subset [c(x),\infty)+\gamma_{2}=[c(x)+\gamma_{2},\infty).\]

Applying  lemma \ref{c(x)+y}, we get:  $c(x)+\gamma_{2} = c((\gamma_{1},\gamma_{2})_{\Gamma})$. Thus, for any  $\sigma\in Syz((\gamma_{1},\gamma_{2})_{\Gamma})$,  we conclude that: $\sigma \geq  c((\gamma_{1},\gamma_{2})_{\Gamma})\geq c(\Delta)$. 
\end{proof}
\label{syz large that c in monomial semigroup}
\end{corollary}

\textbf{Acknowledgments.} The authors are grateful to Eugene Gorsky for discussions related to this work. We extend our deep gratitude to Laurent Evain and Evelia Rosa Garc\'ia Barroso for their meticulous review of the manuscript; their insightful comments were invaluable in refining the final version. The second author was partially supported by he ANR-SINTROP. The third author also acknowledges financial support from the Guangzhou Elites Sponsorship Council (China) through the Overseas Study Program of the Guangzhou Elite Project.

\medskip
\printbibliography
\end{document}